\documentclass[11pt]{amsart}

\usepackage{amsmath, amsthm, amssymb, amsfonts, enumerate}
\usepackage[colorlinks=true,linkcolor=blue,urlcolor=blue]{hyperref}
\usepackage{dsfont}
\usepackage{color}
\usepackage{geometry}
\usepackage{todonotes}
\usepackage{amssymb}
\definecolor{red}{rgb}{1.0,0.0,0.0}

\definecolor{blu}{rgb}{0.0,0.0,1.0}

\definecolor{gre}{rgb}{0.03,0.50,0.03}

\definecolor{amethyst}{rgb}{0.6, 0.4, 0.8}

\definecolor{blue-violet}{rgb}{0.54, 0.17, 0.89}

\definecolor{darkviolet}{rgb}{0.58, 0.0, 0.83}
\def\dvio#1{{\textcolor{darkviolet}{#1}}}

\newtheorem{theorem}{Theorem}[section]
\newtheorem{remark}[theorem]{Remark}
\newtheorem{assumption}[theorem]{Assumption}
\newtheorem{lemma}[theorem]{Lemma}
\newtheorem{proposition}[theorem]{Proposition}
\newtheorem{corollary}[theorem]{Corollary}

\def \calo{{\mathcal O}}
\def \cA{{\mathcal A}}
\def \cF{{\mathcal F}}
\def \cL{{\mathcal L}}
\def \cC{{\mathcal C}}
\def \cS{{\mathcal S}}
\def \cI{{\mathcal I}}
\def \E{\mathsf{E}}
\def \P{\mathsf{P}}
\def \R{\mathbb{R}}
\def \F{\mathbb{F}}
\def \N{\mathbb{N}}

\def \eps{\varepsilon}

\title[Optimal dividend under stochastic discounting]{Optimal Dividend Payout \\ under Stochastic Discounting}
\author[Bandini, De Angelis, Ferrari, Gozzi]{Elena Bandini, Tiziano De Angelis, Giorgio Ferrari, Fausto Gozzi}
\address{E.~Bandini: Universit\`a degli Studi di Milano-Bicocca, Via R. Cozzi 55, 20125 Milano, Italy.}
\email{\href{mailto:elena.bandini@unimib.it}{elena.bandini@unimib.it}}

\address{T.~De Angelis: School of Management and Economics (Dept.\ ESOMAS), University of Turin, C.so Unione Sovietica 218bis, 10134, Turin (Italy); Collegio Carlo Alberto, Piazza Arbarello 8, 10122, Turin (Italy); }
\email{\href{mailto:tiziano.deangelis@unito.it}{tiziano.deangelis@unito.it}}

\address{G.~Ferrari: Center for Mathematical Economics, Bielefeld University, Universit\"atsstrasse 25, 33615 Bielefeld, Germany.}
\email{\href{mailto:giorgio.ferrari@uni-bielefeld.de}{giorgio.ferrari@uni-bielefeld.de}}

\address{F.~Gozzi: LUISS Guido Carli, Viale Romania, 32, 00197 Roma, Italy.}
\email{\href{mailto:fgozzi@luiss.it}{fgozzi@luiss.it}}
\date{\today}

\numberwithin{equation}{section}

\begin{document}

\begin{abstract}
Adopting a probabilistic approach we determine the optimal dividend payout policy of a firm whose surplus process follows a controlled arithmetic Brownian motion and whose cash-flows are discounted at a stochastic dynamic rate. Dividends can be paid to shareholders at unrestricted rates so that the problem is cast as one of singular stochastic control. The stochastic interest rate is modelled by a Cox-Ingersoll-Ross (CIR) process and the firm's objective is to maximize the total expected flow of discounted dividends until a possible insolvency time.

We find an optimal dividend payout policy which is such that the surplus process is kept below an endogenously determined stochastic threshold expressed as a decreasing {\color{black}{continuous}} function $r\mapsto b(r)$ of the current interest rate value. We also prove that the value function of the singular control problem solves a variational inequality associated to a second-order, non-degenerate elliptic operator, with a gradient constraint.
\end{abstract}

\maketitle

%\tableofcontents

\section{Introduction}

\subsection{The Problem}

{\color{black}{In this paper we solve an optimal dividend problem with stochastic discounting. In our model, the company pays dividends to shareholders at unrestricted rates and any dividend payment instantaneously reduces the company's surplus. The aim is to maximize the total expected discounted return of dividend payments, up to a possible insolvency time. We assume that dividends are discounted exponentially at a stochastic rate given by a deterministic nondecreasing and nonnegative function $\rho$ of the short interest rate $R$. As we also discuss in Remark \ref{rem:discount}, when $\rho(R)=R$ such a discounting force might be justified, e.g., by thinking that the company discounts at the cost of equity which, in a risk-neutral world, coincides with the risk-free interest rate according to the capital asset pricing model. Alternatively, looking at the company as a dividend paying security in a complete financial market, the stochastic discount factor can be then interpreted as a classical deflator process. Accordingly, the company's value is given by the total expected discounted flow of dividends.
In classical optimal dividend models the discount rate is often deterministic (and constant), so that shareholders are only exposed to risks arising from the random profitability of the firm (see also Section \ref{se:relLit} below).}} On the contrary, in our setting shareholders are also exposed to uncertainty from the wider macro-economic activity via random fluctuations in the interest rate.
\bigskip

From a mathematical point of view, we model the previous problem as a \emph{two-dimensional singular stochastic control problem}. The two coordinates of the state process are the surplus process and the short interest rate. The surplus process evolves as a Brownian motion $(Z^D_t)_{t\geq0}$ with drift $\mu$ and volatility $\sigma$, which is linearly controlled {via} a nondecreasing stochastic process $(D_t)_{t\geq0}$ representing the cumulative amount of distributed dividends. The uncontrolled short interest rate $(R_t)_{t\geq0}$ {enters into} the exponential discount factor appearing in the expected return of {dividend} payments. The process $(R_t)_{t\geq0}$ is assumed to be independent of the surplus' process, and to follow a mean-reverting dynamics specified by the Cox-Ingersoll-Ross (CIR) model. We require that the coefficients of the CIR process fulfill the so-called Feller condition (see \eqref{eq:hpCIR} below), so that the short interest rate is strictly positive at any time with probability one. The discount rate at time $t$ is of the form $\rho(R_t)$ (hence, total discounting up to time $t$ is $e^{-\int_0^t \rho(R_s) ds}$), for some nonnegative and nondecreasing function $\rho$ satisfying suitable growth conditions (see Assumption \ref{ass:rho} below). Notice that our requirements on $\rho$ are such that the cases of constant and linear discounting forces (i.e., like $\rho(r)=\rho_0>0$ or $\rho(r)=r$ for all $r \in \mathbb{R}_+$) are included in our setting. The aim is to maximize the total expected discounted value of dividends, up to the random time $\tau^D:=\inf\{t\geq 0: Z^D_t \leq \alpha\}$, for a given and fixed solvency level $\alpha\ge 0$. If $\alpha = 0$ we find the classical bankruptcy condition for this kind of models.

\subsection{Methodology and Results}
\label{sec:intromethod}

The key challenge in our work arises from the two-dimensional (non-degenerate) diffusive nature of the set-up. Indeed, dynamic programming ideas link the stochastic control problem to a variational problem involving an elliptic partial differential equation (PDE) with gradient constraint that is not amenable to an explicit solution. This stands in contrast with some of the more classical versions of the same problem where the state process is purely one-dimensional (see \cite{JS} for an early formulation and, for example, \cite{LOKKA} and \cite{SethiTaksar} among more recent contributions). Indeed, the dynamic programming equation arising in one-dimensional problems involves an ordinary differential equation (ODE) so that a so-called guess-and-verify approach can be implemented. The latter consists of an educated guess on the structure of the problem's solution, leading to an ODE for the value function with suitable boundary conditions (usually involving {\em smooth-fit}). The ODE can be solved explicitly and a verification theorem allows to prove that such solution is indeed the value function of the problem. That approach fails in our set-up since explicit solutions are not available.

\smallskip

In order to solve our two-dimensional optimal dividend problem, here we follow ideas developed in \cite{DeAE17} and later extended in \cite{DeA}. We link the optimal dividend problem to an auxiliary problem of optimal stopping whose underlying process is a two-dimensional reflecting diffusion $(R,K)$ and whose payoff increases upon each new reflection of $(R,K)$, but it is discounted with the same stochastic dynamic rate as in the original dividend problem.
In both \cite{DeAE17} and \cite{DeA} the interest rate is constant although the state-space is two-dimensional. In \cite{DeAE17} the problem is set on a finite-time horizon but the diffusive dynamics only affects one state variable. In \cite{DeA} the time-horizon is infinite but there is partial information that leads to the same Brownian motion driving a two-dimensional SDE (hence degenerate). On the contrary, here we have a fully two-dimensional diffusive set-up so that the construction of the auxiliary optimal stopping problem is different to those in \cite{DeAE17} and \cite{DeA} (e.g., here it preserves the stochastic discounting) and the subsequent analysis of the optimal dividend policy must follow a different line of argument. In particular, the use of a stochastic discount rate with CIR dynamics leads to numerous technical complications. These arise, e.g., in the proof of a preliminary verification theorem for the dividend problem (Theorem \ref{thm:verif}), as well as in showing boundedness and regularity of the value in the optimal stopping problem (Propositions \ref{prop:U-bound} and \ref{prop:C1}). Also it is worth noticing that the dynamic programming equation in \cite{DeAE17} and \cite{DeA} involves a one-dimensional parabolic PDE, while in our problem we have a two-dimensional elliptic PDE.

\smallskip

In the auxiliary optimal stopping problem that we consider (see the beginning of Section \ref{sec:OS}), the state variable consists of the original short interest rate $R$ appearing in the discount factor, and of a Brownian motion $K$ with drift {$\mu$} and volatility $\sigma$, which is reflected at the solvency level $\alpha$. By making use of almost exclusively probabilistic arguments, we show that the optimal stopping time is expressed in terms of the hitting time of the process $t\mapsto K_t$ to a (stochastic) moving boundary $t\mapsto b(R_t)$, where $b$ is a nonincreasing and {\color{black}{continuous function on $[0,\infty)$ whose properties are collected in Lemma \ref{lem:bnew}, Theorem \ref{thm:bC} and Proposition \ref{prop:bC}.}} Moreover, using that the underlying process $(R,K)$ is {a strong Feller process} and that the boundary points are regular (in the probabilistic sense) for the stopping region, we can show (Proposition \ref{prop:C1}) that the value function $U$ of the stopping problem is everywhere continuously differentiable (see also \cite{DeAPe17} for general results in this direction).

The smoothness of the function $U$ allows to construct the value function $V$ of the dividend problem by a simple integration
(formula \eqref{eq:v} in Section \ref{sec:solution}) and provides nice regularity properties for $V$. Indeed, as a function of the state variables $(r,z)$ associated to the process $(R,Z^D)$, the mapping $(r,z)\mapsto V(r,z)$ is globally $C^1$, with second order derivatives $\partial_{zz}V$ and $\partial_{rz}V$ that are continuous everywhere. Furthermore, the second order derivative $\partial_{rr}V$ is locally bounded in the whole space and continuous away from the boundary $z=b(r)$ with well-defined limits up to the boundary (Propositions \ref{prop:vrr} and \ref{prop:C2}). A direct approach to the variational problem with gradient constraint for the function $V$ is involved, especially because of an additional boundary condition along the solvency level, i.e.\ $V(r,\alpha)=0$ (see, e.g., \cite{G13,G14,G15}). In this respect, our probabilistic approach overcomes the difficulties arising in the PDE arguments.

The main result of the paper is Theorem \ref{thm:verifico} which, thanks to the verification Theorem \ref{thm:verif} and to the regularity results mentioned above, {links} the value functions $U$ and $V$ and provides an optimal dividend strategy as a Skorokhod reflection of the process $t\mapsto Z^D_t$ below the stochastic boundary $t\mapsto b(R_t)$. {\color{black}{The structure of the optimal dividend policy is discussed in Section \ref{rem:secfinal}, and numerical illustrations of the free boundary and of the {value function $U$ for the} optimal stopping problem are presented in Section \ref{sec:num}. }}

\subsection{Related Literature}
\label{se:relLit}

The first version of an optimal dividend problem was formulated by Bruno de Finetti in 1957 in \cite{deFinetti}. De Finetti proposed to measure the value of an insurance company in terms of the discounted value of its future dividend payments. Since then the optimal dividend problem has been studied extensively and it has become a cornerstone of the modern Mathematical Finance/Actuarial Mathematics literature. Early contributions addressing the dividend problem via control-theoretic techniques include, e.g., \cite{JS}, where the authors consider several problem formulations, including controls with bounded-velocity and singular controls (see also \cite{RS}, which appeared in the same years). A broad class of infinite-time horizon singular control problems for one-dimensional diffusions, inspired by the optimal dividend problem, were analysed in \cite{Shreveetal} who obtained general formulae. Numerous extensions and refinements of those early models have appeared in the literature; here we only mention a few of them and our review is certainly not exhaustive. For example, in \cite{Cad} the cash reserve has a mean-reverting dynamics and lump sum dividend payments are made at optimally chosen discrete dates (i.e., impulsive controls are considered); \cite{Reppeneatal} studies a model with stochastic drift in the dynamics of the company's surplus process; in \cite{Belhaj} the surplus process evolves as a jump-diffusion so that the company faces two types of liquidity risk: a Brownian risk and a Poisson risk. On an infinite-time horizon, \cite{LOKKA} allows capital injections in order to avoid company's bankruptcy, whereas \cite {Ferrari19} considers a general diffusive model with ``forced'' capital injections (see also \cite{FerrariSchuhmann} for the finite-time horizon version). In the series of papers \cite{G13,G14,G15} the author solves the optimal dividend problem with finite-time horizon by means of purely PDE methods, whereas \cite{DeAE17} addresses the problem probabilistically. Additional references can be also found in the review \cite{Avanzi} and in the book \cite{Schmidli}.

More closely related to our work are the papers considering stochastic discounting, many of which have appeared in recent years. In a discrete-time setting, the analysis is typically considered in the context of risk models for insurance companies (see, e.g., \cite{XieZou10} and the more recent \cite{TanEtAl15}). In continuous-time we find, e.g., \cite{Soneretal} and \cite{Pistorius} where the wealth process is a drifted Brownian motion and the interest rate is modulated by a continuous-time Markov chain (more recently \cite{Jiang} extends \cite{Pistorius} to the case of a jump-diffusive surplus process). Fixed-point methods are adopted in \cite{Pistorius} and \cite{Jiang}, whereas dynamic programming ideas appear in \cite{Soneretal}.

The papers \cite{Eisenberg1} and \cite{EisenbergMishura} consider discounting factors of the form $e^{-U_t}$. In \cite{Eisenberg1} the process $(U_t)_{t\geq0}$ is either a drifted Brownian motion or an integrated Ornstein-Uhlenbeck process, while it is a CIR process in \cite{EisenbergMishura}. It is worth noticing that the CIR process in \cite{EisenbergMishura} does not mean-revert to a finite value but explodes as $t$ diverges to infinity, in order to guarantee a finite value of the problem. With such specifications of the discount factor, the nature of the optimal dividend problems considered in \cite{Eisenberg1} and \cite{EisenbergMishura} is very different from ours. In our paper indeed it is the {\em discount rate} - and not the cumulative discounting force - that takes a mean-reverting CIR dynamics. At the technical level, when $(U_t)_{t\geq0}$ in \cite{Eisenberg1} is a Brownian motion with drift, a change of measure allows a reduction to a one-dimensional diffusive set-up. When $(U_t)_{t\geq0}$ is an integrated Ornstein-Uhlenbeck process a viscosity characterization of the value function is provided but without an optimal dividend policy. In \cite{EisenbergMishura}, explicit solutions are obtained when the surplus process is deterministic; the case of a stochastic surplus is instead investigated only in a regime of small volatility. Extensions of \cite{Eisenberg1} to the case in which $(U_t)_{t\geq0}$ is a L\'evy process can be found in \cite{Cheng17}, \cite{EisenbergKrunher}, and \cite{Jiang18}.

Compared to the existing literature we provide a detailed analysis of the value function and of the optimal dividend policy in a two-dimensional diffusive setting, under very mild assumptions on the discount rate (cf.\ Assumption \ref{ass:rho} below), and under the Feller condition \eqref{eq:hpCIR} that guarantees strictly positive interest rates.

%%%%%%%%%%%%%%%%%%%%%%%%%%%%%%%%%%%%%%%%%%%%%%%%%%%%%%%%%%%%%%%%%%%%%%%%%%%%%

\subsection{Plan of the paper}

The rest of the paper is organized as follows. In Section \ref{sec:setting} we set up the problem and prove a preliminary verification theorem. The auxiliary optimal stopping problem is studied in Section \ref{sec:OS}, while in Section \ref{sec:solution} we construct the value function of the optimal dividend problem together with its optimal dividend strategy. {\color{black}{Finally, Section \ref{rem:secfinal} contains a financial discussion on the optimal dividend policy {which is accompanied by} numerical illustrations presented in Section \ref{sec:num}. Section \ref{sec:corrBM} {discusses possible extensions to a model including correlation between the interest rate and the company's surplus processes}.}}

%%%%%%%%%%%%%%%%%%%%%%%%%%%%%%%%%%%%%%%%%%%%%%%%%%%

\section{Problem Setting and Verification Theorem}
\label{sec:setting}

\subsection{Problem Formulation and Assumptions}
We consider a probability space $(\Omega,\cF,\P)$ that carries two \emph{independent} Brownian motions $(B_t)_{t\ge0}$ and $(W_t)_{t\ge0}$. We denote by $\F:= (\cF_t)_{t \geq 0}$ the filtration generated by $(B,W)$ and augmented with $\P$-null sets.
We fix $\alpha\ge 0$, representing a minimum capital requirement, and we assume that the cash reserve (or surplus) of a company follows the controlled dynamics
\begin{align}
\label{eq:Z}
Z^D_t=z+\mu\,t+\sigma B_t-D_t,\qquad t\ge 0,
\end{align}
where $\mu \in \R$, $\sigma >0$, $z\ge \alpha$, and $(D_t)_{t\ge0}$ is right-continuous and nondecreasing. Indeed, $D_t$ denotes the total amount of dividends paid to the shareholders up to time $t$. The set of admissible cumulative dividend payments is given by
\begin{align}
\label{cA}
\cA:=\{D\,:&\,\text{$D$ is $\mathbb{F}$-adapted, nondecreasing, right-continuous and such that,}\notag\\
&\,\text{setting $D_{0-}=0$, we have $D_t-D_{t-}\le Z^D_{t-}-\alpha$, $\forall t\ge 0$, $\P$-a.s.}\}.
\end{align}
In the rest of the paper we denote by $Z^0$ the {dynamics of $Z^D$} with $D \equiv 0$.

The interest rate follows a CIR dynamics and, in particular, we have, for all $t\ge0$,
\begin{align}\label{eq:R}
d R_t = k(\theta- R_t) \,dt +\gamma\,\sqrt{R_t} \,d W_t,\quad	R_0=r\ge 0,
\end{align}
where $k$, $\theta$ and $\gamma$ are fixed constants. We assume the so-called Feller condition
\begin{equation}
\label{eq:hpCIR}
2k\theta\ge \gamma^2
\end{equation}
so that $R_t>0$, $\P$-a.s.~for all $t>0$ (see, e.g., \cite[p.\ 357 and Section 6.1.3]{JYC}).
In what follows we find sometimes convenient to use the notation $R_t^r$ for the interest rate process that starts at time zero from $R_0=r$. Similarly, we denote by $Z^{z,D}_t$ the surplus process started at time $0-$ (i.e., before any dividend payment) from the level $z\ge\alpha$, and by $Z^{z,0}_t$ the process $z+\mu t+\sigma B_t$. Accordingly, we will denote by $\P_{r,z}$ the probability measure on $(\Omega, \mathcal{F})$ such that $\P_{r,z}(\,\cdot\,)=\P(\,\cdot\,|R_0\!=\!r,\,Z^D_{0-}\!=\!z)$, and we define $\E_{r,z}$ the corresponding expected value. Also, $\E_{r}$ will denote the expectation under $\P_{r}(\,\cdot\,)=\P(\,\cdot\,|R_0\!=\!r)$ and $\E_{z}$ the expectation under $\P_{z}(\,\cdot\,)=\P(\,\cdot\,|Z^D_{0-}\!=\!z)$.

We assume that the firm's manager discounts dividends at a rate $\rho$ that depends on the current level of the interest rate.
The manager aims at maximizing the total expected discounted flow of dividends up to a possible insolvency time of the firm. Then the value function of the problem reads
\begin{align}
\label{eq:Vbis}
V(r,z):=\sup_{D\in\cA}\E_{r,z}\left[\int_{0-}^{\tau_\alpha^D}e^{-\int_0^t\rho(R_t)dt}dD_t\right],
\end{align}
where, for any $D \in \cA$, the random time horizon
\begin{equation}
\label{tausolv}
\tau^D_\alpha:=\inf\{t\ge0\,:\,Z^D_t\le \alpha\}
\end{equation}
enforces the solvency requirement $Z^D_t\ge \alpha$ for all $t\ge 0$. The notation $0-$ in the integral means that we include a possible jump $D_0-D_{0-}\le z-\alpha$ at time zero. If $\alpha=0$ we recover the classical bankruptcy condition for this kind of models (see, e.g., \cite[Chapter 2, Section 2.5]{Schmidli}).

The following assumptions on the discount rate will be standing.
\begin{assumption}\label{ass:rho}
The discount rate $\rho:\R_+\to \R_+$ is a continuous function. Moreover
\begin{itemize}
\item[(i)] it is nondecreasing;
%and either convex or concave on $\R_+$,
\item[(ii)] there exist two non-negative constants $c_1$ and $c_2$ such that $c_1+c_2>0$ and $\rho(r)\ge c_1+c_2\, r$ for $r\ge0$;
\item[(iii)]
there exists $c_3>0$ and $q \in \N$ such that, for $r_1>r_2\ge 0$,
\begin{equation}\label{eq:rholip}
\rho(r_1)-\rho(r_2)\le c_3(1+r_1^q)(\sqrt{r_1}-\sqrt{r_2}).
\end{equation}
\end{itemize}
\end{assumption}

\smallskip
\begin{remark}
\label{rem:rho}
We observe that (i) and (ii) of Assumption \ref{ass:rho} above will be used to prove all the results below.
\begin{itemize}
  \item
Condition (i) enables to obtain monotonicity properties of the value function.
  \item
Condition $(ii)$ is a mild requirement which allows us to deal with the (possibly) infinite horizon in Problem \eqref{eq:Vbis}.
\end{itemize}
Assumption \ref{ass:rho}-(iii) above is only needed in order to prove the $C^1$ property of Proposition \ref{prop:C1}; hence all the results obtained  before Proposition \ref{prop:C1} do actually hold without Assumption \ref{ass:rho}-(iii). {\color{black}{Furthermore, notice that Assumption \ref{ass:rho}-(iii) is satisfied if $\rho$ is such that $ 0\leq \rho(r_1)-\rho(r_2) \le \bar{c}_3(1+r_1^{\bar{q}})(r_1-r_2)$, for some $\bar{c}_3 >0$, $\bar{q} \in \N$ and for any $r_1 \geq r_2 \geq 0$.}}

Observe also that condition \eqref{eq:rholip} is verified, e.g., when
$\rho \in C^1(\mathbb{R}^+)$ and there exist $C>0$ and $q\in\mathbb{N}$ such that $\rho'(r) \leq C\big(1 + r^q\big)$ for any $r\ge0$.
\\
Finally, notice that $(i)+(ii)+(iii)$ is consistent with reasonable models for the discount rate, including $\rho(r)=r$ and $\rho(r)=\text{const.}$, which are canonical.
\end{remark}

{\color{black}{
\begin{remark}
\label{rem:discount}
As already discussed in the {\em Introduction}, the canonical case $\rho(r)=r$ {has various economic/financial interpretations}. For example, we might think that the company evaluates the risk-adjusted present value of each future dividend by discounting it at the cost of equity. In a risk-neutral world, the latter {cost} coincides with the risk-free interest rate, according to the capital asset pricing model \cite{Sharpe}. Alternatively, the discount factor can be thought of as a classical deflator process, if we interpret the company's value as the fair price of a {dividend paying} security in a complete financial market (see, e.g., Sections 6L and 6M in \cite{Duffie}).

In this paper, for the sake of mathematical generality, we take a generic $\rho$ satisfying Assumption \ref{ass:rho}. {That allows an interpretation of the model in which discounting is understood as an ``opportunity cost''. In this interpretation the personal time-preferences of a representative shareholder are linked to the financial market's evolution and, in particular, to the interest paid by an alternative form of investment in a `safe' asset, like a  bond}. {Determining} the structural form of agents' time preferences is a fundamental problem in experimental economics related to utility theory. {A} definitive answer has not been obtained yet {and} we refer to the reviews \cite{Frederik,Harrison} for experimental methods and findings.
\end{remark}
}}

For frequent future use we recall that for any $\beta>0$ one has (see, e.g., \cite{JYC}, Corollary 6.3.4.3, p.\ 362)
\begin{equation}
\label{Laplace}
\E_{r}\Big[e^{-\beta\int_0^{t}R_u du}\Big] = e^{-A_{\beta}(t) - r G_{\beta}(t)},
\end{equation}
with
\begin{equation}\label{GA}
\begin{split}
&G_{\beta}(t)\!:=\! \frac{2\beta\left(e^{\eta_{\beta}\,t}-1\right)}{\eta_{\beta}\left(e^{\eta_{\beta}\,t}+1\right) + k \left(e^{\eta_{\beta}\,t}-1\right)},\\
&A_{\beta}(t)\!:=\! -\frac{2k\theta}{\gamma^2}\ln\left[\frac{2\eta_{\beta} e^{(\eta_{\beta} + k)\frac{t}{2}}}{(\eta_{\beta}+ k)\left(e^{\eta_{\beta}\,t}-1\right) + 2\eta_{\beta}}\right],
\end{split}
\end{equation}
and $\eta_{\beta} := \sqrt{k^2 + 2\,\gamma^2\beta}$.

\subsection{Verification Theorem}
\label{sec:verifThm}

The infinitesimal generator $\cL$ of the pair $(Z^0,R)$ is defined by its action on twice-continuously differentiable functions $f$ as
\begin{align}\label{eq:L}
(\cL f)(r,z) :=\frac{1}{2} \,\sigma^2\, f_{zz}(r,z) + \mu\, f_z(r,z) + \frac{1}{2} \,\gamma^2\,r\, f_{rr}(r,z) + k(\theta- r)\, f_r(r,z),
\end{align}
where we adopt the notation $f_r:=\tfrac{\partial}{\partial r}f$, $f_z:=\tfrac{\partial}{\partial z}f$, $f_{rr}:=\tfrac{\partial^2}{\partial r^2}f$, $f_{rz}:=\tfrac{\partial^2}{\partial r\partial z}f$ and $f_{zz}:=\tfrac{\partial^2}{\partial z^2}f$.

The financial intuition suggests that the firm's manager is more likely to pay dividends when the firm performs well. We thus expect that for each value $r$ of the interest rate, there should be a critical value of the surplus process, such that dividends are paid if $z$ is larger than such a value. Motivated by this intuition and by the idea that a dynamic programming principle should also hold, we formulate the following verification theorem.

For the ease of notation we introduce the sets
\[
\mathcal{O}:=(0,\infty)\times(\alpha,\infty)\:\:\text{and}\:\:\overline{\mathcal{O}}:=[0,\infty)\times[\alpha,\infty).
\]
{\color{black}{Moreover, for an interval $(x_1,x_2)$ of the real line, we adopt the convention that $(x_1,x_2)=\varnothing$ whenever $x_2 \leq x_1$.}}

\begin{theorem}
\label{thm:verif}
{\color{black}{Let Assumption \ref{ass:rho}
% \ref{ass:mu}
and condition \eqref{eq:hpCIR} hold.}}
Assume that there exists functions $a:(0,+\infty)\to [\alpha,+\infty)$ and $v:\overline{\mathcal{O}}\to\R_+$ with the following properties.
\begin{itemize}
	\item [(i)]The mapping $r\mapsto a(r)$ is right-continuous and non-increasing.
\item [(ii)]
The function $v$ is such that $v\in C^1(\mathcal{O})\cap C(\overline{\mathcal{O}})$ with $v_{zz},v_{rz}\in C(\mathcal{O})$
and $v_{rr}\in L^{\infty}_{loc}(\mathcal{O})
\cap C(\bar{\cI} \cap {\mathcal{O}} )$,
where
\begin{align}
\label{eq:I}
\cI:= \{(r,z)\in \mathcal{O}: v_z(r,z) >1\}.
\end{align}
\item[(iii)]The couple $(v,a)$ solves the free-boundary problem
\begin{align} \label{HJB1}
	\left\{
	\begin{array}{ll}
	\cL v(r,z) - \rho(r)\, v(r,z)\leq 0,\quad & {\text{a.e.}\:(r,z) \in \mathcal{O}}\\[+3pt]
	\cL v(r,z) - \rho(r)\, v(r,z)=0,\quad & \alpha < z < a(r),
 \,r > 0\\[+3pt]
	v_z(r,z)  > 1,\quad &\alpha < z < a(r), \,r > 0\\[+3pt]
	v_z(r,z)  = 1,\quad &  z \geq a(r), \,r > 0\\[+3pt]
	v(r, \alpha) =0, \quad  &r \geq 0.
	\end{array}
	\right.
	\end{align}
\end{itemize}
Then, $v\ge V$ on $\overline{\mathcal{O}}$.	
	
In addition, if $v(r,z)\le c(z-\alpha)$ 	
for all $(r,z)\in \overline{\mathcal{O}}$ and some $c>0$, then for every $(r,z)\in \overline{\mathcal{O}}$ we have $v(r,z)=V(r,z)$
and the process
\begin{align}\label{eq:Oc}
D^a_t:=\sup_{0\le s\le t}\left[Z^{z,0}_s-a(R^r_s)\right]^+\,,\qquad t\geq 0,
\end{align}
{\color{black}{with $D^a_{0-}=0$,}} is optimal at $(r,z)$; i.e.,
\begin{equation}
\label{eq:v=Vnew}
v(r,z)=V(r,z)=\E_{r,z}\left[\int_{0-}^{\tau_\alpha^{D^a}}
e^{-\int_0^t\rho(R_t)dt}dD^a_t\right].
\end{equation}
\end{theorem}

\begin{proof}
\textbf{Part 1: Proof that $v\ge V$ on $\overline{\mathcal{O}}$.}

We start arguing as in \cite{FleSo}, Chapter VIII, Theorem 4.1. More precisely, for each $k \geq 1$, we introduce the standard mollifier {\color{black}{$\phi_k(r,z)= k^{-2}\phi(kr,kz)$
with $\phi \in C_c^\infty(B_1(0))$, $\phi \geq 0$, $\int_{\R^2} \phi(r,z) drdz=1$ (where $B_1(0)$ is the ball in $\R^2$ centered in zero with radius one), so that  $\phi_k(r,z) \in C_c^\infty(B_{1/k}(0))$.}} Then we define $(v^k)_{k \geq 1}\subset C^{\infty}(\overline{\mathcal{O}})$ by convolution as  $v^k := v \ast \phi_k$. Thanks to the regularity assumptions on $v$,  for any compact set $K \subset \mathcal{O}$
 we have
\begin{align}
	&\lim_{k \rightarrow \infty}||v^k -v||_{L^\infty(K)}=0, \label{conv_vk}\\
	&\lim_{k \rightarrow \infty}||v_z^k -v_z||_{L^\infty(K)}=0,
	\qquad\!\! 	\lim_{k \rightarrow \infty}||v_r^k -v_r||_{L^\infty(K)}=0\label{conv_vkzvkr},\\
&\lim_{k \rightarrow\infty}||v_{zz}^k -v_{zz}||_{L^\infty(K)}=0, \quad  \lim_{k \rightarrow\infty}||v_{rz}^k -v_{rz}||_{L^\infty(K)}=0. \label{conv_vkzz}
\end{align}
In general $v_{rr}^k$ will not converge to $v_{rr}$ uniformly on every compact subset of $\mathcal O$, since $v_{rr}$ is not continuous. Therefore we cannot expect that $\mathcal L v^k$ converges to
$\mathcal L v$ uniformly on compact sets.
However, by the definition of weak derivative and since $v_{rr} \in L^{\infty}_{loc}(\mathcal{O})$, we have $(v^k)_{rr}=(v_{rr} \ast \phi_k)$. Then, thanks to the continuity of the coefficients in $\mathcal L$ we have
\begin{equation}\label{convLvk}
\lim_{k \rightarrow \infty} ||(\mathcal L v^k)-[(\mathcal L v)\ast \phi_k]||_{L^{\infty}(K)} =0,
\end{equation}
for every compact $K\subset \calo$, using that the minimal distance from $K$ to $\mathcal O$ is strictly positive.
Recalling that $\cL v - \rho(\cdot)\, v\leq 0$ a.e.\ in $\mathcal O$, then it also holds that $(\cL v - \rho(\cdot)\, v)\ast \phi_k\leq 0$ everywhere in $\mathcal O$. Hence \eqref{convLvk} yields
\begin{equation}\label{est_genrho}
\limsup_{k \rightarrow \infty}	\,\sup_{(r,z)\in K}\, (\cL v^k - \rho(r)  v^k)(r,z)\leq 0.
\end{equation}

Let now $(r,z)\in \mathcal{O}$ be given and fixed, and consider an arbitrary admissible dividend strategy $D \in \mathcal{A}$.
 For $0<\varepsilon < z-\alpha$,  set
\[
\eta^{Z^D}_\eps:=\inf\{t\geq 0: \alpha \leq Z^{z,D}_t \leq \alpha \!+\! \varepsilon\}.
\]
Notice that when $\tau^D_\alpha(\omega)=0$ (recall that $\tau^D_\alpha$ is defined in \eqref{tausolv}) also $\eta^{Z^D}_\eps(\omega)=0$ for every $\eps \in(0,z-\alpha)$. Moreover, if $\tau^D_\alpha(\omega)>0$, for every $\delta >0$ sufficiently small we have
$$
\inf_{0\le t\le \tau^D_\alpha(\omega)-\delta} Z^{z,D}_t(\omega)>\alpha,
$$
hence for every $\delta>0$ we find $\eps>0$
such that
$$
\inf_{0\le t\le \tau^D_\alpha(\omega)-\delta} Z^{z,D}_t(\omega)>\alpha+\eps
\; \Longrightarrow \;
\tau^D_\alpha(\omega)-\delta\le \eta^{Z^D}_{\varepsilon}(\omega)
\le \tau^D_\alpha(\omega).
$$
Since $\eta^{Z^D}_{\varepsilon}(\omega)$ is increasing in $\eps$
we conclude that $\eta^{Z^D}_{\varepsilon}(\omega) \uparrow \tau_{\alpha}^D(\omega)$,  $\P_{r,z}$ a.s., as
$\varepsilon \downarrow 0$.

Let us also define
\begin{align*}
\tau^{Z^D}_\eps:=\inf\left\{t \geq 0: \, Z_t^{z,D} \geq  \frac{1}{\eps}\right\}, \quad \eta^R_\eps:=\inf\left\{t\ge 0: R^r_t \notin\left(\eps, \frac{1}{\eps}\right)\right\},
\end{align*}
and
\begin{align*}
\vartheta^D_{\eps}:=\eta^{Z^D}_\eps\wedge \eta_\eps^R \wedge \tau^{Z^D}_\eps.
\end{align*}
We have  $\vartheta^D_{\eps}=\inf\{t \geq 0: \, (R_t^r,Z_t^{z,D}) \notin K_\eps\}$, where  $K_{\eps}= (\eps,\frac{1}{\eps})\times  (\alpha + \varepsilon, \frac{1}{\eps})$.
Since $+\infty$ is unattainable for the processes $R$ and $Z^{D}$ and $0$ is unattainable for $R$, we also have
$\vartheta^D_{\varepsilon} \uparrow \tau_{\alpha}^D$ $\P_{r,z}$ a.s., as $\varepsilon \downarrow 0$.

Let us now  fix $t>0$.
The Dynkin formula applied to the process $e^{-\int_0^{s}\rho(R_u)du}v^k(R_{s},Z^D_{s})$ on the (random) time interval $[0,\vartheta^D_{\varepsilon} \wedge t]$ gives

\begin{align}
\label{Ito2}
 v^k(r,z) =&\, \E_{r,z}\Big[e^{-\int_0^{\vartheta^D_{\varepsilon}\wedge t}\rho(R_u)du}v^k(R_{\vartheta^D_{\varepsilon}\wedge t},Z^D_{\vartheta^D_{\varepsilon}\wedge t})\Big] \\
& - \E_{r,z}\bigg[\int_{0}^{\vartheta^D_{\varepsilon}\wedge t}e^{-\int_0^{s}\rho(R_u)du}\big(\mathcal{L}-\rho(R_s)\big)v^k(R_s,Z^D_s) ds\bigg]    \nonumber\\
& + \E_{r,z}\bigg[\int_{0}^{\vartheta^D_{\varepsilon}\wedge t}e^{-\int_0^{s}\rho(R_u)du}v_z^k(R_s,Z^D_s) dD^c_s\bigg]  \nonumber \\
& - \E_{r,z}\Big[
\sum_{0\leq s \leq \vartheta^D_{\varepsilon}\wedge t}
e^{-\int_0^{s}\rho(R_u)du}\big(v^k(R_s,Z^D_s) - v^k(R_s,Z^D_{s-})\big)\Big], \nonumber
\end{align}
where $D^c$ denotes the continuous part of $D$ and
the final sum is non-zero only for (at most countably many) times $s$ such that
$\Delta D_s:=D_s - D_{s-}>0$.
Notice  that
\begin{align*}
& \sum_{0\leq s \leq \vartheta^D_{\varepsilon}\wedge t}
e^{-\int_0^{s}\rho(R_u)du}\big(v^k(R_s,Z^D_s) - v^k(R_s,Z^D_{s-})\big) \nonumber \\
& = - \sum_{0\leq s \leq \vartheta^D_{\varepsilon}\wedge t}
e^{-\int_0^{s}\rho(R_u)du}\,\int_0^{\Delta D_s} v^k_z(R_s,Z^D_{s-}-y)dy.
\end{align*}
Since $(Z_s^{z,D}, R_s^r)_{0 \leq s\leq \vartheta^D_{\varepsilon}\wedge t} \in {\overline K_\eps}$, using \eqref{conv_vk}-\eqref{conv_vkzvkr}-\eqref{conv_vkzz} and \eqref{est_genrho}, \eqref{Ito2} we obtain, sending $k\to +\infty$,
 \begin{align}
\label{Ito3}
 v(r,z) \geq&\, \E_{r,z}\Big[e^{-\int_0^{\vartheta^D_{\varepsilon}\wedge t}\rho(R_u)du}v(R_{\vartheta^D_{\varepsilon}\wedge t},Z^D_{\vartheta^D_{\varepsilon}\wedge t})\Big]  \\
& + \E_{r,z}\bigg[\int_{0}^{\vartheta^D_{\varepsilon}\wedge t}e^{-\int_0^{s}\rho(R_u)du}v_z(R_s,Z^D_s) dD^c_s\bigg]  \nonumber \\
& + \E_{r,z}\Big[
\sum_{0\leq s \leq \vartheta^D_{\varepsilon}\wedge t}
e^{-\int_0^{s}\rho(R_u)du}\,\int_0^{\Delta D_s} v_z(R_s,Z^D_{s-}-y)dy\Big]. \nonumber
\end{align}
Recalling that $v_z\ge 1$ on $\mathcal{O}$ by \eqref{HJB1} (hence $v\ge 0$ too, since $v(r,\alpha)=0$ for any $r\ge 0$) we obtain from \eqref{Ito3} that
\begin{align}
\label{Ito4First}
 v(r,z) &\geq  \E_{r,z}\Big[e^{-\int_0^{\vartheta^D_{\varepsilon}\wedge t}\rho(R_u)du}v(R_{\vartheta^D_{\varepsilon}\wedge t},Z^D_{\vartheta^D_{\varepsilon}\wedge t})\Big] + \E_{r,z}\bigg[\int_{0}^{\vartheta^D_{\varepsilon}\wedge t}e^{-\int_0^{s}\rho(R_u)du}\, dD_s\bigg] \\
& \geq
\E_{r,z}\bigg[\int_{0}^{\vartheta^D_{\varepsilon}\wedge t}e^{-\int_0^{s}\rho(R_u)du}\, dD_s\bigg]. \nonumber
\end{align}
Then, we can take limits first as $t \uparrow \infty$, and then as $\varepsilon \downarrow 0$, and employ monotone convergence to obtain
\begin{equation}
\label{Ito4}
v(r,z) \geq \E_{r,z}\bigg[\int_{0}^{\tau_{\alpha}^D}e^{-\int_0^{s}\rho(R_u)du}\, dD_s\bigg].
\end{equation}
Since $v\in C(\overline{\mathcal O})$ and $r\mapsto \rho(R^r_t)$ is {$\P$-a.s.}\ continuous and nondecreasing, an application of Fatou's lemma also gives
\begin{align*}
v(0,z)=\lim_{r\downarrow 0}v(r,z)\ge&\, \E_{z}\bigg[\int_{0}^{\tau_{\alpha}^D}\liminf_{r\downarrow 0}e^{-\int_0^{s}\rho(R^r_u)du}\, dD_s\bigg]\\
= &\E_{z}\bigg[\int_{0}^{\tau_{\alpha}^D}e^{-\int_0^{s}\rho(R^0_u)du}\, dD_s\bigg],
\end{align*}
upon noticing that $\tau_\alpha^D$ is independent of $r$. {Finally, we also have $v(r,\alpha) = 0 = V(r,\alpha)$, where the second equality is by definition of $V$.}

{Thus} \eqref{Ito4} is true for any $D\in \mathcal{A}$ and for any $(r,z)\in [0,\infty)\times{[\alpha,\infty)}$ and we conclude that $v \geq V$ on $\overline{\mathcal{O}}$.
\vspace{0.25cm}

\textbf{Part 2: Proof of $v=V$ and \eqref{eq:v=Vnew}}. We divide this part of the proof into three steps.
\vspace{+3pt}

\textbf{Step 1.}
Fix $(r,z)\in [0,+\infty)\times (\alpha,+\infty)$. We are going to prove that the process $D^a$
in \eqref{eq:Oc} belongs to $\cA$ and, $\P_{r,z}$-a.s.,
\begin{equation}
\label{Sk}
Z^{D^a}_t \leq a(R_t)\quad \text{for all}\quad 0\le t\le \tau^{D^a}_\alpha.
\end{equation}
Moreover, we show the Skorokhod minimality condition:
\begin{align}\label{eq:Sk2}
\int_{0}^{\tau^{D^a}_\alpha} \mathds{1}_{\{Z^{D^a}_{t-} < a(R_t)\}}\,dD^a_t = \sum_{0\le t\le \tau^{D^a}_\alpha}\int^{\Delta D^a_t}_0\mathds{1}_{\{Z^{D^a}_{t-}-\zeta < a(R_t)\}}d\zeta=0.
\end{align}

To prove these facts observe first that $D^a$ is by construction  $\F$-adapted and nondecreasing.
Moreover, by definition
of $D^a$ we easily get, for $0\le t\le \tau^{D^a}_\alpha$,
$$
D^a_t-D^a_{t-}=\max\{0, (Z^0_t-a(R_t))^+ -D^a_{t-}\}
=\max\{0,Z^{D^a}_{t-}-a(R_t)\}
\le Z^{D^a}_{t-}-\alpha,
$$
where in the last inequality we used that $a \ge \alpha$. The second equality above also implies
\[
Z^{D^a}_{t-}-\Delta D^a_t=\min\{Z^{D^a}_{t-},a(R_t)\},
\]
which guarantees that the second integral in \eqref{eq:Sk2} equals zero.
Condition \eqref{Sk} follows by definition of $D^a$, upon noticing that
$$
Z^{D^a}_t= Z^0_t-D^a_t\le a(R_t)\quad  \textup{for} \,\,0\le t\le \tau^{D^a}_\alpha, \quad \P_{r,z}\textup{-a.s.}
$$

It remains to show that $D^a$ is right-continuous and that the first integral in \eqref{eq:Sk2} is also zero. Fix $\omega \in \Omega$ (outside of a null set so that $t\mapsto (Z^0_t(\omega),R_t(\omega))$ are continuous) {and let $t\in(0,\tau^{D^a}_\alpha(\omega)]$ be} {\color{black}{such that $Z^{D^a}_{t-}(\omega) = Z^{0}_t(\omega) - D^a_{t-}(\omega) < a(R_t(\omega))$. {Since} $D^a$ is nondecreasing, we have $Z^{D^a}_{t}(\omega) =Z^{0}_t(\omega) - D^a_{t}(\omega) < a(R_t(\omega))$}}, i.e.\ $Z^{0}_t(\omega) - a(R_t(\omega)) < D^a_t(\omega)$. Recalling that $r\mapsto a(r)$ is right-continuous and non-increasing, then it is also lower semi-continuous. Hence $t\mapsto Z^0_t(\omega) - a(R_t(\omega))$ is upper semi-continuous. Then there exists some $\varepsilon:=\varepsilon(\omega,t)>0$ such that
$$\sup_{s\in[t,t+\varepsilon]}\big[Z^{0}_{s}(\omega) - a(R_{s}(\omega))\big]^+ \leq D^a_t(\omega).$$
It thus follows that for all $s\in [t,t+\varepsilon]$ we have
\begin{equation}\label{Daineq}
D^a_s(\omega) = D^a_t(\omega) \vee \sup_{u\in(t,s]}\big[Z^{0}_{u}(\omega) - a(R_{u}(\omega))\big]^+ = D^a_t(\omega).
\end{equation}
In particular, this proves the right-continuity of $D^a$, so that the process $D^a$ belongs to $\cA$.
Moreover, since \eqref{Daineq} holds for any $0<t \leq \tau^{D^a}_\alpha(\omega)$ such that $Z^{D^a}_{t-}(\omega)<a(R_t(\omega))$, the first integral in \eqref{eq:Sk2} is zero.

The above implies that the {triple} $(Z^{D^a},R,D^a)$ solves the Skorokhod reflection problem for the process $(Z^0,R)$ (with reflecting direction $(-1,0)$) in the set $\{\alpha\le z<a(r), \, r\ge 0\}$,
seen as a relatively open\footnote{Note that this set is open since $a$ is right-continuous.} subset of the orthant $\overline{\mathcal{O}}$.
By construction, the process cannot jump into the set $\{\alpha\le z<a(r), \, r\ge 0\}$. Indeed jumps are allowed only at points of left discontinuity of $a$ (hence
when the boundary $\{z=a(r)\}$ contains a vertical segment) and cannot go out of this boundary.
\vspace{+3pt}

\textbf{Step 2.} Here we show that $v=V$.
Fix $(r,z)\in \mathcal{O}$.
We know that \eqref{Ito2} holds for the special choice of control $D^a$.
The process $(Z^{D^a},R)$ is constrained to evolve in the set $\{\alpha\le z\le a(r), \, r\ge 0\}=\overline{\mathcal{I}}$ (cf.~\eqref{eq:I}), and $v_{rr}$ is assumed to be continuous therein.

{\color{black}{It follows that $(Z_s^{z,D^a}, R_s^r)_{0 \leq s\leq \vartheta^{D^a_{\varepsilon}}} \in {\overline{K_\eps \cap \mathcal I}}$}}
and, consequently, that $\mathcal L v^k\to \mathcal L v$ on ${\overline {K_\eps \cap \mathcal I}}$.
Exploiting the second {equation} in \eqref{HJB1} and the continuity of $\mathcal L, \rho, v$, this implies that the second term of the right hand side of \eqref{Ito2} converges to $0$ as $k\to \infty$.
The limit for the first, the third and the fourth term of \eqref{Ito2}
can be instead obtained as in Part 1, thus yielding \eqref{Ito3} with equality for the control $D^a$.
Now, recalling \eqref{Sk}, we see that the random measure $t\mapsto d D^a_t$ is supported on the (random) set of times $t\in[0,\tau^{D^a}_\alpha]$ for which $Z^{D^a}_{t-}\ge a(R_t)$; hence, using the fourth of \eqref{HJB1}, also the inequality of the first line of \eqref{Ito4First} becomes equality when $D=D^a$.

Hence, for $r>0$ we have
\begin{align}
\label{Ito3-bis}
v(r,z) = \E_{r,z}\bigg[e^{-\int_0^{\vartheta^{D^a}_{\varepsilon}\wedge t}\rho(R_u)du}v(R_{\vartheta^{D^a}_{\varepsilon}\wedge t},Z^{D^a}_{\vartheta^{D^a}_{\varepsilon}\wedge t})+\int_{0}^{\vartheta^{D^a}_{\varepsilon}\wedge t}e^{-\int_0^{s}\rho(R_u)du}\, dD^a_s\bigg]
\end{align}
and it remains to take limits as $t \uparrow \infty$ and $\varepsilon \downarrow 0$. Assume for a moment that
\begin{align}\label{eq:trans}
\lim_{\eps\downarrow 0}\lim_{t\uparrow \infty}\E_{r,z}\Big[e^{-\int_0^{\vartheta^{D^a}_{\varepsilon}\wedge t}\rho(R_u)du}v(R_{\vartheta^{D^a}_{\varepsilon}\wedge t},Z^{D^a}_{\vartheta^{D^a}_{\varepsilon}\wedge t})\Big]=0,
\end{align}
then the second term in \eqref{Ito3-bis} also converges by monotone convergence as in \eqref{Ito4} and we have
\[
v(r,z)=\E_{r,z}\bigg[\int_{0}^{\tau^{D^a}_\alpha}e^{-\int_0^{s}\rho(R_u)du}\, dD^a_s\bigg]\le V(r,z)
\]
for all $(r,z)\in\mathcal{O}$. By the result in Part 1 of the proof we conclude that $v=V$ on $\mathcal O$ and $v(r,\alpha)=V(r,\alpha)=0$ for all $r\ge 0$. The result extends to $r=0$ by recalling that $r\mapsto\rho(r)$ is nondecreasing (hence $\rho(R^r_t)\ge \rho(R^0_t)$ for all $t\ge 0$, $\P$-a.s.) and $v\in C(\overline{O})$. Indeed we have
\begin{align*}
V(0,z)\le v(0,z)=&\,\lim_{r\downarrow 0}v(r,z)=\lim_{r\downarrow 0}\sup_{D\in\cA}\E_{r,z}\bigg[\int_{0}^{\tau^{D}_\alpha}e^{-\int_0^{s}\rho(R_u)du}\, dD_s\bigg]\\
\le&\, \sup_{D\in\cA}\E_{0,z}\bigg[\int_{0}^{\tau^{D}_\alpha}e^{-\int_0^{s}\rho(R_u)du}\, dD_s\bigg]=V(0,z),
\end{align*}
where the first inequality was proven in Part 1 above and the second inequality also uses that the set $\cA$ {and the stopping time $\tau_\alpha^D$} do not depend on $r\ge 0$.
\vspace{0.25cm}

\textbf{Step 3.} In this step it only remains to prove \eqref{eq:trans}.
By using that, by assumption, $v(r,z) \leq c(z-\alpha)$ for some $c>0$, we have
\begin{align}
\label{stima1}
& \E_{r,z}\Big[e^{-\int_0^{\vartheta^{D^a}_{\varepsilon}\wedge t}\rho(R_u)du}v(R_{\vartheta^{D^a}_{\varepsilon}\wedge t},Z^{D^a}_{\vartheta^{D^a}_{\varepsilon}\wedge t})\Big]  \\
& \leq c\,\E_{r,z}\Big[e^{-\int_0^{\vartheta^{D^a}_{\varepsilon}}\rho(R_u)du}\big(Z^{D^a}_{\vartheta^{D^a}_{\varepsilon}}-\alpha\big)\mathds{1}_{\{\vartheta^{D^a}_{\varepsilon} < t\}}
\mathds{1}_{\{\vartheta^{D^a}_{\varepsilon} = \eta_\eps^{Z^{D^a}}\}}
\Big]\notag\\
&\quad+ c\,\E_{r,z}\Big[e^{-\int_0^{\vartheta^{D^a}_{\varepsilon}}\rho(R_u)du}\big(Z^{D^a}_{\vartheta^{D^a}_{\varepsilon}}-\alpha\big)\mathds{1}_{\{\vartheta^{D^a}_{\varepsilon} < t\}}
\mathds{1}_{\{\vartheta^{D^a}_{\varepsilon} \neq \eta_\eps^{Z^{D^a}}\}}
\Big]\notag\\
& \quad +  c\,\E_{r,z}\Big[e^{-\int_0^{t}\rho(R_u)du}\big(Z^{D^a}_{t}-\alpha\big)\mathds{1}_{\{\vartheta^{D^a}_{\varepsilon} \geq t\}}\Big] \notag\\
& \leq
c\varepsilon \P_{r,z}[\vartheta^{D^a}_{\varepsilon} = \eta_\eps^{Z^{D^a}}]
\nonumber\\
&\quad+ c\,\E_{r,z}\Big[e^{-\int_0^{\vartheta^{D^a}_{\varepsilon}}
\rho(R_u)du}
\big(z - \alpha + \mu \vartheta^{D^a}_{\varepsilon} + B_{\vartheta^{D^a}_{\varepsilon}}\big)
\mathds{1}_{\{\vartheta^{D^a}_{\varepsilon} < t\}}\mathds{1}_{\{\vartheta^{D^a}_{\varepsilon} \neq \eta_\eps^{Z^{D^a}}\}}\Big]\nonumber\\
&\quad+ c\,
\E_{r,z}\Big[e^{-\int_0^{t}\rho(R_u)du}\big(z-\alpha + \mu t + \sigma B_t\big)\mathds{1}_{\{\vartheta^{D^a}_{\varepsilon} \geq t\}
}\Big]\notag
\end{align}
where we have used that $Z^{D^a}_{\vartheta^{D^a}_{\varepsilon}} \leq \alpha + \varepsilon$ on the event $\{\vartheta^{D^a}_{\varepsilon} = \eta_\eps^{Z^{D^a}}\}$, as well as that $Z^{D^a}_{t} \leq Z^0_t=z + \mu t + \sigma B_t$ for all $t \geq 0$, by \eqref{eq:Z}.
We now estimate the last two terms of \eqref{stima1}.
For the third one, the independence of $B$ and $W$ and standard inequalities give
\begin{align}
\label{stima1bis}
&\E_{r,z}\Big[e^{-\int_0^{t}\rho(R_u)du}\big(z-\alpha + \mu t + \sigma B_t\big)\mathds{1}_{\{\vartheta^{D^a}_{\varepsilon} \geq t\}
}\Big]\\
&\le\big(z - \alpha + {\color{black}{|\mu|}} t + \E[|B_t|]\big)\E_{r}\Big[e^{-\int_0^{t}\rho(R_u)du}\Big]\nonumber\\
&\le(z - \alpha + {\color{black}{|\mu|}} t + \sqrt{t}) \E_{r}\Big[e^{-\int_0^{t}\rho(R_u)du}\Big].\notag
\end{align}

Now we look at the second term. Since $\vartheta^{D^a}_{\varepsilon}<t$ and $B$ and $W$ are independent, we have
\begin{align}
\label{stima1ter}
& \E_{r,z}\left[
e^{-\int_0^{t}\rho(R_u)du}
\big(z - \alpha + \mu \vartheta^{D^a}_{\varepsilon} + B_{\vartheta^{D^a}_{\varepsilon}}\big)
\mathds{1}_{\{\vartheta^{D^a}_{\varepsilon} < t\}}\mathds{1}_{\{\vartheta^{D^a}_{\varepsilon} \neq \eta_\eps^{Z^{D^a}}\}}\right]\\
& \le(z - \alpha + {\color{black}{|\mu|}} t) \E_{r}\Big[e^{-\int_0^{t}\rho(R_u)du}\Big] + \E_{r}\left[e^{-\int_0^{t}\rho(R_u)du}\right] \E\big[\sup_{0\le s \le t}B_{s}\big]
\nonumber\\
&
\le \E_{r}\Big[e^{-\int_0^{t}\rho(R_u)du}\Big] \Big( z - \alpha + {\color{black}{|\mu|}} t + 2 \sqrt{t}\Big),\nonumber
\end{align}
where the final inequality follows by Jensen's and Doob's inequalities for $B$.
Feeding \eqref{stima1bis} and \eqref{stima1ter} back into \eqref{stima1} we obtain, for a suitable constant $C>0$,
\begin{align}
\label{stima2}
\E_{r,z}&
\Big[e^{-\int_0^{\vartheta^{D^a}_{\varepsilon}\wedge t}\rho(R_u)du}v(R_{\vartheta^{D^a}_{\varepsilon}\wedge t},Z^{D^a}_{\vartheta^{D^a}_{\varepsilon}\wedge t})\Big] \\
\leq & c\varepsilon
+ C(z - \alpha + {\color{black}{|\mu|}}t + \sqrt{t}) \E_{r}\Big[e^{-\int_0^{t}\rho(R_u)du}\Big].\notag
\end{align}

We now distinguish two cases coming from Assumption \ref{ass:rho}-(ii). If $\rho(r)\geq c_1$ for any $r\geq 0$ and for some $c_1>0$ then \eqref{eq:trans} is immediately deduced from \eqref{stima2}.
If $\rho(r) \geq c_2 r$ for some $c_2>0$, then
$$
\E\Big[e^{-\int_0^{t}\rho(R^r_u)du}\Big] \leq
\E\Big[e^{-c_2 \int_0^{t} R^r_u du}\Big] =  e^{-A_{c_2}(t) - r G_{c_2}(t)},$$
where we used \eqref{Laplace} and \eqref{GA} for the equality.

Plugging the latter back into \eqref{stima2} we get
\begin{align*}
\E_{r,z}
\Big[e^{-\int_0^{\vartheta^{D^a}_{\varepsilon}\wedge t}\rho(R_u)du}v(R_{\vartheta^{D^a}_{\varepsilon}\wedge t},Z^{D^a}_{\vartheta^{D^a}_{\varepsilon}\wedge t})\Big] \leq  c\varepsilon
+ C(z - \alpha +{\color{black}{|\mu|}} t + \sqrt{t}) e^{-A_{c_2}(t) - r G_{c_2}(t)}
\end{align*}
and \eqref{eq:trans} holds since (cf.\ \eqref{GA}) $G_{c_2}(t)\ge 0$ and $A_{c_2}(t) \approx \frac{k\theta}{\gamma^2}(\eta_{c_2} - k)t$ for $t$ sufficiently large, with $\eta_{c_2}>k$.
\end{proof}

{\color{black}{
In the case $\mu\le 0$ it is intuitively clear that the firm's manager wants to liquidate the fund immediately, by paying dividends in a single transaction, i.e.~$D_0=z-\alpha$. {It is indeed immediate to check that for $\mu\le 0$ the couple $v(r,z)=z-\alpha$ and $a(r)\equiv \alpha$ satisfies (i)--(iii) in Theorem \ref{thm:verif}. Thus the next corollary holds as a simple application of the theorem.}

{
\begin{corollary}
Suppose that $\mu\le 0$. Then $V(r,z)=z-\alpha$ for any $(r,z) \in \overline{\mathcal{O}}$ and {the optimal dividend policy is given by} $(D^{\alpha}_t)_{t\geq0}$ such that $D^{\alpha}_{0-}=0$ and $D^{\alpha}_t=z-\alpha$ for $t\geq 0$.
\end{corollary}
}

As a consequence of the {corollary}, from now on we {require}:
\begin{assumption}
\label{ass:mu}
We have $\mu>0$.
\end{assumption}
{In the rest of the paper we shall always} assume that \eqref{eq:hpCIR} and Assumptions \ref{ass:rho} and \ref{ass:mu} hold without {further mention}.
}}

%%%%%%%%%%%%%%%%%%%%%%%%%%%%%%%%%%%

\section{An Auxiliary Two-Dimensional Optimal Stopping Problem}
\label{sec:OS}

{\color{black}{As we discussed in the Introduction (Subsection \ref{sec:intromethod}), in order to tackle our singular control problem we follow the approach taken in \cite{DeAE17}: (i) we guess a link between the dividend problem and an optimal stopping problem with value function $U$; (ii) we solve the latter by characterizing its optimal stopping boundary $b$; (iii) we go back to the original problem by showing that (cf.\ Theorem \ref{thm:verifico} in Section \ref{sec:solution})
\begin{align}\label{VintU}
V(r,z)=\int_\alpha^z U(r,y)dy,
\end{align}
and that the optimal stopping boundary $b$ of $U$ also triggers the optimal dividend policy (i.e.\ it plays the role of $a$ in \eqref{eq:Oc}).

The present section is devoted to introducing and studying the optimal stopping problem ``associated'' to our original optimal dividend problem. In the optimal stopping problem the underlying process consists of the interest rate process $R$ and of a reflecting diffusion $K$. Moreover, the stopping payoff increases upon each new reflection of $(R,K)$, but it is discounted with the same stochastic dynamic rate as in the original dividend problem. {The heuristic derivation of the connection between the dividend problem and the stopping problem is provided in Section \ref{sec:corrBM} following arguments originally developed in \cite[Section 3]{DeAE17} and later expanded in \cite{DeA}.

}}}

After {formulating the} optimal stopping problem, we divide this section into two parts. First, in Section \ref{sec:basicU} we provide basic properties of the stopping value function $U$ (monotonicity, boundedness and continuity, respectively in Lemma \ref{lem:U-monot}, Proposition \ref{prop:U-bound}, Proposition \ref{prop:Ucont}), which in turn allow us to show that $U$ solves a suitable free boundary problem (Corollary \ref{cor:fbp}).
Second, in Section \ref{sec:U-C1} we prove the global regularity of $U$ (i.e.\ even across the free boundary; cf.\ Proposition \ref{prop:C1}), and {three} additional results on a required boundary condition (Corollary \ref{cor:elb}) and on the {regularity of the} optimal stopping boundary ({\color{black}{Theorem \ref{thm:bC} and}} Proposition \ref{prop:bC}).

We denote {$\cF^B_\infty:=\sigma(B_t,\,t\ge0)$}.
For $t\geq 0$, let
\begin{align}
Y_t:=-\mu t + \sigma B_t, \quad
S_t:=\sup_{0\leq u \leq t}Y_u,\quad\text{and}\quad
K^z_t:=(z-\alpha) \vee S_t - Y_t{+\alpha}.\label{K}
\end{align}
When clear from the context, we will simply write $K_t$ instead $K_t^z$. Notice that, {the process $K$ is an arithmetic Brownian motion reflecting at $\alpha$ and,} according to the discussion at p.\ 2 of \cite{P06}, it is a Markov process. Then, setting
\begin{equation}\label{eq:deflambda}
\lambda=\frac{2\mu}{\sigma^2},
\end{equation}
we introduce the optimal stopping problem
\begin{align}
\label{eq:U2}
U(r,z) = \sup_{\tau \geq 0} \E\left[ e^{\lambda\,\left((z-\alpha) \vee S_\tau-(z-\alpha)\right) - \int_0^\tau \rho(R^r_s) \,ds}\right],\quad (r,z)\in\overline{\mathcal{O}},
\end{align} 	
where the optimization is taken over all the $\mathbb{F}^{K, W}$-stopping times,  where  $\F^{K, W} := (\mathcal F_t^{K, W})_{t \geq 0}$ is the filtration  generated by $K$ and $W$, augmented by the $\P$-null sets.
{Problem \eqref{eq:U2} is the one that we expect to be associated to the original optimal dividend problem via the formula \eqref{VintU} (see Section \ref{sec:corrBM} for details).}

\begin{remark}
\label{rem:OS2}
Due to the presence of the processes $S_t$ and $\int_0^t \rho(R_s^r) ds$ in the exponential of the gain process, the optimal stopping problem \eqref{eq:U2} may appear non-standard in our Markovian set-up. Indeed, the standard form of a Markovian problem involves the expectation of a function of a Markov process, stopped at a stopping time, while the process $S_t$ and $\int_0^t \rho(R_s^r) ds$ are not Markovian. We now show that \eqref{eq:U2} can be rewritten easily as a standard optimal stopping problem.

Denote $I^{i,r}_t:= i + \int_0^t \rho(R_s^r) ds$,
$Y^y_t:= y - \mu t + \sigma B_t$ and notice that $K^z_t + \dvio{Y_t{-\alpha}} = (z-\alpha) \vee S_t$ by \eqref{K} and that the process $(K,Y)$ is Markovian. Then, it is easy to see that for $U$ as in \eqref{eq:U2} we have
\begin{equation}
\label{eq:reprU2}
U(r,z) = e^{i-\lambda y} \widehat{U}(r,z,y,i),
\end{equation}
where $\widehat{U}$ is the value function of the standard optimal stopping problem
\begin{equation}
\label{eq:U2gen}
\widehat{U}(r,z,y,i) = \sup_{\tau \geq 0} \E\left[ e^{\lambda\,\left(K^z_{\tau} + {Y^{y-\alpha}_{\tau}} - (z-\alpha)\right) - I^{i,r}_{\tau}}\right],\quad (r,z,y,i)\in \overline{\mathcal{O}} \times \R \times \R_+,
\end{equation}
for the four-dimensional Markov process $(R_t,K_t,Y_t,I_t)_{t \geq 0}$. However, due to \eqref{eq:reprU2}, we can abandon the general standard formulation \eqref{eq:U2gen} and just consider a problem of optimal stopping for the process $(R_t,K_t)_{t \geq 0}$ rather than for the process $(R_t,K_t,Y_t,I_t)_{t \geq 0}$.
\end{remark}

\begin{remark}
\label{rem:S}
It is worth noticing that, for $r\ge0$,
\[
L_r:=\lim_{t \rightarrow \infty} \left(\lambda\, S_t - \int_0^t \rho(R^r_s) \,ds\right)\le \lambda\, S_\infty
\]
and by \cite[Sec. 3.5.C, Eq.~(5.13)]{KS}
\[
\P(S_\infty > x) = e^{-\lambda x}.
\]
Hence $\P(L_r = + \infty)\le \P(S_\infty=+\infty)=0$ for all $r\ge 0$, since $\mu>0$ (Assumption \ref{ass:mu}).
\end{remark}
From now on we focus on the study of problem \eqref{eq:U2}. We will then prove in Section \ref{sec:solution} how such an optimal stopping problem is related to the original optimal dividend problem.

%%%%%%%%%%%%%%%%%%%%%%%%%%%%%%%%%%%%%%

\subsection{Basic Properties of $U$ and a Free Boundary Problem}
\label{sec:basicU}

It is not hard to verify that, $\P$-almost surely, the map
\begin{align}\label{eq:map}
(r,z)\mapsto \lambda\,[(z-\alpha) \vee S_\tau-(z-\alpha)] - \int_0^\tau \rho(R^r_s) \,ds
\end{align}
is nonincreasing in $z$. Moreover, using comparison theorems for \eqref{eq:R}, we also have that the map in \eqref{eq:map} is nonincreasing in $r$ since $\rho(\,\cdot\,)$  is nondecreasing. These facts imply the next simple result, whose proof is omitted for brevity.
\begin{lemma}\label{lem:U-monot}
The map $z\mapsto U(r,z)$ is nonincreasing for each $r\in\R_+$. Moreover, the map $r\mapsto U(r,z)$ is nonincreasing for each $z\in[\alpha,+\infty)$.
\end{lemma}

The next proposition gives us an important bound on $U$, and estimates obtained in its proof will be used several times in the rest of the paper. It is useful to introduce here the random variables
\begin{align}\label{eq:H}
H^r :=  1 +
\int_0^\infty e^{- c_2\int_0^t R^r_s \,ds} \,\lambda\,e^{\lambda\,S_t}\,d S_t
\end{align}
and
\begin{align}
\label{eq:Sp}
S^p:=\sup_{0\le t<\infty}\left(B_t -p t\right),
\end{align}
where $p:= \mu/\sigma+c_1\sigma/2\mu$ and the constants $c_1,c_2\ge0$ are as in $(ii)$ of Assumption \ref{ass:rho}.
\begin{proposition}\label{prop:U-bound}
Recall $c_1$ and $c_2$ from $(ii)$ in Assumption \ref{ass:rho}. We have
\begin{align}\label{eq:U-bound}
0\le U(r,z)\le h_0,\qquad\text{for all $(r,z)\in\overline{\mathcal{O}}$,}
\end{align}
where
\begin{align*}
\text{$h_0:=\E\big[e^{\lambda\sigma S^p}\big]<+\infty$ if $c_1>0$ \qquad  and \qquad  $h_0:=\sup_{r\in\R_+}\E[ H^r ]<+\infty$ if $c_2>0$.}
\end{align*}
\end{proposition}
\begin{proof}
The lower bound in \eqref{eq:U-bound} is trivial. For the upper bound instead we use Assumption \ref{ass:rho} to write
\begin{align*}
\E&\left[ e^{\lambda\,  ((z-\alpha)\vee S_\tau-(z-\alpha)) - \int_0^\tau \rho(R^r_s) \,ds}\right]\le \E\left[ e^{\lambda\,S_\tau - c_1\tau - c_2 \int_0^\tau R^r_s \,ds}\right].
\end{align*}
Now, if $c_1>0$ we have, by using \eqref{eq:deflambda},
\begin{align}
\label{eq:new}
U(r,z)&\le \sup_\tau \E\left[ e^{\lambda\,S_\tau - c_1\tau}\right]\le \E\big[e^{\lambda\sigma S^p}\big]\\
&=2p\!\int_0^\infty e^{\frac{2\mu}{\sigma}y}e^{-2p y}dy=2p\!\int_0^\infty e^{-\frac{c_1\sigma}{\mu} y}dy<+\infty,\notag
\end{align}
where we used that $\P(S^p> x)=\exp (-2p x)$ (see \cite[Sec. 3.5.C, Eq.~(5.13)]{KS}).

If instead $c_2>0$ (and in particular when $c_1=0$) calculations are a bit more involved. Noticing that the process $S$ is of finite variation, we first use an integration by parts to obtain
\begin{align}\label{eq:Ub1}
U(r,z)\le&\sup_\tau\E\left[ e^{\lambda\,S_\tau - c_2\int_0^\tau R^r_s \,ds}\right]\\
=& 1+\sup_\tau\E\left[\int_0^\tau e^{- c_2\int_0^t R^r_s \,ds} \,\lambda\,e^{\lambda\,S_t}\,d S_t-c_2\!\int_0^\tau e^{- c_2\int_0^t R^r_s \,ds} R^r_t\,e^{\lambda\,S_t}\,dt\right]\notag\\
\le& \E[ H^r ] \le \sup_{r\in\R_+}\E[ H^r ],\notag
\end{align}
where in the last inequality we used that $R^r_t\ge0$ for all $t\ge0$. It remains to prove that $h_0=\sup_{r\in\R_+}\E[ H^r ]<+\infty$.
Letting
\begin{equation}\label{HTr}
H^r_T:=\int_0^T e^{- c_2\int_0^t R^r_s \,ds} \,\lambda\,e^{\lambda\,S_t}\,d S_t
\end{equation}
we have {\color{black}$\E [H^r]=1+ \lim_{T\to\infty}\E [H^r_T]$} by monotone convergence. It is therefore sufficient to find a bound for $\E [H^r_T]$ which is independent of $T$ and $r$.
Using independence of $B$ and $W$, Fubini's theorem and explicit formulae for CIR model (see, e.g., \cite{JYC}, p.~361), we obtain
\begin{align}\label{eq:Ub2}
\E[ H^r_T ]=&\E\left[\E\left(\int_0^T  e^{- c_2\int_0^t R^r_s \,ds}\,\lambda\,e^{\lambda\,S_t}\,d S_t\,\Big|{\cF^B_\infty}\right)\right]\\
=&\E\left[\int_0^T \E\left( e^{- c_2\int_0^t R^r_s \,ds}\,\Big|{\cF^B_\infty}\right)\,\lambda\,e^{\lambda\,S_t}\,d S_t\right]\notag\\[+3pt]
=&\E\left[\int_0^T \E\left( e^{- c_2\int_0^t R^r_s \,ds}\right)\,\lambda\,e^{\lambda\,S_t}\,d S_t\right]\notag\\
=&\E\left[\int_0^T e^{-A_{c_2}(t)-rG_{c_2}(t)}\,\lambda\,e^{\lambda\,S_t}\,d S_t\right]\notag
\end{align}
where $G_{c_2}$ and $A_{c_2}$ are as in \eqref{GA} with $\beta = c_2$, and where $\eta_{c_2} := \sqrt{k^2 + 2\,\gamma^2 \,c_2}$. Setting $f(t):=\E[e^{\lambda S_t}]$, integrating by parts in \eqref{eq:Ub2}, using Fubini and undoing the integration by parts we get
\begin{align}
\nonumber\E[ H^r_T ]=&\,e^{-A_{c_2}(T)-rG_{c_2}(T)}f(T)-e^{-A_{c_2}(0)-rG_{c_2}(0)}- \int_0^T \!\!\E\big[e^{\lambda\,S_t}\big]d\big(e^{-A_{c_2}(t)-rG_{c_2}(t)}\big)\\
=&\int_0^Te^{-A_{c_2}(t)-rG_{c_2}(t)}f'(t)dt,
%\notag
\label{eq:Ub3}\end{align}
where by Sec.~3.5.C in \cite{KS} (upon using equations (5.11) and (5.12) therein, and noticing that our $\P(S_t>b)$ is equal to $\P^{(-\mu)}(T_b\le t)$ in the notation of \cite{KS}) we have
\begin{align}
\label{eq:f}&f(t)=\int_0^\infty e^{\lambda\,z}\left(\int_0^t\frac{1}{\sqrt{2\pi\sigma^2s^3}}\Big(\frac{z+\mu s}{\sigma^2 s}z-1\Big)e^{-\frac{(z+\mu s)^2}{2\sigma^2 s}}ds\right)dz,\\[+3pt]
\label{eq:f'}&f'(t)=\frac{1}{\sqrt{2\pi\sigma^2t^3}}\int_0^\infty e^{\lambda\,z-\frac{(z+\mu t)^2}{2\sigma^2 t}}\Big(\frac{z+\mu t}{\sigma^2 t}z-1\Big)dz.
\end{align}
Recalling that $\lambda=2\mu/\sigma^2$, straightforward algebra gives
\[
\lambda z-\frac{(z+\mu t)^2}{2\sigma^2 t}=-\frac{(z-\mu t)^2}{2\sigma^2 t}.
\]
Changing variable in the integral \eqref{eq:f'} we obtain
\begin{align}\label{tonum}
f'(t)=&\frac{1}{\sqrt{2\pi\sigma^2t^3}}\int_0^\infty e^{-\frac{(z-\mu t)^2}{2\sigma^2 t}}\Big(\frac{z+\mu t}{\sigma^2 t}z-1\Big)dz\\[+3pt]
=&\frac{1}{t}\int_{-\mu t}^\infty\Big(\frac{y+2\mu t}{\sigma^2 t}(y+\mu t)-1\Big)\frac{1}{\sqrt{2\pi\sigma^2t}}e^{-\frac{y^2}{2\sigma^2 t}}dy\notag\\[+3pt]
=&\frac{1}{t}\E\left[\mathds{1}_{\{\sigma B_t\ge -\mu t\}}\Big(\frac{\sigma B_t+2\mu t}{\sigma^2 t}(\sigma B_t+\mu t)-1\Big)\right]\notag \\[+3pt]
=&\frac{1}{t}\E\left[\mathds{1}_{\{\sigma B_t\ge -\mu t\}}\Big(\frac{B^2_t}{t}-1+\frac{3\mu}{\sigma}B_t+\frac{2\mu^2}{\sigma^2}t\Big)\right]\notag\\[+3pt]
\le& \frac{2\mu^2}{\sigma^2}+\frac{3\mu}{\sigma\sqrt{t}}+\frac{1}{t}\E\left[\mathds{1}_{\{\sigma B_t\ge -\mu t\}}\Big(\frac{B^2_t}{t}-1\Big)\right]. \notag
\end{align}
The last term above may be evaluated as follows:
\begin{align}
\label{estimateBM}
\frac{1}{t}&\E\left[\mathds{1}_{\{\sigma B_t\ge -\mu t\}}\Big(\frac{B^2_t}{t}-1\Big)\right]=\frac{1}{t}\int^\infty_{-\frac{\mu t}{\sigma}}\frac{1}{\sqrt{2\pi t}}\Big(\frac{y^2}{t}-1\Big)e^{-\frac{y^2}{2t}}dy\\[+3pt]
=&\frac{1}{t}\left(\int_{-\frac{\mu t}{\sigma}}^\infty \frac{1}{\sqrt{2\pi t}}y\Big(-e^{-\frac{y^2}{2t}}\Big)'dy-\int_{-\frac{\mu t}{\sigma}}^\infty\frac{1}{\sqrt{2\pi t}} e^{-\frac{y^2}{2t}}dy\right)=
-\frac{\mu}{\sigma\sqrt{2\pi}}\frac{1}{\sqrt{t}}  e^{-\frac{\mu^2 t}{2\sigma^2}}<0,\nonumber
\end{align}
where, in the last equality, we have used the integration by parts.
Using \eqref{tonum}-\eqref{estimateBM} above in \eqref{eq:Ub3} we then conclude
\begin{align}\label{EHTr}
\E [H^r_T]\le& \int_0^Te^{-A_{c_2}(t)-rG_{c_2}(t)}\Big(\frac{2\mu^2}{\sigma^2}+\frac{3\mu}{\sigma\sqrt{t}}\Big)dt\\[+3pt]
\le& \int_0^\infty e^{-A_{c_2}(t)}\Big(\frac{2\mu^2}{\sigma^2}+\frac{3\mu}{\sigma\sqrt{t}}\Big)dt<+\infty,\notag
\end{align}
where the last integral is finite because $A_{c_2}(t)\approx \frac{k\theta}{\gamma^2}(\eta_{c_2}-k)t$ as $t\to\infty$, $\eta_{c_2}>k$, and $A_{c_2}(0)=0$.
\end{proof}

An important consequence of the proof of Proposition \ref{prop:U-bound} is that
\begin{align}\label{eq:sup}
\E\left[\sup_{0\le t<\infty}e^{\lambda [(z-\alpha) \vee S_t] -\int_0^t\rho(R^r_s)ds}\right]<+\infty,\quad\text{for all $(r,z)\in\overline{\mathcal{O}}$}.
\end{align}
Moreover, it is not hard to verify that the Markov process $(K_t, S_t, Y_t, R_t, \int^t_0\rho(R_s)ds)_{t\ge0}$ is also of Feller type. Then, \cite[Lemma 3, Sec.~3.2.3 and Lemma 4, Sec.~3.2.4]{S} guarantee that there exists a lower semi-continuous function $u$ which is the smallest superharmonic function larger than one (see Remark \ref{rem:OS2} for a detailed comparison with \cite{S}). Here, superharmonic refers to the property
\[
u(r,z)\ge \E\left[ e^{\lambda [(z-\alpha) \vee S_\tau{-(z-\alpha)}] -\int_0^\tau\rho(R^r_s)ds}u\big(R^r_\tau,K^z_\tau\big)\right]
\]
for any stopping time $\tau$ and any $(r,z)\in \overline{\mathcal {O}}$. Now, let us introduce the sets
\begin{align}
\label{eq:cC} \cC&:=\{(r,z)\in\overline{\mathcal{O}}\,:\,U(r,z)>1\},\\%[+3pt]
\label{eq:cS} \cS &:=\{(r,z)\in\overline{\mathcal{O}}\,:\,U(r,z)=1\},
\end{align}
known in the literature as {\em continuation} and {\em stopping} sets, respectively. Thanks to \cite[Thm. 1, Sec. 3.3.1 and Thm. 3, Sec. 3.3.3]{S}, and the fact that $U$ is lower semi-continuous, we have that $U=u$ and that
\begin{align}\label{eq:tau*}
\tau_*:=\inf\{t\ge0\,:\, (R_t,K_t)\in \cS\}
\end{align}
is the smallest optimal stopping time for \eqref{eq:U2}, provided that $\P_{r,z}(\tau_*<+\infty)=1$, otherwise it is an optimal Markov time.  In some instances below we will stress the dependence on the data $(r,z)$ of the optimal stopping time, i.e.,
\begin{align}\label{eq:tau*bis}
\tau_*(r,z):=\inf\{t\ge0\,:\, (R^r_t,K^z_t)\in \cS\}.
\end{align}

Moreover, recalling again that $U$ is lower semi-continuous and given the process
\begin{align*}
\Lambda_t:=e^{\lambda\left((z-\alpha) \vee S_t-(z-\alpha)\right) -\int_0^t\rho(R_s)ds}U(R_t,K_t),\qquad t\ge0,
\end{align*}
then
\begin{align}\label{eq:supm}
&\text{$(\Lambda_t)_{t\ge0}$ is a $\P_{r,z}$-supermartingale}
\end{align}
and
\begin{align}\label{eq:m}
&\text{$(\Lambda_{t\wedge\tau_*})_{t\ge0}$ is a $\P_{r,z}$-martingale}
\end{align}
for all $(r,z)\in\overline{\mathcal{O}}$ (see \cite[Thm. 2.4, Sec. 2, Chapter I]{PS} or \cite[Sec.~3.4]{S}).
\vspace{+5pt}

Next we provide a technical lemma which is useful to prove continuity of $U$ later on.
\begin{lemma}\label{lem:UT}
For $n>0$, let us denote
\[
U^n(r,z)=\sup_{0\le \tau\le n}\E\left[ e^{\lambda\,\left((z-\alpha) \vee S_\tau-(z-\alpha)\right) - \int_0^\tau \rho(R^r_s) \,ds}\right],\quad (r,z)\in\overline{\mathcal{O}}.
\]
Then for all $(r,z)\in\overline{\mathcal{O}}$ we have
\[
\lim_{n\to\infty}U^n(r,z)=U(r,z).
\]
\end{lemma}
\begin{proof}
Clearly $(U^n)_{n> 0}$ is an increasing sequence and $U^n\le U$ for all $n>0$. Therefore we denote its limit $U^\infty:=\lim_{n\to\infty} U^n\le U$. Let us now fix $(r,z)\in\R_+\times[\alpha,+\infty)$ and let $\tau_*=\tau_*(r,z)$ be optimal for $U(r,z)$. Then
{\color{black}{
\begin{align*}
U^n(r,z)\ge \E_{r,z}\left[e^{\lambda\,\left((z-\alpha) \vee S_{\tau_*\wedge n}-(z-\alpha)\right) -\int_0^{\tau_*\wedge n}\rho(R_t)dt}\right]
\end{align*}
}}
and using Fatou's lemma we conclude
{\color{black}{
\begin{align*}
U^{\infty}(r,z)=\liminf_{n\to\infty}U^n(r,z)\ge &\E_{r,z}\left[\liminf_{n\to\infty}e^{\lambda\,\left((z-\alpha) \vee S_{\tau_*\wedge n}-(z-\alpha)\right)-\int_0^{\tau_*\wedge n}\rho(R_t)dt}\right]\\[+3pt]
=&\E_{r,z}\left[e^{\lambda\,\left((z-\alpha) \vee S_{\tau_*}-(z-\alpha)\right)-\int_0^{\tau_*}\rho(R_t)dt}\right]=U(r,z).
\end{align*}
}}
\end{proof}

We close this section by proving that $U$ is indeed continuous. It is worth remarking that all our results hold without any restriction on $\mu$, $\sigma$, and the only requirement is $2k\theta\ge \gamma^2$ to guarantee strictly positive rates.
\begin{proposition}\label{prop:Ucont}
{\bf (Continuity of $U$)}
The function $U$ is continuous on $\overline{\mathcal{O}}$ and $z\mapsto U(r,z)$ is convex for each $r\in\R_+$.
\end{proposition}
\begin{proof}
First we show convexity. Since
\[
z\mapsto e^{\lambda\,[(z-\alpha) \vee S_\tau-(z-\alpha)] - \int_0^\tau \rho(R^r_s) \,ds}
\]
is convex and $\sup(f+g)\le \sup(f)+\sup(g)$, we easily obtain
\begin{align*}
U&(r,\beta z_1+(1-\beta)z_2)\\
\le& \sup_{\tau\ge0}\E\left[\Big(\beta e^{\lambda\,[(z_1-\alpha) \vee S_\tau-(z_1-\alpha)]}+(1-\beta)e^{\lambda\,[(z_2-\alpha) \vee S_\tau-(z_2-\alpha)]}\Big)e^{-\int_0^\tau \rho(R^r_s) \,ds}\right]\\
\le&\beta U(r,z_1)+(1-\beta)U(r,z_2)
\end{align*}
for all $\beta\in(0,1)$.

Now we show that $z\mapsto U(r,z)$ is continuous uniformly with respect to $r\in\R_+$. Recall that $U(r,\,\cdot\,)$ is decreasing (Lemma \ref{lem:U-monot}), let $z_2>z_1$ and denote by $\tau_1:=\tau_*(r,z_1)$ the optimal stopping time for $U(r,z_1)$. Since $\tau_1$ is suboptimal in $U(r,z_2)$
we get
\begin{align}\label{lip-z}
0\le&\, U(r,z_1)-U(r,z_2)\\[+3pt]
\le&\,\E\left[e^{-\int_0^{\tau_1}\rho(R_t)dt}\Big(e^{\lambda( (z_1-\alpha)\vee S_{\tau_1}-(z_1-\alpha)}-e^{\lambda( (z_2-\alpha)\vee S_{\tau_1}-(z_2-\alpha)}\Big)\right]\notag\\[+3pt]
\le &\,\E\left[\mathds{1}_{\{S_{\tau_1}>z_1-\alpha\}}e^{\lambda S_{\tau_1}-\int_0^{\tau_1}\rho(R_t)dt}\Big(e^{-\lambda (z_1-\alpha)}-e^{-\lambda(z_2-\alpha)}\Big)\right]\notag\\[+3pt]
\le &\,h_0 \Big(e^{-\lambda (z_1-\alpha)}-e^{-\lambda(z_2-\alpha)}\Big)\notag
\end{align}
where $h_0$ is as in Proposition \ref{prop:U-bound} and we have also used that
\[
e^{\lambda ( S_{\tau_1}-(z_2-\alpha))}\le e^{\lambda((z_2-\alpha)\vee S_{\tau_1}-(z_2-\alpha))}.
\]

It only remains to prove that $r\mapsto U(r,z)$ is continuous for each $z\in[\alpha,+\infty)$ given and fixed.

Since $\rho$ is nondecreasing (cf.\ $(i)$ in Assumption \ref{ass:rho}), then $r\mapsto U(r,z)$ is nonincreasing (Lemma \ref{lem:U-monot}) and lower semi-continuous (see the discussion above Lemma \ref{lem:UT}). Hence $r\mapsto U(r,z)$ is right-continuous for each $z\in[\alpha,+\infty)$. Recalling $U^n$ from Lemma \ref{lem:UT}, and noticing that $U(r,z)-U(r-h,z)\le 0$ is {\color{black}{nondecreasing}} as $h\downarrow 0$, we have
\begin{align*}
0\ge& \lim_{h\to 0}\big[U(r,z)-U(r-h,z)\big]=\lim_{h\to 0}\lim_{n\to\infty}\Big[U^n(r,z)-U(r-h,z)\Big]\\[+3pt]
=&\lim_{n\to\infty}\lim_{h\to 0}\Big[U^n(r,z)-U(r-h,z)\Big],
\end{align*}
where we are allowed to swap the limits as both sequences are {\color{black}{nondecreasing}} (as $n\to\infty$ and $h\to 0$). Now we set $\tau_h:=\tau_*(r-h,z)$, which is optimal for $U(r-h,z)$, and consider the suboptimal stopping time $\tau_h\wedge n$ inside $U^n$. With no loss of generality we assume $r-h\ge r_0$ for some $r_0>0$. Then, using that $\rho(R^{r-h}_\cdot)\ge \rho(R^{r_0}_\cdot)$ (in the last term of the expression below), we obtain
\begin{align}\label{eq:cont0}
U^n&(r,z)-U(r-h,z)\\[+3pt]
\ge& \E\left[\mathds{1}_{\{\tau_h\le n\}}e^{\lambda((z-\alpha)\vee S_{\tau_h}-(z-\alpha))-\int_0^{\tau_h}\rho(R^r_t)dt}\left(1-e^{-\int_0^{\tau_h}[\rho(R^{r-h}_t)-\rho(R^r_t)]dt}\right)\right]\notag\\[+3pt]
&+ \E\Big[\mathds{1}_{\{\tau_h> n\}}e^{\lambda((z-\alpha)\vee S_{n}-(z-\alpha))-\int_0^{n}\rho(R^r_t)dt}\cdot\notag\\
&\hspace{+20pt}\cdot\left(1-e^{\lambda((z-\alpha)\vee S_{\tau_h}-(z-\alpha)\vee S_n)-\int_0^{n}[\rho(R^{r-h}_t)-\rho(R^r_t)]dt}e^{-\int_n^{\tau_h}\rho(R^{r_0}_t)dt}\right)\Big].\notag
\end{align}

We make a number of observations: (i) since
{$\tau_h=\inf\{t\ge0\,:\,U(K^z_t,R^{r-h}_t)= 1 \}$}, and $U(z,\,\cdot\,)$ is nonincreasing, we have {\color{black}{$\tau_h\downarrow \eta$, $\P$-a.s.~as $h\to 0$ with $\eta$ a stopping time; (ii) the latter implies that $\P$-a.s.~we have
\begin{align*}
&\lim_{h\to 0} S_{\tau_h}=S_\eta \quad \text{and}\quad\lim_{h\to 0} \int_n^{\tau_h}\rho(R^{r}_t)dt= \int_n^{\eta}\rho(R^{r}_t)dt\quad\text{for all $r>0$};
\end{align*}
}}
(iii) by dominated convergence and continuity of $\rho$ we have, $\P$-a.s.
\[
\lim_{h\to 0}\int_0^{n}\Big|\rho(R^{r-h}_t)-\rho(R^r_t)\Big|dt=0,
\]
which also implies
\[
\lim_{h\to 0}\left(\mathds{1}_{\{\tau_h\le n\}}\int_0^{\tau_h}[\rho(R^{r-h}_t)-\rho(R^r_t)]dt\right)=0.
\]
{\color{black}{Recalling \eqref{eq:sup} we can use dominated convergence in \eqref{eq:cont0} to obtain
\begin{align}\label{eq:cont1}
0\ge& \lim_{n\to\infty}\lim_{h\to 0}\Big[U^n(r,z)-U(r-h,z)\Big]\\[+3pt]
\ge&\lim_{n\to\infty}\E\Big[\mathds{1}_{\{\eta\ge n\}}e^{\lambda((z-\alpha)\vee S_{\eta\wedge n}-(z-\alpha))-\int_0^{\eta\wedge n}\rho(R^r_t)dt}\cdot\notag\\
&\hspace{+40pt}\cdot\left(1-e^{\lambda((z-\alpha)\vee S_{\eta}-(z-\alpha)\vee S_{\eta\wedge n})-\int_{\eta\wedge n}^{\eta}\rho(R^{r_0}_t)dt}\right)\Big].\notag
\end{align}
}}
It is now easy to check that, $\P$-a.s.
{\color{black}{
\[
\lim_{n\to\infty} \Big[\lambda((z-\alpha)\vee S_{\eta}-(z-\alpha)\vee S_{\eta\wedge n})-\int_{\eta\wedge n}^{\eta}\rho(R^{r_0}_t)dt\Big]=0.
\]
}}
Hence, using dominated convergence once again in \eqref{eq:cont1}, gives
\begin{align*}
0\ge&\lim_{h\to 0}\Big[U(r,z)-U(r-h,z)\Big]= \lim_{n\to\infty}\lim_{h\to 0}\Big[U^n(r,z)-U(r-h,z)\Big]\ge 0
\end{align*}
as claimed.
\end{proof}
Continuity of $U$ immediately implies that $\cS$ is closed and that $\cC$ is relatively open in $\overline{\mathcal{O}}$: indeed, by its definition, $\cC$ may not be open in $\R^2$ since it may include a portion of the lines $\{r=0\}$ and $\{z=\alpha\}$. For this reason we will use the notation $\partial \cC$ for the boundary of $\cC$ in $\R^2$ and $\partial_{\overline{\mathcal{O}}} \cC$ for the relative boundary in $\overline{\mathcal{O}}$. Moreover $\rm Int\,\cC$ will denote the interior of $\cC$ in $\R^2$.

Observe now that the (super)martingale property of the process $\Lambda$ (see \eqref{eq:supm} and \eqref{eq:m}), along with standard arguments (see, e.g., \cite[Theorem~2.7.7]{KS2}) give the following corollary.
\begin{corollary}{\bf (Free boundary problem)}\label{cor:fbp}
The function $U$ belongs to $C^{2}$ separately in the interior of $\cC$ and in the interior of $\cS$ (so away from $\partial\cC$), and it satisfies
\begin{align}
&\label{eq:fbp1}\cL U(r,z)-\rho(r)U(r,z)=0, & \text{for $(r,z)\in \rm Int\,\cC$}\\[+3pt]
\label{eq:fbp2}&\cL U(r,z)-\rho(r)U(r,z)=-\rho(r), & \text{for $(r,z)\in \rm Int\,\cS$}\\[+3pt]
&U(r,z)=1, & \text{for $(r,z)\in\partial_{\overline{\mathcal{O}}}\cC$}.
\end{align}
\end{corollary}
\noindent Refined regularity of $U$ and its behaviour at $\R_+\times\{\alpha\}$ will be provided in the next section.

%%%%%%%%%%%%%%%%%%%%%%%%%%%%

\subsection{Differentiability of $U$}
\label{sec:U-C1}

In order to obtain higher regularity properties for $U$ we need some information on the shape of the stopping region $\cS$.  Recalling Lemma \ref{lem:U-monot} (in particular the fact that $U$ is nonincreasing in $z$) and defining, for $r\ge 0$,
\begin{align}\label{eq:b}
b(r):=\sup\{z\in[\alpha,+\infty)\,:\, U(r,z)>1\}
\end{align}
with the convention that $\sup\varnothing=\alpha$, we immediately find, for $r\in\R_+$,
\begin{align}\label{eq:Sr}
\cS_r:=\{z\in[\alpha,+\infty)\,:\,(r,z)\in\cS\}=[b(r),+\infty).
\end{align}
This means that the $r$-section of the stopping set is connected and the graph of the map $r\mapsto b(r)$ describes the boundary
that separates $\cS$ from $\cC$ {(denoted by $\partial_{\overline{\mathcal{O}}}\cC$ above)}. Next we state few important properties of the optimal boundary.
\begin{lemma}
\label{lem:bnew}
{Consider the map $b:\R_+ \to [\alpha,+\infty]$ defined in \eqref{eq:b}}. Then
\begin{align}
\label{eq:b-rc2}
\text{$r\mapsto b(r)$ is nonincreasing and right-continuous.}
\end{align}
Moreover, $b(r)>\alpha$ for all $r\ge 0$.
\end{lemma}
\begin{proof}
 The fact that $\cS$ is closed and \eqref{eq:Sr} imply that $r\mapsto b(r)$ is lower semi-continuous. Indeed take any sequence $(r_n)_{n\ge 1}$ converging to some $r_0\ge 0$. Then
\[
(r_n,b(r_n))\in\cS\implies \cS\ni\liminf_{n\to\infty}(r_n,b(r_n))=(r_0,\liminf_{n\to\infty} b(r_n))
\]
and by \eqref{eq:b} we have $\liminf_{n\to\infty} b(r_n)\ge b(r_0)$.
Using again Lemma \ref{lem:U-monot} (in particular the fact that $U$ is nonincreasing in $r$) we have
\begin{align}
\label{eq:b-monot2} &
(r,z)\in\cS\implies [r,+\infty)\times\{z\}\in\cS,
\end{align}
i.e.,
$r\mapsto b(r)$ is nonincreasing. Since $b(\cdot)$ is also lower semi-continuous, then \eqref{eq:b-rc2} holds.

It only remains to prove the final statement. Take any $r_0\ge 0$, fix $\eps>0$ and denote $\tau_\eps=\inf\{t\ge 0: R^{r_0}_t\ge r_0+\eps\}$. For any $t>0$ the stopping time $\tau_\eps\wedge t$ is admissible and suboptimal for $U(r_0,\alpha)$ so that
\begin{align}\label{eq:ba0}
U(r_0,\alpha)\ge \E\left[e^{\lambda S_{\tau_\eps\wedge t}-\int_0^{\tau_\eps\wedge t}\rho(R^{r_0}_s)ds}\right]\ge \exp\left(\E\left[\lambda S_{\tau_\eps\wedge t}-\int_0^{\tau_\eps\wedge t}\rho(R^{r_0}_s)ds\right]\right),
\end{align}
where the final inequality is due to Jensen's inequality. Recalling that $\rho$ is nondecreasing (Assumption \ref{ass:rho}) we have
\[
\int_0^{\tau_\eps\wedge t}\rho(R^{r_0}_s)ds\le \bar \rho_\eps(\tau_\eps\wedge t),
\]
with $\bar \rho_\eps:=\rho(r_0+\eps)=\sup_{0\le r\le r_0+\eps}\rho(r)$. Now we use estimates as in \cite[Lemma 15]{P19}. In particular, we have
\begin{align}\label{eq:ba1}
&\E\left[\lambda S_{\tau_\eps\wedge t}-\int_0^{\tau_\eps\wedge t}\rho(R^{r_0}_s)ds\right]\ge\E\left[\lambda\sigma \sup_{0\le s\le \tau_\eps\wedge t}B_s-(\mu+\bar\rho_\eps)(\tau_\eps\wedge t)\right] \\
&\ge \lambda\sigma\E\left[\sup_{0\le s\le t}B_s-\mathds{1}_{\{\tau_\eps\le t\}}\sup_{0\le s\le t}B_s \right]-(\mu+\bar\rho_\eps)t\notag\\
&\ge \lambda\sigma\E\left[\sup_{0\le s\le t}B_s\right]-\lambda\sigma\P(\tau_\eps\le t)^\frac{1}{2}\E\left[\big(\sup_{0\le s\le t}B_s\big)^2\right]^{\frac{1}{2}}-(\mu+\bar\rho_\eps)t\notag\\
&=\lambda\sigma\sqrt{t}\left(1-\P(\tau_\eps\le t)^\frac{1}{2}\right)-(\mu+\bar\rho_\eps)t\notag
\end{align}
where in the final inequality we used that $\sup_{0\le s\le t}B_s=|B_t|$ in law. Since $\P(\tau_\eps> 0)=1$ and, consequently, $\P(\tau_\eps \le t)\to 0$ as we let $t\to 0$, we have that the term involving $\sqrt{t}$ dominates. Hence, plugging \eqref{eq:ba1} in \eqref{eq:ba0} and choosing $t$ sufficiently small we reach
$U(r_0,\alpha)>1$ which implies $b(r_0)>\alpha$. Since $r_0\ge 0$ was arbitrary, the proof is complete.
\end{proof}

The simple properties that we have obtained above are crucial to guarantee global $C^1$ regularity of $U$.
We start by noticing that $K$ and $R$ are independent and have transition densities $p^K(t,z;z')$ and $p^R(t,r;r')$, respectively, which are continuous  with respect to the initial point, i.e.~$z\mapsto p^K(t,z;z')$ and $r\mapsto p^R(t,r;r')$ are continuous for all $t>0$, $z' \in [\alpha, + \infty)$, $r' \in [0,\,+ \infty)$. Then it is not hard to verify that the process $(R_t,K_t)_{t\ge 0}$ is strong Feller, i.e.~for any Borel measurable and bounded function $f:\R_+\times \R_+$ and any $t>0$, it holds that $(r,z)\mapsto\E_{r,z}[f(R_t,K_t)]$ is continuous. We then have the following important result.

\begin{lemma}
\label{lem:cont-taustar}
For any $(r_0,z_0)\in\partial_{\overline{\mathcal O}}\cC$ and any sequence $(r_n,z_n)_{n\ge 1}\subset \cC$ such that $(r_n,z_n) \rightarrow (r_0,z_0)$ as $n\to\infty$, we have
\begin{equation}\label{taustar_cont}
\lim_{n\to\infty}	\tau_*(r_n,z_n)= 0,\quad  \P\textup{-a.s.}
\end{equation}
\end{lemma}
\begin{proof}
Let us denote by $\sigma_*$ the first hitting time of $(K,R)$ to $\cS$:
$$
\sigma_*(r,z):=\inf\{t>0\,:\,(R^r_t, K^z_t)\in \cS\}.
$$
It is well known (see \cite[Chapter 13.1-2, Vol. II]{DYN}) that since $(R_t, K_t)_{t\ge 0}$ is a strong Feller process, \eqref{taustar_cont} holds if and only if all the boundary points are regular for $\cS$, namely
\begin{equation}
\label{lem:reg-b}
\P_{r,z}(\sigma_*=0) =1  \qquad   \forall(r,z)\in\partial_{\overline{\mathcal O}}\cC.
\end{equation}
(For further details on the above statement the reader may consult, e.g., \cite[Theorem 2.12, Ch. 4.2]{KS} and \cite[pp.\ 4-5 and Corollary 2]{DeAPe17}.)

Denoting by
$$\widehat{\sigma}_*(r,z):=\inf\{t\geq 0\,:\,(R^r_t, K^z_t)\in \text{Int}\,\cS\}$$
the first entry time of $(R^r_t, K^z_t)_{t\geq0}$ to the interior of $\cS$, and noticing that $\sigma_* \leq \widehat{\sigma}_*$, we now prove \eqref{lem:reg-b} by showing that
\begin{equation}
\label{lem:reg-b-2}
\P_{r,z}(\widehat{\sigma}_*=0) =1  \qquad   \forall(r,z)\in\partial_{\overline{\mathcal O}}\cC.
\end{equation}

Let $(r_0, z_0)\in \partial_{\overline{\mathcal O}} \mathcal{C}$. Define $\mathcal{R}:=[r_0,\infty) \times [z_0,\infty)$, and denote by  $\text{Int}\,\mathcal R$ and $\partial \mathcal R$ respectively its interior and its boundary in $\R^2$. Since $r \mapsto b(r)$ is nonincreasing, we have $\mathcal{R} \subseteq \mathcal{S}$. Also, let $\mathcal{K}$ be a compact neighbourhood of {\color{black}{$(r_0,z_0)$}} and let $\text{Int}\,\mathcal K$ and $\partial \mathcal K$ denote respectively its interior and its boundary in $\R^2$. Since $(r_0,z_0)\in \partial_{\overline{\mathcal O}} \mathcal{C}$ then $r_0>0$ and we assume that $\mathcal K\cap\{r=0\}=\varnothing$.
%, $\mathcal{K} \cap \mathcal{R} \neq \varnothing$, and $((0,\infty) \times [\alpha,\infty)) \setminus \mathcal{K} \cap \mathcal{R} \neq \varnothing$.
%\gre{FG: I do not understand how such $\mathcal{K}$ is chosen. If we ask $(r_0,z_0)\in \mathcal{K}$ then automatically we have $\mathcal{K} \cap \mathcal{R} \neq \varnothing$, we do not have to require it. Similarly, by compactness, $((0,\infty) \times [\alpha,\infty)) \setminus \mathcal{K} \cap \mathcal{R} \neq \varnothing$ is automatically true. Maybe I miss something?}
Then there exists some $\eta_{\mathcal{K}}>0$ such that
\begin{equation}
\label{eq:uiK}
\eta^{-1}_{\mathcal{K}} \geq \gamma^2 r \geq \eta_{\mathcal{K}}\quad\text{on $\mathcal K$}
\end{equation}
so that the diffusion coefficient of the process $(R_t)_{t\geq0}$ is uniformly non degenerate over $\mathcal{K}$. Let us define an auxiliary process $(\widetilde{R}_t)_{t\geq0}$ with dynamics
\begin{align*}
&d\widetilde{R}_t = b_{\mathcal{K}}(\widetilde{R}_t) dt + \gamma_{\mathcal{K}}(\widetilde{R}_t) dW_t, \qquad \widetilde{R}_0 = r,\\
&d\widetilde{K}_t = \mu dt + \sigma dB_t, \qquad \qquad\qquad\quad  \widetilde{K}_0 = z,
\end{align*}
where $b_{\mathcal{K}}(r)=\kappa(\theta-r)$ and $\gamma_{\mathcal{K}}(r) = \gamma \sqrt{r}$ on $\mathcal{K}$, and are continuously extended to be constant outside $\mathcal{K}$. Notice that the uniform ellipticity condition \eqref{eq:uiK} holds for $\gamma_{\mathcal K}$ on the whole $\R$.

Since the process $(\widetilde{R}_t, \widetilde{K}_t)_{t\ge 0}$ is non degenerate over the whole $\R^2$, it admits a continuous transition density $\widetilde p(\cdot, \cdot, \cdot; r,z)$ such that, for any $t >0$
\begin{equation}
\label{stime-p}
\frac{M}{t} e^{-\lambda_0 \frac{|r-\bar r|^2+|z - \bar z|^2}{t}} \geq \widetilde{p}(t, \bar r, \bar z; r,z) \geq \frac{m}{t}e^{-\Lambda_0 \frac{|r-\bar r|^2+|z - \bar z|^2}{t}}
\end{equation}
for some constants $M >m >0$, $\Lambda_0 >\lambda_0 >0$ (see, e.g., \cite[Theorem 1]{Aronson}).
Moreover, denoting
\[
\tau_{\mathcal{K}}:= \inf\{t\geq 0:\, ({R}_t,K_t) \notin\text{Int}\,\mathcal{K}\times(\alpha,\infty)\}
\]
and
\[
\widetilde{\tau}_{\mathcal{K}}:= \inf\{t\geq 0:\, (\widetilde{R}_t,\widetilde{K}_t) \notin \text{Int}\,\mathcal{K}\times(\alpha,\infty)\},
\]
we have that
\begin{align}\label{SDEuniq}
({R}_{t\wedge \tau_{\mathcal{K}}},K_{t\wedge \tau_{\mathcal{K}}}) = (\widetilde{R}_{t \wedge \widetilde{\tau}_{\mathcal{K}}},\widetilde{K}_{t\wedge \widetilde{\tau}_{\mathcal{K}}}), \quad \P_{r_0,z_0}\textup{-a.s.}
\end{align}
by uniqueness of the solution of the SDE (recall that the reflected process $K$ is just a Brownian motion with drift away from the reflection point $\alpha$).

Now, let $\mathcal{R}'$ be a (half) cone with vertex in $(r_0,z_0)$, whose closure is contained in $\text{Int}\,\mathcal{R}\cup (r_0,z_0)$,
and denote by $\widehat{\sigma}'_{\mathcal{R}}$ and $\widetilde{\sigma}'_{\mathcal{R}}$ the corresponding entry times of $(R,K)$ and $(\widetilde R,\widetilde K)$, respectively, {\color{black}{into the interior of $\mathcal{R}'$.}}
%denoted by $\partial \mathcal R'$.
Notice that this additional cone is needed in the argument that follows because $(t_0,z_0)$ may lie on a horizontal/vertical stretch of the boundary $\partial_{\overline{\mathcal O}}\cC$, in which case $(\partial_{\overline{\mathcal O}}\cC\cap\partial\mathcal R)\setminus(r_0,z_0)\neq\varnothing$ whereas $(\partial_{\overline{\mathcal O}}\cC\cap\partial\mathcal R')\setminus(r_0,z_0)=\varnothing$ always holds.
Fixing $t>0$ we then have, {\color{black}{using first that
$\text{Int}\,\mathcal{R}\subseteq \text{Int}\,\mathcal{S}$,
\begin{align}
\label{eq:estimsigma}
&\P_{r_0,z_0}(\widehat{\sigma}_* \leq t)  \geq \P_{r_0,z_0}(\widehat{\sigma}'_{\mathcal{R}} \leq t) \geq \P_{r_0,z_0}(\widehat{\sigma}'_{\mathcal{R}} \leq t, \tau_{\mathcal{K}}>t) \\
& = \P_{r_0,z_0}(\widetilde{\sigma}'_{\mathcal{R}} \leq t, \widetilde{\tau}_{\mathcal{K}}>t)=
\P_{r_0,z_0}(\widetilde{\sigma}'_{\mathcal{R}} \leq t) - \P_{r_0,z_0}(\widetilde{\sigma}'_{\mathcal{R}} \leq t,\widetilde{\tau}_{\mathcal{K}}\leq t)
\notag \\
&\ge
\P_{r_0,z_0}(\widetilde{\sigma}'_{\mathcal{R}} \leq t) - \P_{r_0,z_0}(\widetilde{\tau}_{\mathcal{K}}\leq t),\notag
\end{align}
}}
where the first equality holds by \eqref{SDEuniq}.
Thanks to \eqref{stime-p}
\begin{align}
\label{eq:estimsigma-2}
& \P_{r_0,z_0}(\widetilde{\sigma}'_{\mathcal{R}} \leq t) = \int_{\mathcal{R}'} \widetilde{p}(t, r_0, z_0 ; r,z) dr dz \geq \int_{\mathcal{R}'} \frac{m}{t}e^{-\Lambda_0 \frac{|r-r_0|^2+|z -z_0|^2}{t}}dr dz.
\end{align}
Using the fact that  the change of variable  $s=\frac{r-r_0}{\sqrt {t}}$, $\zeta=\frac{z-z_0}{\sqrt {t}}$ maps the cone $\mathcal{R}'$ into a cone $\mathcal R'_0$ with the same aperture but vertex in $(0,0)$, we get
\begin{align*}
\P_{r_0,z_0}(\widetilde{\sigma}'_{\mathcal{R}} \leq t) & \ge \int_{\mathcal{R}'_0} m \,e^{-\Lambda_0 (s^2 +\zeta^2)}ds d\zeta=: q >0.
\end{align*}
Letting $t \rightarrow 0$ we obtain $\P_{r_0,z_0}(\widetilde{\sigma}'_{\mathcal{R}} = 0)  \ge  q >0$ and therefore, by \eqref{eq:estimsigma}, also that
$\P_{r_0,z_0}(\widehat{\sigma}_* = 0) \ge  q >0$ upon noting that $\P_{r_0,z_0}(\widetilde{\tau}_{\mathcal{K}}\leq t) \rightarrow 0$ as $t \rightarrow 0$.

Since $\{\widehat{\sigma}_* = 0\}$ is measurable with respect to the trivial $\sigma$-algebra $\mathcal F_0^{K, W}$, by the Blumenthal's 0-1 Law we obtain $\P_{r_0,z_0}(\widehat{\sigma}_* = 0)=1$, which completes the proof.
\end{proof}

%To proceed further and prove that $U \in C^{1}$ we impose the next condition on $\rho$, which, together with \eqref{eq:hpCIR}, Assumption \ref{ass:rho} and Assumption \ref{ass:mu}, will be standing throughout the rest of this paper.
%
%\gre{FG: maybe this assumption should be put at the beginning? E: se lo anticipiamo dobbiamo poi specificare dove  si usano le ipotesi}
%
%\begin{assumption}
%\label{ass:rho2}
%We have $\rho \in C^1(\mathbb{R}^+)$ and there exist $C>0$ and $q\in\mathbb{N}$ such that $\rho'(r) \leq C\big(1 + r^q\big)$ for any $r\ge0$.
%\end{assumption}
%
%Using the assumption and arguments as in the proof of Proposition \ref{prop:U-bound} we obtain the next result, which will be useful in proving $C^1$-regularity of $U$.
\begin{lemma}
\label{lem:bound}
{Fix $q\in\mathbb N$}. There is a constant $c>0$ such that, for all $\mathbb{F}^{K,W}$-stopping times $\tau$, and any $(r,z)\in\R_+\times[\alpha,+\infty)$, it holds
\begin{align}
\E\left[e^{\lambda S_\tau-\int_0^\tau\rho(R^r_t)dt}\int_0^\tau e^{-\frac{k}{2}t}
%\rho'(R_t)
\left[1+(R^r_t)^{1+q}\right]
\sqrt{R_t}\, dt\right]\le c.
\end{align}
Moreover the family
\[
\left\{e^{\lambda S_\tau-\int_0^\tau\rho(R^r_t)dt}\int_0^\tau e^{-\frac{k}{2}t}
%\rho'(R_t)
\left[1+(R^r_t)^{1+q}\right]
\sqrt{R_t}\, dt,\:\:\tau\ge 0\right\}
\]
is uniformly integrable.
\end{lemma}
\begin{proof}
Using that $S$ is of finite variation we integrate by parts
%change of variable formula
to get a first, convenient, upper bound
\begin{align*}
&e^{\lambda S_\tau-\int_0^\tau\rho(R^r_t)dt}\int_0^\tau e^{-\frac{k}{2}t}
\left[1+(R^r_t)^{1+q}\right]
\sqrt{R^r_t}dt\\
&\le\!\!
\lambda\int_0^\tau\!\! e^{\lambda S_t-\int_0^t\rho(R^r_s)ds}\left(\int_0^t\! e^{-\frac{k}{2}s}
\left[1+(R^r_s)^{1+q}\right]
\sqrt{R^r_s}ds\right)dS_t \\
& \hspace{1.5cm}
+ \int_0^\tau\!\! e^{\lambda S_t-\int_0^t\rho(R^r_s)ds-\frac{k}{2}t}
\left[1+(R^r_t)^{1+q}\right]
\sqrt{R^r_t}dt\\
&\le\!\!
\lambda\int_0^\infty\!\! e^{\lambda S_t-\int_0^t\rho(R^r_s)ds}\left(\int_0^t\! e^{-\frac{k}{2}s}
\left[1+(R^r_s)^{1+q}\right]
\sqrt{R^r_s}ds\right)dS_t\\
& \hspace{1.5cm} +\!\!\int_0^\infty\!\! e^{\lambda S_t-\int_0^t\rho(R^r_s)ds-\frac{k}{2}t}
\left[1+(R^r_t)^{1+q}\right]
\sqrt{R^r_t}dt\\
&=:A+B.
\end{align*}
Hence, to prove both claims of this lemma it is enough to show that $\E[A]+\E[B]<+\infty$.

We start by proving that $\E[B]<+\infty$. Using that $\rho\ge 0$ (see Assumption \ref{ass:rho}), that $\sqrt{r}\le 1+r$, Fubini's theorem and independence of $S_t$ and $R_t$ we obtain
{\color{black}{
\begin{align*}
\E[B]\le c\int_0^\infty\E\left[
e^{\lambda S_t-\frac{k}{4}t}\right]
e^{-\frac{k}{4}t}\E\left[1+ R^r_t + (R^r_t)^{1+q} + (R^r_t)^{2+q}\right]dt
\end{align*}
}}for some constant $c>0$, which will vary from line to line.
Observe now that (recall \eqref{eq:Sp})
$$
\lambda S_t-\frac{k}{4}t\le
\lambda\sigma \sup_{0\le s \le t} \left(B_s-\frac{\mu}{\sigma}s-\frac{k}{4 \lambda\sigma}s\right)
\le\lambda\sigma S^p, \qquad \hbox{with $p=\frac{\mu}{\sigma}+\frac{k}{4 \lambda\sigma}$.}
$$
Since $\P(S^p_\infty>x)=\exp(-2px)$ for $p>0$ (see Remark \ref{rem:S}), as in \eqref{eq:new} we easily get $\E\left[\exp(\lambda S_t-kt/4 )\right]\le c'$ for some $c'>0$.
Hence
{\color{black}{
\begin{align}\label{eq:b0}
\E[B]\le c\int_0^\infty e^{-\frac{k}{4}t}\E\left[1+ R^r_t + (R^r_t)^{1+q} + (R^r_t)^{2+q}\right]dt.
\end{align}
Now we recall \cite[Thm.~2.3]{Du}, which states that, for any $\zeta \in \mathbb{N}$, there is a constant $C_{\zeta}\!>\!0$, only depending on $\zeta$ and the coefficients of the SDE \eqref{eq:R}, such that
\begin{align}\label{bDu}
\E\left[(R^r_t)^{\zeta}\right]\le C_{\zeta},\quad\text{for all $t\ge 0$}.
\end{align}
Using the latter bound in \eqref{eq:b0} for $\zeta=\{1,1+q,2+q\}$ we get $\E[B]<+\infty$.}}

Next we show that $\E[A]<+\infty$. We only provide full details in the case $\rho(r)\ge c_2 r$ (see Assumption \ref{ass:rho}), since the case $\rho(r)\ge c_1$ is easier and can be dealt with in the same way. Below we use $\E[A]=\E[\E(A|{\cF^B_\infty})]$ and independence of $R$ from ${\cF^B_\infty}$. Then, recalling that $\sqrt{r}\le 1+r$, by Fubini's theorem we obtain {\color{black}{
\begin{align}\label{eq:b1}
\E[A]\le& c\,\E\left[\int_0^\infty e^{\lambda S_t}\E\left(e^{-c_2\int_0^t R^r_s ds}\int_0^t(1+ R^r_s + (R^r_s)^{1+q} + (R^r_s)^{2+q})e^{-\frac{k}{2}s}ds\big|{\cF^B_\infty}\right)dS_t\right]\\
\le & c\,\E\left[\int_0^\infty e^{\lambda S_t}\E\left(e^{-c_2\int_0^t R^r_s ds}\int_0^t(1 + R^r_s + (R^r_s)^{1+q} + (R^r_s)^{2+q})e^{-\frac{k}{2}s}ds\right)dS_t\right]\notag
\end{align}
}}for some constant $c>0$, which will vary from line to line. Repeated use of H\"older inequality and \eqref{Laplace} give {\color{black}{
\begin{align*}
&\E\left(e^{-c_2\int_0^t R^r_s ds}\int_0^t(1 + R^r_s + (R^r_s)^{1+q} + (R^r_s)^{2+q})e^{-\frac{k}{2}s}ds\right)\\
&\le \E\left(e^{- 2 c_2\int_0^t R^r_s ds}\right)^{\frac{1}{2}}\E\left[\left(\int_0^t(1 + R^r_s + (R^r_s)^{1+q} + (R^r_s)^{2+q})e^{-\frac{k}{2}s}ds\right)^2\right]^{\frac{1}{2}}\\
&\le e^{- \frac{1}{2} A_{2 c_2}(t)-\frac{r}{2}G_{2c_2}(t)}\E\left[\int_0^t e^{-\frac{k}{2}s}ds \int_0^t(1 + R^r_s + (R^r_s)^{1+q} + (R^r_s)^{2+q})^2e^{-\frac{k}{2}s}ds \right]^{\frac{1}{2}}\\
&\le C'_{q}\,\, e^{- \frac{1}{2} A_{2 c_2}(t)-\frac{r}{2}G_{2c_2}(t)},
\end{align*}
where the final inequality follows from \eqref{bDu}, for $\zeta=\{1,1+q,2+q\}$, and with some $C'_q>0$.}}

Plugging the last expression above in \eqref{eq:b1} gives
\begin{align*}
\E[A]\le&c\,\E\left[\int_0^\infty e^{\lambda S_t} e^{- \frac{1}{2} A_{2 c_2}(t)-\frac{r}{2}G_{2c_2}(t)}dS_t\right].
\end{align*}
The latter can be treated exactly by the same methods that we used to estimate \eqref{eq:Ub2}, hence $\E[A]<+\infty$.
\end{proof}

The methodology that we adopt to prove $C^1$ regularity of the value function was developed in \cite{DeAPe17} for general multi-dimensional, finite-time and infinite-time horizon, optimal stopping problems. However, due to the square root in the diffusion coefficient of the CIR dynamics, some of the integrability conditions required in \cite{DeAPe17} seem difficult to verify directly. So  in the proof of Proposition \ref{prop:C1} below we adapt the method to our setting.

\begin{proposition}{\bf ($C^1$ regularity of $U$)}
\label{prop:C1}
One has that $U\in C^1(\mathcal{O})$. Moreover
%$U_z$ exists on $\R_+\times(\alpha,+\infty)$ and}
\begin{align}
\label{eq:Uz}
U_z(r,z)=-\lambda\E_{r,z}\left[\mathds{1}_{\{S_{\tau_*}>z-\alpha\}}e^{\lambda (S_{\tau_*}-(z-\alpha))-\int_0^{\tau_*}\rho(R_t)dt}\right]
\end{align}
for all $(r,z)\in \mathcal{O}$.
\end{proposition}
\begin{proof}
The proof is organized in two steps.

\emph{Step 1.} We start by noticing that \eqref{eq:Uz} trivially holds in the interior of $\cS$ with $U_z=0$. Further, we know that $U_z$ is continuous in $\rm Int\,\cC$, so that if we can prove \eqref{eq:Uz} in $\rm Int\,\cC$, then Lemma \ref{lem:cont-taustar} and the use of dominated convergence will also imply continuity of $U_z$ across $\partial_{\overline{\mathcal{O}}} \cC$. Finally, to show that \eqref{eq:Uz} holds in $\rm Int\,\cC$ we can repeat the same steps as in the proof of \cite[Thm.~5.3]{DeAE17}, upon replacing the discount factor therein by $\int_0^{\tau_*} \rho(R_s) ds$. We omit further details in the interest of brevity.

\emph{Step 2.} Here we prove that $U_r\in C(\mathcal{O})$. We know that $U_r$ is continuous separately in Int$\,\cC$ and Int$\,\cS$. Then, it suffices to prove continuity across the boundary $\partial_{\overline{\mathcal{O}}} \cC$. We start finding bounds on $U_r$.

Fix $(r,z) \in \rm Int\,\cC$, $\varepsilon \in (0, \varepsilon_0)$, and denote $\tau_* := \tau_*(r, z)$. Recalling Lemma \ref{lem:U-monot} and optimality of $\tau_*$ for $U(r,z)$, we obtain
\begin{align}\label{ineqU1}
0 &\geq \frac{U(r+\eps,z)-U(r, z)}{\varepsilon}
\\
&\geq 	\frac{1}{\varepsilon}\E\left[e^{\lambda((z-\alpha)\vee S_{\tau_*}- (z-\alpha)) -\int_0^{\tau_*} \rho(R_s^{r+\varepsilon})ds} \left(1-e^{\int_0^{\tau_*}(\rho(R_s^{r+\eps})-\rho(R_s^{r}))ds}\right)\right]
\notag
\\
&\geq 	\frac{1}{\varepsilon}\E\left[e^{\lambda((z-\alpha)\vee S_{\tau_*}- (z-\alpha)) -\int_0^{\tau_*} \rho(R_s^{r})ds} \left(1-e^{\int_0^{\tau_*}
%\rho'(\theta_s)
c_3(1+(R^{r+\eps}_s)^q)({\color{black}{\sqrt{R^{r+\eps}_s}-\sqrt{R^r_s}) ds} }}\right)\right],\notag
\end{align}
where in the last inequality we have used
Assumption \ref{ass:rho}, (i) and (iii), and the fact that
$r\mapsto R^r$ is nondecreasing.
%that $\rho'\ge 0$ and
%\[
%\int_0^{\tau_*}(\rho(R_s^{r+\eps})-\rho(R_s^{r})) ds = \int_0^{\tau_*} \rho'(\theta^\eps_s) (R_s^{r+\eps}-R_s^{r})ds,
%\]
%for some $\theta^\eps_s \in (R_s^{r},R_s^{r+\varepsilon})$.

%Still using that $r\mapsto R^r$ is nondecreasing, we can estimate the difference $\Delta_\varepsilon R_s :=R^{r+\eps}_s-R^r_s$ as
%\begin{align}\label{DeltaR}
%\Delta_\varepsilon R_s &= \left(\sqrt{R^{r+\eps}_s}- \sqrt{R^r_s}\right)\left(\sqrt{R^{r+\eps}_s}+\sqrt{R^r_s}\right)\le 2 \sqrt{R_s^{r+\eps_0}} \left(\sqrt{R^{r+\eps}_s}- \sqrt{R^r_s}\right).
%\end{align}
Next, we notice that by Tanaka formula and Yamada-Watanabe's theorem, the process $A:=\sqrt{R}$ is the unique solution to
\[
 d A_t = \left[\left(\frac{k \theta}{2}- \frac{\gamma^2}{8}\right) \frac{1}{A_t} - \frac{k}{2}A_t\right]dt + \frac{\gamma}{2} d W_t,\quad A_0=\sqrt{R_0}.
\]
We then have
\begin{align*}
 d (A_t \, e^{\frac{k}{2}t}) &= e^{\frac{k}{2} t} \, dA_t + A_t  \,\frac{k}{2}e^{\frac{k}{2}t} \, dt\\
 &= e^{\frac{k}{2} t} \left[\left(\frac{k \theta}{2}- \frac{\gamma^2}{8}\right) \frac{1}{A_t} - \frac{k}{2}A_t\right]dt + e^{\frac{k}{2} t}\frac{\gamma}{2} d W_t + A_t  \,\frac{k}{2}e^{\frac{k}{2}t} \, dt\\
 &= e^{\frac{k}{2} t} \left[\left(\frac{k \theta}{2}- \frac{\gamma^2}{8}\right) \frac{1}{A_t}\right]dt + e^{\frac{k}{2} t}\frac{\gamma}{2} d W_t,
 \end{align*}
 which gives in the integral form
 \begin{align*}
A_s \, e^{\frac{k}{2}s}&= A_0 + \left(\frac{k \theta}{2}- \frac{\gamma^2}{8}\right) \int_0^s e^{\frac{k}{2} t} \frac{1}{A_t}dt + \int_0^s e^{\frac{k}{2} t}\frac{\gamma}{2} d W_t.
 \end{align*}
 Hence, using the above formula, we obtain
 \begin{align*}
&\left(\sqrt{R^{r+\eps}_s}- \sqrt{R^r_s}\right)\, e^{\frac{k}{2}s}\notag\\
&=\sqrt{r+\eps}- \sqrt{r} - \left(\frac{k \theta}{2}- \frac{\gamma^2}{8}\right) \int_0^s e^{\frac{k}{2} t} \frac{\sqrt{R^{r+\eps}_t} - \sqrt{R^r_t}}{\sqrt{R^{r+\eps}_t}\sqrt{R^r_t}} dt\\
&\leq \sqrt{r+\eps}- \sqrt{r},
\end{align*}
%Hence, using the integral form of the SDE for $A$ we obtain
%\begin{align*}
%&\left(\sqrt{R^{r+\eps}_s}- \sqrt{R^r_s}\right)\notag\\
%&=\sqrt{r+\eps}- \sqrt{r} - \left(\frac{k \theta}{2}- \frac{\gamma^2}{8}\right) \int_0^s\frac{\sqrt{R^{r+\eps}_u} - \sqrt{R^r_u}}{\sqrt{R^{r+\eps}_u}\sqrt{R^r_u}} du- \frac{k}{2}\int_0^s \left(\sqrt{R^{r+\eps}_u}-\sqrt{R^{r}_u}\right) \,du\\
%&\leq \sqrt{r+\eps}- \sqrt{r}- \frac{k}{2}\int_0^s \left(\sqrt{R^{r+\eps}_u}-\sqrt{R^{r}_u}\right) \,du,
%\end{align*}
where the inequality follows from $R^r \leq R^{r+\eps}$, upon recalling that $2k \theta \ge \gamma^2$. Therefore,
%Gronwall's lemma implies
\begin{align}\label{DeltaS}
\left(\sqrt{R^{r+\eps}_s}- \sqrt{R^r_s}\right)&\leq (\sqrt{r+\eps}- \sqrt{r}) e^{-\frac{k}{2} s}
\end{align}

%and plugging \eqref{DeltaS} into \eqref{DeltaR}, we find
%\begin{align*}
%	\Delta_\varepsilon R_s
%	&\leq  2 \sqrt{R_s^{r+\eps_0}} \,(\sqrt{r+\eps}- \sqrt{r})\, e^{-\frac{k}{2} s}.
%\end{align*}
%Here we carry on the argument in the case $r\mapsto\rho(r)$ is concave (see Assumption \ref{ass:rho}-(ii)) but the case of convex discount rate is analogous up to obvious changes. Then we have $\rho'(\theta^\eps_s)\le \rho'(R^r_s)$ (alternatively, by convexity we have $\rho'(\theta^\eps_s)\le \rho'(R^{r+\eps_0}_s)$), we obtain
{\color{black}{Hence, substituting \eqref{DeltaS} in the last integral of \eqref{ineqU1} and recalling $\varepsilon \leq \varepsilon_0$ we get
\begin{align*}
&\int_0^{\tau_*}
c_3(1+(R^{r+\eps}_s)^q)(\sqrt{R^{r+\eps}_s}-\sqrt{R^r_s})ds
\\
&
\le \left(\sqrt{r+\eps}-\sqrt{r}\right)\int_0^{\tau_*} e^{-\frac{k}{2}s}c_3(1+(R^{r+\eps_0}_s)^q)ds.
\end{align*}
Plugging this expression in \eqref{ineqU1} and using that
\begin{align*}
1-e^{(\sqrt{r+\eps}-\sqrt{r})C}=&-\eps C \int_0^1\frac{1}{2\sqrt{r+\eps u}}e^{(\sqrt{r+\eps u}-\sqrt{r})C}du\\
\ge& -\eps C e^{(\sqrt{r+\eps_0}-\sqrt{r})C} \int_0^1\frac{1}{2\sqrt{r+\eps u}}du\\
=& - C e^{(\sqrt{r+\eps_0}-\sqrt{r})C}(\sqrt{r+\eps }-\sqrt{r}),
\end{align*}
for any $C\ge 0$ independent of $\eps$, we continue with the chain of inequalities
\begin{align*}%\label{ineqU1b}
0 &\geq \frac{U(r+\eps,z)-U(r, z)}{\varepsilon} \\
&\ge-\frac{(\sqrt{r+\eps }-\sqrt{r})}{\eps}\\
&\quad\cdot\E\bigg[e^{\lambda((z-\alpha)\vee S_{\tau_*}- (z-\alpha)) -\int_0^{\tau_*} \rho(R_s^{r})ds}\int_0^{\tau_*}\! e^{-\frac{k}{2}s}
c_3(1+(R^{r+\eps_0}_s)^q)
ds\notag\\
&\qquad\qquad\qquad\cdot \exp\left(\left(\sqrt{r+\eps_0}-\sqrt{r}\right)\int_0^{\tau_*} e^{-\frac{k}{2}s}c_3\left(1+(R^{r+\eps_0}_s)^q\right)
ds\right)\bigg].\notag
\end{align*}
Now we let $\eps\to 0$ first, and then we also let $\eps_0\to0$. Thanks to monotone convergence we obtain
%Next, using the above expression in \eqref{ineqU1}, we obtain an explicit dependence on $\eps$ inside the expectation for the lower bound (up to recalling that $\theta_s\in(R^r_s,R^{r+\eps}_s)$, which causes no concern).
%This allows to compute the limit as $\eps\to 0$ and find
%\begin{align}\label{ineqU2}
%0 &\geq U_r(r,z) \notag\\
%&\ge -\frac{1}{\sqrt{r}}\E\left[e^{\lambda((z-\alpha)\vee S_{\tau_*}- (z-\alpha)) -\int_0^{\tau_*} \rho(R_s^{r})ds} %\int_0^{\tau_*}\!\! \rho'(R^r_s)  \sqrt{R_s^{r+\eps_0}} \, e^{-\frac{k}{2} s} ds\right].
%\end{align}
%Finally, we let $\eps_0\to 0$ and use monotone convergence to obtain
\begin{align}\label{ineqU3}
0 &\geq U_r(r,z) \\
&\ge -\frac{1}{2\sqrt{r}}\E\left[e^{\lambda((z-\alpha)\vee S_{\tau_*}- (z-\alpha)) -\int_0^{\tau_*} \rho(R_s^{r})ds} \int_0^{\tau_*}\!\! %\rho'(R^r_s)
c_3(1+(R^{r}_s)^q)
  e^{-\frac{k}{2} s} ds\right].\notag
\end{align}
}}
We notice that the right-hand side above is bounded by a constant, thanks to Lemma \ref{lem:bound}.

Now, fix $(r_0,z_0)\in\partial_{\overline{\mathcal{O}}} \cC$ and take a sequence ${\rm{Int}}\,\cC \ni {\color{black}{(r_n,z_n)\to(r_0,z_0)}}$, as $n\to \infty$. Using \eqref{ineqU3} with $(r_n,z_n)$ in place of $(r,z)$, recalling that $\tau_*(r_n,z_n)\to 0$ by Lemma \ref{lem:cont-taustar}, and using dominated convergence (justified by the second claim of Lemma \ref{lem:bound}), we get
\[
0 \geq\limsup_{n\to \infty} U_r(r_n,z_n)\ge \liminf_{n\to \infty} U_r(r_n,z_n)\ge 0.
\]
Since the boundary point was arbitrary we conclude that $U_r$ is continuous across $\partial_{\overline{\mathcal{O}}} \cC$.
\end{proof}

An immediate consequence of the above proposition is the following.
\begin{corollary}\label{cor:elb}
For all $r\in\R_+$, we have
\begin{align}
U_z(r,\alpha+)=-\lambda \,U(r,\alpha).
\end{align}
\end{corollary}
\begin{proof}
Fix $r\ge 0$ and let $z_n\downarrow \alpha$ as $n\to\infty$. Then, if
\begin{equation}\label{tauconv}
	\tau_*^n:=\tau_*(r,z_n)\to \tau_*^\alpha=\tau_*(r,\alpha) \,\,\,\textup{as} \,\,n\to \infty,\quad  \P\textup{-a.s.},
\end{equation}
it suffices to take limits in \eqref{eq:Uz}. {Indeed}, by dominated convergence (recall \eqref{eq:sup}) we obtain
\[
U_z(r,\alpha+)=-\lambda\E\left[e^{\lambda S_{\tau^\alpha_*}-\int_0^{\tau_*^\alpha}\rho(R^r_t)dt}\right]=-\lambda \,U(r,\alpha),
\]
where, in order to remove the indicator function in the limit of \eqref{eq:Uz}, we have also used that $\P(S_{\tau_*^\alpha}>0)=1$, being $\P(\tau_*^\alpha>0)=1$ since $b(r)>\alpha$ by Lemma \ref{lem:bnew}.
%\gre{FG: I do not understand why $\P(\tau_*^\alpha>0)=1$}.
So it only remains to prove convergence of the stopping times in \eqref{tauconv}.

The sequence $(K^{z_n})_{n\ge 1}$ is decreasing and therefore the sequence of stopping times $(\tau_*^n)_{n\ge 1}$ is increasing with $\tau^n_*\le \tau_*^\alpha$ for all $n\ge 1$. Hence, $\tau^n_*\uparrow \tau^\infty\le \tau^\alpha_*$, $\P$-a.s., for some stopping time $\tau^\infty$. Now we show that $\tau^\infty=\tau^\alpha_*$ as needed, using an argument similar to those used in \cite[Lem.~4.17]{Ch-DeA} and \cite[Lem.~1.2]{Men} but under different conditions.

Recall that $(t,r,z)\mapsto (R^r_t(\omega),K^z_t(\omega))$ is continuous for all $\omega\in\Omega\setminus N$ and some universal
%\gre{FG: what is the meaning of universal here?}
null set $N$ by Kolmogorov-Chentsov continuity theorem. Fix $\omega\in\Omega\setminus N$.
%If $\tau^\alpha_*(\omega)=0$ the claim is trivial \gre{FG: this should not happen according to the fact, which I do not understand, that $\P(\tau_*^\alpha>0)=1$}.
Let $\delta>0$ be such that $\tau^\alpha_*(\omega)>\delta$, then by continuity of paths there exists $c_\delta>0$ such that
\[
\inf_{0\le t\le \delta}\Big(U(R^r_t(\omega),K^\alpha_t(\omega))-1\Big)\ge c_\delta.
\]
{\color{black}{Thanks to the explicit dynamics of $(K_t)_{t\ge 0}$ in \eqref{K} we find $K^{z_n}_t - K^{\alpha}_t = (z_n-\alpha - S_t)^+ \leq (z_n-\alpha)$. {The latter and} \eqref{lip-z} give
\begin{align*}
&\sup_{0\le t\le \delta}\Big|U(R^r_t(\omega),K^\alpha_t(\omega))
-U(R^r_t(\omega),K^{z_n}_t(\omega))\Big|\\
&\le
h_0
\sup_{0\le t\le \delta} \left(e^{-\lambda(K^\alpha_t(\omega)-\alpha)}
-e^{-\lambda(K^{z_n}_t(\omega)-\alpha)}
\right)\\
&\le
 \lambda h_0 \sup_{0\le t\le \delta}(K^{z_n}_t(\omega)-K^\alpha_t(\omega))
\le
 \lambda h_0 (z_n-\alpha).
\end{align*}
}}
Then there is $n_{\delta,\omega}\ge 1$ such that
\[
\inf_{0\le t\le \delta}\Big(U(R^r_t(\omega),K^{z_n}_t(\omega))-1\Big)\ge \frac{c_\delta}{2}
\]
for all $n\ge n_{\delta,\omega}$. Hence $\lim_{n\to\infty}\tau^n_*(\omega)>\delta$ and, since $\delta$ was arbitrary
\[
\lim_{n\to\infty}\tau^n_*(\omega)\ge \tau^\alpha_*(\omega).
\]
Recalling that $\omega\in\Omega\setminus N$ was also arbitrary, we conclude.
\end{proof}

{We close this section by proving continuity of the optimal boundary (Theorem \ref{thm:bC}), its boundedness and its asymptotic limit as $r\to\infty$ (Proposition \ref{prop:bC}).}
{\color{black}{
%We close this section by providing further results on the optimal boundary.
%In particular, we show the continuity of the free boundary and some estimates of it, see Theorem \ref{thm:bC} and Proposition \ref{prop:bC} below, respectively.
{It is worth noticing that for} the continuity of the boundary, we cannot use \cite[Thm.\ 10]{P19}. The second condition in Eq.\ (3.31) in the statement of that theorem fails in our case as $\frac{d \rho(r)}{dz}=0$.

\begin{theorem}\label{thm:bC}
{Consider the map $b:\R_+ \to [\alpha,+\infty]$} defined in \eqref{eq:b}. Then $r\mapsto b(r)$ is continuous.
\end{theorem}
\begin{proof}
We suitably adapt the proof of \cite[Thm.\ 5.2]{DeA2} {which holds in a parabolic set-up}. We already know that $r \mapsto b(r)$ is nonincreasing and right-continuous by \eqref{eq:b-rc2} in Lemma \ref{lem:bnew}. It thus remains to prove that $r\mapsto b(r)$ is left-continuous. We argue by contradiction.
	
	Assume thus that there exists $r_0>0$ such that $b(r_0-):= \lim_{r \rightarrow r_0}b(r) > b(r_0)$. Then there also exist  $z_1, z_2$ satisfying $b(r_0)< z_1< z_2<b(r_0-)$ and $r_1 < r_0$ such that
	$$
	\Sigma:= (r_1, r_0) \times (z_1, z_2) \subset \cC, \quad \{r_0\} \times (z_1, z_2) \subset {\partial_{\overline{\mathcal{O}}}\cC}.
	$$
Now, by Proposition  \ref{prop:C1}, we know that $U \in C^1(\mathcal O)$ and that \eqref{eq:Uz} holds.
Since {$\P_{r,z}(\tau_*<+\infty,\, S_{\tau_*}>z-\alpha)>0$ for any $(r,z)\in\cC$ and $U_z$ is uniformly continuous} in any compact subset of $\mathcal{C}$, then
formula \eqref{eq:Uz} implies that there exists $\varepsilon_0 >0$ such that
\begin{equation}\label{est_Uz}
	U_z \leq - \varepsilon_0 \quad \textup{on}\,\,\,\,\partial \Sigma \cap \Big\{r \leq \frac{r_1 + r_0}{2}\Big\}=:\partial \Sigma_0.
	\end{equation}
	Moreover, by uniform continuity on any compact set,  for any $\varepsilon >0$, there exists $\delta_\varepsilon >0$ such that  $\delta_\varepsilon \rightarrow 0$ as $\varepsilon \rightarrow 0$ and
	\begin{equation}\label{est_supUrUz}
	\sup_{[r_0 - \delta_\varepsilon, r_0]\times [z_1, z_2]}(|U_r(r,z)| + |U_z(r,z)|)\leq \varepsilon.
	\end{equation}
	In particular,
	\begin{align}\label{est_Uz2}
	U_z(r_0 - \delta_\varepsilon, z) \geq -\varepsilon.
	\end{align}

Let us now set $u:=U_z$. Classical interior regularity results for PDEs (see, e.g., \cite[Thm.\ 10, Ch.\ 3, Sec.\ 5]{Friedman}) guarantee that $u \in C^2(\Sigma) \cap C(\overline \Sigma)$.
 By differentiating the PDE for $U$ given in Corollary \ref{cor:fbp}  and taking into account \eqref{est_Uz}, we get
\begin{equation}\label{PDEu}
\begin{aligned}
&\cL u(r,z)-\rho(r)u(r,z)=0,  &(r,z)\in \Sigma,\\
&u(r,z)=0, & (r,z)\in \partial \Sigma \cap {\partial_{\overline{\mathcal O}} \cC},\\
&u(r,z)\leq 0, &(r,z)\in \overline\Sigma,\\
&u(r,z)\leq -\varepsilon_0, &(r,z)\in \partial \Sigma_0.
%\label{est_u}
\end{aligned}
\end{equation}

On the interval $(r_1, r_0 - \delta_\varepsilon]$ we consider a process that is equal to $(R_t)_{t \geq 0}$ away from $r_0  -\delta_\varepsilon$, it is reflected (downwards) at $r_0 - \delta_\varepsilon$, and it gets absorbed {on the portion of the boundary $\partial\Sigma\setminus\partial_{\overline{\mathcal{O}}}\cC$}. To this end,  we introduce  a  process $\xi^\varepsilon$ with dynamics
\begin{align}\label{eq:xi}
	d \xi^\varepsilon_t = k(\theta - \xi^\varepsilon_t) dt + \gamma \sqrt{\xi^\varepsilon_t} d W_t  - {d A_t^\varepsilon}, \qquad \xi^\varepsilon_0= r_0 - \delta_\varepsilon,
\end{align}
where $A^\varepsilon$ is an increasing and continuous process with $A_0^\varepsilon=0$ such that
\begin{align}\label{eq:A}
\xi^\varepsilon_t \leq r_0 - \delta_\varepsilon \quad \textup{and}\quad d A^\varepsilon_t = 1_{\{\xi^\varepsilon_t = r_0 - \delta_\varepsilon\}} d A^\varepsilon_t \quad \text{for all $t \geq 0$}.
\end{align}
The existence of $\xi^\varepsilon$ follows from standard results on reflecting diffusions, but can also be constructed as a time-change of a scaled reflected Brownian motion, see e.g. \cite{LionsSznitman} or \cite[Sec. 12, Chapter I]{Bass} for more details. Let
\begin{equation}\label{zeta}
\zeta_t := Z_t^0 = z + \mu t + \sigma B_t, \quad \text{for $z \in (z_1, z_2)$},
\end{equation}
and set
\begin{equation}\label{tau_sigma_eps}
\tau_\Sigma^\varepsilon := \inf\{t \geq 0: \,\,(\zeta_t, \xi^\varepsilon_t) {\in \partial\Sigma\setminus\partial_{\overline{\mathcal{O}}}\cC}\}.
\end{equation}
Then, the {process $(\xi^\varepsilon_{t \wedge \tau^\varepsilon_\Sigma},\zeta_{t\wedge \tau^\varepsilon_\Sigma})_{t\geq 0}$ evolves in the rectangle $(r_1,r_0-\delta_\eps]\times(z_1,z_2)$, it is reflected horizontally (inward) at each time $\xi^\eps$ hits $r_0-\delta_\eps$ and it is absorbed upon reaching the portion of boundary $\partial\Sigma\setminus\partial_{\overline{\mathcal{O}}}\cC$. Notice also that $\E[\tau^\eps_\Sigma]<\infty$ since it is dominated by the exit time of $\zeta$ from the bounded interval $[z_1,z_2]$}.
%(FAUSTO: MI PARE CHE QUESTO PROCESSO SIA ASSORBITO SUL BORDO DI $\Sigma$, NON SU $r_1$. COMUNQUE IL RESTO MI PARE FUNZIONI COMUNQUE.)

Let us now apply Dynkin's formula to $e^{-\int_0^\cdot \rho(\xi^\varepsilon_u) du} u(\xi^\varepsilon_\cdot, \zeta_\cdot)$ on the (random) time interval $[0, \tau_\Sigma^\varepsilon]$ {and use the first equation in \eqref{PDEu}}:
\begin{align}\label{Dynkin_u}
\E\Big[e^{-\int_0^{\tau^{\varepsilon}_\Sigma}\rho(\xi^\varepsilon_u)du} u(\xi^{\varepsilon}_{\tau^{\varepsilon}_\Sigma},\zeta_{\tau^{\varepsilon}_\Sigma})\Big] &= u(r_0 - \delta_\varepsilon,z) - \E\Big[\int_0^{\tau^{\varepsilon}_\Sigma}e^{-\int_0^{t}\rho(\xi^\varepsilon_u)du}u_r(\xi^{\varepsilon}_t,\zeta_t)\,dA^\varepsilon_t\Big]\notag\\
 &= u(r_0 - \delta_\varepsilon,z) - \E\Big[\int_0^{\tau^{\varepsilon}_\Sigma}e^{-\int_0^{t}\rho(\xi^\varepsilon_u)du}u_r(r_0 - \delta_\varepsilon,\zeta_t)\,dA^\varepsilon_t\Big],
\end{align}
where in the second equality we used \eqref{eq:A}.
{From the final condition in \eqref{PDEu} and using that $\rho \ge 0$ is bounded on $\overline \Sigma$}, on the left-hand side of \eqref{Dynkin_u} we have
\begin{align}\label{Dynkin_u_left}
\E\Big[e^{-\int_0^{\tau^{\varepsilon}_\Sigma}\rho(\xi^\varepsilon_u)du} u(\xi^{\varepsilon}_{\tau^{\varepsilon}_\Sigma},\zeta_{\tau^{\varepsilon}_\Sigma})\Big] &\leq - \varepsilon_0 \, \E\Big[e^{-\int_0^{\tau^{\varepsilon}_\Sigma}\rho(\xi^\varepsilon_u)du} 1_{\{(\xi^{\varepsilon}_{\tau^{\varepsilon}_\Sigma},\zeta_{\tau^{\varepsilon}_\Sigma})\in \partial \Sigma_0\}}\Big]\\
&\leq - \varepsilon_0\,C\,\P((\xi^{\varepsilon}_{\tau^{\varepsilon}_\Sigma},\zeta_{\tau^{\varepsilon}_\Sigma})\in \partial \Sigma_0),\notag
\end{align}
{for some constant $C>0$ only depending on $\Sigma$}.
{Thanks to \eqref{est_Uz2}, on} the right-hand side of \eqref{Dynkin_u} we get
\begin{align}\label{Dynkin_u_right}
 &u(r_0 - \delta_\varepsilon,z) - \E\Big[\int_0^{\tau^{\varepsilon}_\Sigma}e^{-\int_0^{t}\rho(\xi^\varepsilon_u)du}u_r(r_0 - \delta_\varepsilon,\zeta_t)\,d A^\varepsilon_t\Big]\\
 & \geq - \varepsilon - \E\Big[\int_0^{\tau^{\varepsilon}_\Sigma}e^{-\int_0^{t}\rho(\xi^\varepsilon_u)du}U_{zr}(r_0 - \delta_\varepsilon,\zeta_t)\,d A^\varepsilon_t\Big].\notag
\end{align}
Collecting \eqref{Dynkin_u_left}-\eqref{Dynkin_u_right}, we obtain
\begin{align}\label{ineq_final}
- \varepsilon_0 \,C\, \P((\xi^{\varepsilon}_{\tau^{\varepsilon}_\Sigma},\zeta_{\tau^{\varepsilon}_\Sigma})\in \partial \Sigma_0)\geq - \varepsilon - \E\Big[\int_0^{\tau^{\varepsilon}_\Sigma}e^{-\int_0^{t}\rho(\xi^\varepsilon_u)du}U_{zr}(r_0 - \delta_\varepsilon,\zeta_t)\,dA^\varepsilon_t\Big].
\end{align}

{Next we want to take limits in \eqref{ineq_final} as $\varepsilon \rightarrow 0$. In order to avoid potential difficulties with the continuity of $U_{zr}$ at the boundary $\{r_0\}\times(z_1,z_2)\subset\partial_{\overline{\mathcal O}}\cC$}, we adopt an approach using test functions. Let us take $\varphi \in C^\infty_c((z_1, z_2))$, $\varphi \geq 0$. {Thanks to \eqref{tau_sigma_eps} we can write  $\tau_\Sigma^\varepsilon= \tau_1(z) \wedge \tau_2(z) \wedge \eta^\varepsilon$, where
\begin{align}
\eta^\varepsilon&:=\inf \{t \geq 0:\,\,\xi^\varepsilon_t \leq r_1\},\label{etavarepsilon}\\
\tau_1(z) &:= \inf \{t \geq 0:\,\,\zeta_t \leq z_1\},\qquad\tau_2(z):=\inf \{t \geq 0:\,\,\zeta_t \geq z_2\}.\notag
\end{align}
and notice that $\eta^\eps$ is independent of the initial condition $z$ for the process $\zeta_t$}.
{Multiplying \eqref{ineq_final} by $\varphi(z)$, integrating over $(z_1,z_2)$ and using Fubini's theorem, we get}

%Using Fubini's theorem and recalling \eqref{zeta}, we get {
%\begin{align}\label{ineq_testfunctions}
%&- \varepsilon_0 \,\int_{z_1}^{z_2}\varphi(z)\P((\xi^{\varepsilon}_{\tau^{\varepsilon}_\Sigma},\zeta_{\tau^{\varepsilon}_\Sigma})\in \partial %\Sigma_0)dz\notag\\
%&\geq- \varepsilon - \int_{z_1}^{z_2}\varphi(z)\E\Big[\int_0^{\infty}\!\!1_{\{t<\tau^{\varepsilon}_\Sigma\}}e^{-\int_0^{t}\rho(\xi^\varepsilon_u)du}U_{zr}(r_0 - \delta_\varepsilon,\zeta_t)\,d A^\varepsilon_t\Big]dz\\
%&=- \varepsilon -\E\Big[\int_0^{\infty}\!\!e^{-\int_0^{t}\rho(\xi^\varepsilon_u)du} \Big(\int_{z_1}^{z_2}1_{\{t<\tau^{\varepsilon}_\Sigma\}} %U_{zr}(r_0 - \delta_\varepsilon,z + \mu t + B_t)\,\varphi(z)\,dz\Big)dA^\varepsilon_t\Big],\notag
%\end{align}
%where we need to remember that $\tau^\eps_\Sigma=\tau^\eps_\Sigma(r_0-\delta_\eps,z)$}.
%So  \eqref{ineq_testfunctions} reads
\begin{align}\label{ineq_testfunctions_2}
&- \varepsilon_0 \,C\!\! \int_{z_1}^{z_2}\!\!\varphi(z)\P \Big((\xi^{\varepsilon}_{\tau^{\varepsilon}_\Sigma},\zeta_{\tau^{\varepsilon}_\Sigma})\in \partial \Sigma_0\Big)dz\\
&\geq\! - \varepsilon\! -\E\bigg[\!\int_0^{\eta^{\varepsilon}}\!\!e^{-\int_0^{t}\rho(\xi^\varepsilon_u)du} \bigg(\!\int_{z_1}^{z_2}\!\!1_{\{t < (\tau_1\wedge \tau_2)(z)\}}\,U_{zr}(r_0\! -\! \delta_\varepsilon,z\! +\! \mu t\! +\! B_t)\,\varphi(z)\,dz\!\bigg)dA^\varepsilon_t\bigg].\notag
\end{align}

{The mapping} $z \mapsto \tau_1(z, \omega)$ (resp. $z \mapsto \tau_2(z, \omega)$) is increasing and continuous {for $\P$-a.e.-$\omega$} (resp.\ decreasing, continuous). {Monotonicity is by pathwise comparisons whereas continuity is a known result for one dimensional diffusions (it may also be deduced by arguments analogous to those in Lemma \ref{lem:reg-b}).}
It follows that $z \mapsto \tau_1(z) \wedge \tau_2(z)$ is $\P$-a.s.\ continuous and it changes its monotonicity at most once. In particular, for any $\omega \in \Omega \setminus N$ {with $\P(N)=0$}, there exist $\underline z(t,\omega)$ and $\bar z(t, \omega)$ satisfying $z_1 < \underline z(t,\omega) < \bar z(t, \omega) < z_2$ and such that
\begin{align}\label{barz}
\{z \in (z_1, z_2) :\,\,\tau_1(z, \omega) \wedge \tau_2(z, \omega)>t\}=(\underline z(t,\omega), \bar z(t, \omega)).
\end{align}
Plugging \eqref{barz} into \eqref{ineq_testfunctions_2} we obtain
\begin{align*}
&- \varepsilon_0 \, C\,\int_{z_1}^{z_2}\varphi(z)\P\Big((\xi^{\varepsilon}_{\tau^{\varepsilon}_\Sigma},\zeta_{\tau^{\varepsilon}_\Sigma})\in \partial \Sigma_0\Big)dz\\
&\geq - \varepsilon -\E\Big[\int_0^{\eta^{\varepsilon}}e^{-\int_0^{t}\rho(\xi^\varepsilon_u)du} \Big(\int_{\underline z(t, \omega)}^{\bar z(t, \omega)}\,U_{zr}(r_0 - \delta_\varepsilon,z + \mu t + B_t)\,\varphi(z)\,dz\Big)dA^\varepsilon_t\Big].\notag
\end{align*}
Integrating by parts, recalling \eqref{est_supUrUz} and using that $(\underline z(t,\omega), \bar z(t, \omega)) \subset (z_1, z_2)$, we get
\begin{align*}
&\Big|\int_{\underline z(t, \omega)}^{\bar z(t, \omega)}\,U_{zr}(r_0 - \delta_\varepsilon,z + \mu t + B_t)\,\varphi(z)\,dz \Big|\\
&=\Big |U_{r}(r_0 - \delta_\varepsilon,z + \mu t + B_t)\,\varphi(z)\Big|_{\underline z(t, \omega)}^{\bar z(t, \omega)}-\int_{\underline z(t, \omega)}^{\bar z(t, \omega)}\,U_{r}(r_0 - \delta_\varepsilon,z + \mu t + B_t)\,\varphi'(z)\,dz\Big|\\
&\leq \varepsilon (z_2 - z_1)(||\varphi||_{\infty} + ||\varphi'||_{\infty}),
\end{align*}
where $||\cdot||_{\infty}$ denotes the supremum norm on $(z_1, z_2)$. {Since this bound is deterministic an independent of $z$, when we plug it back into the integral with respect to $dA^\varepsilon_t$ we obtain}
\begin{align*}
&- \varepsilon_0 \, C\,\int_{z_1}^{z_2}\varphi(z)\P\Big((\xi^{\varepsilon}_{\tau^{\varepsilon}_\Sigma},\zeta_{\tau^{\varepsilon}_\Sigma})\in \partial \Sigma_0\Big)dz\geq - \varepsilon\Big(1 + (z_2 - z_1)(||\varphi||_{\infty} + ||\varphi'||_{\infty})\E[A_{\eta^\varepsilon}^\varepsilon]\Big).\notag
\end{align*}
From the integral form of $\xi^\varepsilon$ we obtain
\begin{align*}
\E[A_{\eta^\varepsilon}^\varepsilon]  &= \E\Big[r_0 - \delta_\varepsilon  +\int_0^{\eta^\varepsilon} k(\theta - \xi^\varepsilon_s) ds % + \int_0^{\eta^\varepsilon}\gamma \sqrt{\xi^\varepsilon_s} d W_s
 \Big]- \E[\xi^\varepsilon_{\eta^\varepsilon}]\leq r_0 + k\theta\E[\eta^\varepsilon] - r_1.
\end{align*}
{By construction, $\xi^\eps\le R^{r_0}$ for any $\eps>0$.
%and with $\xi^0$ the unique solution of \eqref{eq:xi} with $\delta_\eps\equiv 0$.
Then $\E[\eta^\eps]\le \E[\eta]$, where $\eta:=\inf \{t \geq 0:\,\, R^{r_0}_t\leq r_1\}$,
%is the analogue of \eqref{etavarepsilon} for the process $\xi^0$
and the expectation of the latter hitting time is finite since the CIR process is positively recurrent {\color{black}{(cf.\ Section 12 in Chapter II of \cite{BorSal})}}. Then we have $\E[A^\eps_{\eta^\eps}]\le C_1$, for a constant $C_1>0$ independent of $\eps$.}

Finally, we get
\begin{align*}
\varepsilon_0 \, C\,\int_{z_1}^{z_2}\varphi(z)\P\Big((\xi^{\varepsilon}_{\tau^{\varepsilon}_\Sigma},\zeta_{\tau^{\varepsilon}_\Sigma})\in \partial \Sigma_0\Big)dz\leq  \varepsilon \Big(1+ C_1(z_2 - z_1)(||\varphi||_{\infty} + ||\varphi'||_{\infty})\Big).
\end{align*}
Then, taking limits as $\varepsilon$ goes to zero, the previous inequality yields
\begin{align}\label{ineq_testfunctions_lim}
 \limsup_{\varepsilon \rightarrow 0} \int_{z_1}^{z_2}\varphi(z)\P\Big((\xi^{\varepsilon}_{\tau^{\varepsilon}_\Sigma},\zeta_{\tau^{\varepsilon}_\Sigma})\in \partial \Sigma_0\Big)dz\leq  0.
\end{align}
{We now show that the above inequality leads to a contradiction. N}otice that
$$
\P\Big((\xi^{\varepsilon}_{\tau^{\varepsilon}_\Sigma},\zeta_{\tau^{\varepsilon}_\Sigma})\in \partial \Sigma_0\Big) \geq \P\Big(\eta^\varepsilon < (\tau_1 \wedge \tau_2)(z)\Big)
$$
{and, letting $\mu(dt;z)$ denote the (well-known) law of $(\tau_1\wedge \tau_2)(z)$, we have
$$
\P\Big(\eta^\varepsilon < (\tau_1 \wedge \tau_2)(z)\Big) = \int_0^\infty \P(\eta^\varepsilon < t)\, \mu(dt;z),
$$
by independence of $\eta^\eps$ and $\tau_1\wedge\tau_2$.}

{Since $\xi^\eps\le R^{r_0}$ by pathwise comparison, then
$$
\P(\eta^\varepsilon < t)\ge\P(\eta < t),\qquad\text{for all $\eps>0$}.
$$
with $\eta$ introduced above.
Therefore we have}
$$
\liminf_{\varepsilon \rightarrow 0} \P\Big(\eta^\varepsilon < (\tau_1 \wedge \tau_2)(z)\Big)\color{black}\geq \int_0^\infty \P(\eta < t)\, \mu(dt;z):=f(z) >0,\quad \forall z \in (z_1, z_2).
$$
Then, from \eqref{ineq_testfunctions_lim} and Fatou's lemma we reach a contradiction. {Thus $r\mapsto b(r)$ is continuous.}
\end{proof}
}}

\begin{proposition}
\label{prop:bC}
One has:
\begin{itemize}
\item[(i)] $b(r) < + \infty$ for all $r>0$;
\item[(ii)] {\color{black}{if $\rho(r) \geq c_1$ for some $c_1 >0$, then $b(r) \leq z^{\star}_{c_1}$ for all $r\geq 0$, where $z^{\star}_{c_1} \in (\alpha, \infty)$ is the free boundary of the optimal stopping problem \eqref{eq:U2} with $\rho(r)\equiv c_1$;}}
\item[(iii)] if $\rho(r) \geq c_2 r$ for some $c_2 >0$, then $\lim_{r\uparrow \infty}b(r)=\alpha$.
%Moreover, denoting by $\rho_o:=\lim_{r \downarrow 0}\rho(r)$ we have the following:
%\begin{enumerate}
%\item if $\rho_o >0$ then there exists $z^{\star}_{\rho_o}\in (\alpha, \infty)$ such that $b(r) \leq z^{\star}_{\rho_o}$ for all $r\geq 0$;
%\item \textcolor{red}{if $\rho_o =0$ then $\lim_{r\downarrow 0}b(r)=....$.}
%\end{enumerate}
\end{itemize}
\end{proposition}
\begin{proof}
We prove each item separately.
\vspace{0.25cm}

(i) Suppose that there exists $r_0>0$ such that $b(r_0)= + \infty$. Then, by monotonicity, $b(r) = + \infty$ for all $r\in [0,r_0)$. Then take $r \in [0,r_0)$ and set $\widehat{\tau}:=\inf\{t\geq0: R_t ^r \geq r_0\}$, $\P$-a.s. Clearly, $\widehat{\tau} \leq \tau_*$ $\P_{r,z}$-a.s.\ for all $z\geq \alpha$, and therefore the superharmonic property property of the value $U$ (cf.\ \eqref{eq:supm} and \eqref{eq:m}) implies that
\begin{align}
\label{limit-a}
1 <& U(r,z) = \E\left[ e^{\lambda\,\left((z-\alpha) \vee S_{\widehat{\tau}}-(z-\alpha)\right) - \int_0^{\widehat{\tau}} \rho(R^r_s) \,ds} U(R^r_{\widehat{\tau}},K^z_{\widehat{\tau}})\right] \\
\leq& \E\left[ \mathds{1}_{\{S_{\widehat{\tau}} \geq z-\alpha\}}e^{\lambda\,\left(S_{\widehat{\tau}}-(z-\alpha)\right)-\int_0^{\hat \tau}\rho(R^r_s)ds} h_0\right]\notag\\
& + \E\left[ \mathds{1}_{\{S_{\widehat{\tau}} < z-\alpha\}}e^{-\int_0^{\widehat{\tau}} \rho(R^r_s) \,ds} U(R^r_{\widehat{\tau}},K^z_{\widehat{\tau}})\right]\notag\\
\leq & e^{-\lambda\,(z-\alpha)} \E\left[e^{\lambda\,S_{\widehat{\tau}}-\int_0^{\hat \tau}\rho(R^r_s)ds} h_0\right] + \E\left[e^{-\int_0^{\widehat{\tau}} \rho(R^r_s) \,ds} U(R^r_{\widehat{\tau}},K^z_{\widehat{\tau}})\right].\notag
\end{align}
By noticing that $\widehat{\tau}$ does not depend on $z$, recalling \eqref{eq:sup}, and taking limits as $z \uparrow \infty$
we obtain
$$
\lim_{z \rightarrow \infty} e^{-\lambda\,(z-\alpha)} \E\left[e^{\lambda\,S_{\widehat{\tau}} -\int_0^{\hat \tau}\rho(R^r_s)ds} h_0\right]=0.
$$
On the other hand, for any $r \in  [0,\,r_0]$ we have
\begin{align*}
1 <
U(r,z)  &= \sup_{\tau \geq 0} \E\left[ e^{\lambda\,\left((z-\alpha) \vee S_\tau-(z-\alpha)\right) - \int_0^\tau \rho(R^r_s) \,ds}\right]\\
& \leq \sup_{\tau \geq 0} \E\left[ \mathds{1}_{\{S_{\tau} \geq z-\alpha\}} \, e^{\lambda\,\left(S_\tau-(z-\alpha)\right) - \int_0^\tau \rho(R^r_s) \,ds}\right]+ \sup_{\tau \geq 0} \E\left[ \mathds{1}_{\{S_{\tau} < z-\alpha\}} \, \right]\\
& \leq e^{-\lambda(z-\alpha)} \sup_{\tau \geq 0} \E\left[ e^{\lambda\,S_\tau - \int_0^\tau \rho(R^r_s) \,ds}\right] + 1\\
& \leq h_0 \, e^{-\lambda(z-\alpha)}  + 1.
\end{align*}
%Since $U(r,z) > 1$ for any $r \in  [0,\,r_o]$,
It follows that $\lim_{z \rightarrow +\infty} U(r,z) =1$ for any $r \in  [0,\,r_0]$.
% for any $r \in [0,\,r_0)$.
%This implies that  %for any $r \in [0,\,r_0)$,
%$$
%\lim_{z \rightarrow +\infty} U(r,k(z)) =1 \,\,\textup{for any}\,\,k: [\alpha, \infty) \rightarrow [0,\,\infty)\,\,\textup{such that}\,\lim_{z \rightarrow \infty} k(z)= +\infty
%$$
%for any $r \in  [0,\,r_o]$.
%Being $z \mapsto U(r,z)$ non-increasing, and r
Recalling that $\lim_{z \rightarrow +\infty} K^z_t = +\infty$ a.s., and noticing that the CIR process is positively recurrent,   this in turn yields
$$
\lim_{z \rightarrow \infty} U(R^r_{\widehat{\tau}},K^z_{\widehat{\tau}}) =1 \,\,\textup{a.s.}
$$
Thus, applying the Lebesgue dominated convergence theorem in \eqref{limit-a}, we get
$$
1 \leq  \E\left[e^{-\int_0^{\widehat{\tau}} \rho(R^r_s) \,ds} \right].
$$
Being
 $\P_r(\widehat{\tau}>0)>0$ for any $r \in [0,r_0)$,
we reach a contradiction.

%(i) Suppose that there exists $r_o>0$ such that $b(r_o)= + \infty$. Then, by monotonicity, $b(r) = + \infty$ for all $r\in [0,r_o]$. Then take $r \in [0,r_o)$ and set $\widehat{\tau}:=\inf\{t\geq0: R_t \geq r_o\}$, $\P_r$-a.s. Clearly, $\widehat{\tau} \leq \tau_*$ $\P_{r,z}$-a.s.\ for all $z\geq \alpha$, and therefore the superharmonic property property of the value $U$ (cf.\ \eqref{eq:supm} and \eqref{eq:m}) implies that
%\begin{align}
%\label{limit-a}
%& 1 < U(r,z) = \E\left[ e^{\lambda\,\left((z-\alpha) \vee S_{\widehat{\tau}}-(z-\alpha)\right) - \int_0^{\widehat{\tau}} \rho(R^r_s) \,ds}\right] \nonumber \\
%& \leq \E\left[ \mathds{1}_{\{S_{\widehat{\tau}} \geq z-\alpha\}}e^{\lambda\,\left(S_{\widehat{\tau}}-(z-\alpha)\right)}\right] + \E\left[ \mathds{1}_{\{S_{\widehat{\tau}} < z-\alpha\}}e^{-\int_0^{\widehat{\tau}} \rho(R^r_s) \,ds}\right].
%\end{align}
%By noticing that $\widehat{\tau}$ does not depend on $z$, and taking limits as $z \uparrow \infty$ we obtain
%$1 \leq \E\left[e^{-\int_0^{\widehat{\tau}} \rho(R^r_s) \,ds}\right],$
%and we reach a contradiction since $\P_r(\widehat{\tau}>0)>0$ for any $r \in [0,r_o)$.
\vspace{0.25cm}

(ii) Assume that $\rho(r) \geq c_1$ for some $c_1 >0$. Because
$$U(r,z) \leq \sup_{\tau\geq0} \E\left[ e^{\lambda\,\left((z-\alpha) \vee S_{\widehat{\tau}}-(z-\alpha)\right) - c_1 \tau}\right]=:v(z;c_1),$$
one has for any $r\geq0$ that
$$\{z > \alpha:\, z \geq b(r)\} = \{z> \alpha:\, U(r,z) = 1\} \supseteq \{z> \alpha:\, v(z;c_1) = 1\}.$$
Notice now that $v(z;c_1) \le e^{\lambda(z-\alpha)}\overline v$ for some constant $\overline v>0$ for all $z\geq0$ (cf.\ \eqref{eq:new}), and that $\{z> \alpha:\, v(z;c_1) = 1\} = \{z> \alpha:\, z \geq z^{\star}_{c_1}\}$ for some $z^{\star}_{c_1} \in (\alpha,\infty)$. Hence we conclude that $b(r) \leq z^{\star}_{c_1}$.
\vspace{0.25cm}

(iii) Assume that $\rho(r) \geq c_2 r$ for some $c_2 >0$. To prove that $\lim_{r\uparrow \infty}b(r)=\alpha$ we argue by contradiction and we suppose that $b_{\infty}:=\lim_{r\uparrow \infty}b(r)>\alpha$. Then take $z_1,z_2$ such that $\alpha < z_1<z_2 < b_{\infty}$ and for $z \in (z_1,z_2)$ and $r \geq 0$ set $\widehat{\sigma}:=\inf\{t\geq 0: K_t \notin (z_1,z_2)\}$ $\P_z$-a.s.
Clearly, $\R_+ \times (z_1,z_2) \subset \mathcal{C}$, and therefore $\widehat{\sigma} \leq \tau_*$ $\P_{r,z}$-a.s., and this fact implies that (see \eqref{eq:m})
\begin{align}
\label{limit3}
 1 &< U(r,z) = \E\left[ e^{\lambda\,\left((z-\alpha) \vee S_{\widehat{\sigma}}-(z-\alpha)\right) - \int_0^{\widehat{\sigma}} \rho(R^r_s) \,ds} U(R^r_{\widehat{\sigma}}, K^z_{\widehat{\sigma}})\right] \nonumber \\
& \leq h_0 \E\left[ e^{\lambda\,\left((z-\alpha) \vee S_{\widehat{\sigma}}-(z-\alpha)\right) - c_2 \int_0^{\widehat{\sigma}} R^r_s \,ds}\right] \\
& = h_0 \E\left[ e^{\lambda S_{\widehat{\sigma}} - A_{c_2}(\widehat{\sigma}) - r G_{c_2}(\widehat{\sigma})}\right]. \nonumber
\end{align}
Here, \eqref{eq:U-bound} has been used for the penultimate step, while the independence of the Brownian motions $W$ and $B$ led to the last equality, together \eqref{Laplace} and \eqref{GA}.
Since the last expectation on the right-hand side of \eqref{limit3}
%{\color{black}(TDA: wrong reference?)}
 can be made arbitrarily small by taking $r$ sufficiently large, we reach a contradiction and we have thus proved that $\lim_{r\uparrow \infty}b(r)=\alpha$.
%\vspace{0.1cm}
%To prove (v)-(1) notice that by monotonicity $\rho(r)\geq \rho_o \vee c_2 r$. The claim then follows by the same arguments as in the proof of (iii).\
%\vspace{0.1cm}
%
%\textcolor{red}{As for (v)-(2) .....TO BE DONE}.
\end{proof}

%%%%%%%%%%%%%%%%%%%%%%%%%%%%%%%%%%%%%%%%%%%%%%%%%%%%%%%

\section{Solution to the dividend problem}
\label{sec:solution}

In this section we show that we can find a couple $(v,a)$ that satisfies all the assumptions in Theorem \ref{thm:verif}, hence we obtain a full solution to problem \eqref{eq:Vbis}.

Let us define the function $v:\overline{\mathcal{O}}\rightarrow \R_+$ as follows
\begin{align}\label{eq:v}
v(r,z):=\int_\alpha^z U(r,y)dy.
\end{align}
Using Proposition \ref{prop:C1} we obtain that the functions $v_z$, $v_{zz}$, $v_r$ and $v_{zr}$ are continuous on $\mathcal{O}$.

\begin{proposition}
\label{prop:vrr}
The function $v$ has a weak derivative $v_{rr} \in L^{\infty}_{loc}(\mathcal{O})$. Moreover, we can select an element of the equivalence class of $v_{rr} \in L^{\infty}_{loc}(\mathcal{O})$ (still denoted by $v_{rr}$) such that
\begin{align}
\label{eq:vrr}
&v_{rr}(r,z)\\
&=\mathds{1}_{\{b_-(r)\ge \alpha\}}\frac{2}{\gamma^2}\left(\int_\alpha^{b_-(r)\wedge z}\Big[\rho(r)U(r,y)-\mu U_z(r,y)-k(\theta-r)U_r(r,y)\Big]dy\right)r^{-1}\notag\\
&\hspace{1.5cm} -\mathds{1}_{\{b_-(r)\ge \alpha\}}\frac{\sigma^2}{\gamma^2}\left(U_{z}(r,z\wedge b_-(r))-U_z(r,\alpha+)\right)r^{-1},\notag
\end{align}
where $b_-(\cdot):=\lim_{\varepsilon \downarrow 0}b(\cdot-\varepsilon)$.
\end{proposition}
\begin{proof}
The main idea in this proof is to compute explicitly the weak derivative $v_{rr}$.

Since $v_r(\cdot,z)$ is a continuous function for all $z > \alpha$, we say that its weak derivative with respect to $r$ is a function $f\in L^1(\mathcal{O})$ such that, for any $\varphi\ge 0$ with $\varphi\in C^\infty_c(\R_+)$, it holds
\[
\int_0^\infty v_r(\eta,z)\varphi'(\eta)d\eta=-\int_0^\infty f(\eta,z)\varphi(\eta)d\eta,\quad \text{for $z\in(\alpha,+\infty)$.}
\]
We denote by $g$ the generalised, right-continuous, inverse of the decreasing function $b$ and, for future frequent use, we also define $g_\eps(\,\cdot\,):=g(\,\cdot\,)-\eps$ for $\eps>0$.

Using that $U_r$ is continuous, with $U_r(\eta,y)=0$ for $\eta\ge g(y)$, and employing Fubini's theorem we can write
\begin{align}\label{eq:c2-1}
\int_0^\infty& v_r(\eta,z)\varphi'(\eta)d\eta\\[+3pt]
=&\int_0^\infty \left(\int_\alpha^z U_r(\eta,y)dy\right)\varphi'(\eta)d\eta=\int_\alpha^z\left(\int_0^\infty U_r(\eta,y)\varphi'(\eta)d\eta\right)dy\notag\\[+3pt]
=&\int_\alpha^z\left(\int_0^{g(y)} U_r(\eta,y)\varphi'(\eta)d\eta\right)dy=\int_\alpha^z\left(\lim_{\eps\to0 }\int_0^{g_\eps(y)} U_r(\eta,y)\varphi'(\eta)d\eta\right)dy\notag\\[+3pt]
=&\lim_{\eps\to0 }\int_\alpha^z\left(\int_0^{g_\eps(y)} U_r(\eta,y)\varphi'(\eta)d\eta\right)dy,\notag
\end{align}
where in the last line we used dominated convergence. We now recall that
\begin{align*}
\tfrac{\gamma^2}{2}r U_{rr}=\rho(r)U-\tfrac{\sigma^2}{2}U_{zz}-\mu U_z-k(\theta-r)U_r
\end{align*}
in $\cC$ and that $U_{rr}$ is continuous away from $\partial\cC$. This implies that for fixed $\eps>0$ we can write (recalling that $\varphi(0)=0$)
\begin{align*}
&\hspace{-10pt}\int_0^{g_\eps(y)} U_r(\eta,y)\varphi'(\eta)d\eta=U_r(g_\eps(y),y)\varphi(g_\eps(y))-\int_0^{g_\eps(y)} U_{rr}(\eta,y)\varphi(\eta)d\eta\\[+3pt]
=&\,U_r(g_\eps(y),y)\varphi(g_\eps(y))\\[+3pt]
&-\frac{2}{\gamma^2}\left(\int_0^{g_\eps(y)}\!\!\eta^{-1}\Big[\rho(\eta)U(\eta,y)\!-\!\tfrac{\sigma^2}{2}U_{zz}(\eta,y)\!-\!\mu U_z(\eta,y)\!-\!k(\theta\!-\!\eta)U_r(\eta,y)\Big]\varphi(\eta)d\eta\right).
\end{align*}
Plugging the latter in \eqref{eq:c2-1} we find
\begin{align}\label{splittedlimit}
&\hspace{-10pt}\lim_{\eps\to0 }\int_\alpha^z\left(\int_0^{g_\eps(y)} U_r(\eta,y)\varphi'(\eta)d\eta\right)dy\\
=&\lim_{\eps\to0 }\int_\alpha^z U_r(g_\eps(y),y)\varphi(g_\eps(y))dy\notag\\
&-\!\lim_{\eps\to0 }\!\int_\alpha^z\!\!\frac{2}{\gamma^2}\!\left(\!\int_0^{g_\eps(y)}\!\eta^{-1}\Big[\rho(\eta)U(\eta,y)\!-\!\mu U_z(\eta,y)\!-\!k(\theta\!-\!\eta)U_r(\eta,y)\Big]\varphi(\eta)d\eta\!\right)\!dy\notag\\
&+\lim_{\eps\to0 }\int_\alpha^z\frac{\sigma^2}{\gamma^2}\left(\int_0^{g_\eps(y)}\eta^{-1}U_{zz}(\eta,y)\varphi(\eta)d\eta\right)dy.\notag
\end{align}
For the first two limits on the right-hand side of \eqref{splittedlimit} we can use dominated convergence and recall that $U_r(g(y),y)=0$ to get
\begin{align}\label{eq:c2-2}
&\hspace{-10pt}\int_0^\infty v_r(\eta,z)\varphi'(\eta)d\eta\\
=&-\int_\alpha^z\frac{2}{\gamma^2}\!\left(\!\int_0^{g(y)}\!\eta^{-1}\Big[\rho(\eta)U(\eta,y)\!-\!\mu U_z(\eta,y)\!-\!k(\theta\!-\!\eta)U_r(\eta,y)\Big]\varphi(\eta)d\eta\right)dy\notag\\
&+\lim_{\eps\to0 }\int_\alpha^z\frac{\sigma^2}{\gamma^2}\left(\int_0^{g_\eps(y)}\eta^{-1}U_{zz}(\eta,y)\varphi(\eta)d\eta\right)dy.\notag
\end{align}
For the remaining term on the right-hand side of \eqref{splittedlimit}, we set $b_\eps(\eta)$ as the generalized inverse of  $g_\eps(\eta)$, use Fubini's theorem and obtain
\begin{align}\label{eq:c2-3}
&\hspace{-10pt}\lim_{\eps\to0 }\int_\alpha^z\frac{\sigma^2}{\gamma^2}\left(\int_0^{g_\eps(y)}\eta^{-1}U_{zz}(\eta,y)\varphi(\eta)d\eta\right)dy\\
=&\frac{\sigma^2}{\gamma^2}\lim_{\eps\to0 }\int_0^{g_\eps(\alpha)}\left(\int_\alpha^{z\wedge b_\eps(\eta)}U_{zz}(\eta,y)dy\right)\eta^{-1}\varphi(\eta)d\eta\notag\\
=&\frac{\sigma^2}{\gamma^2}\int_0^{g(\alpha)}\left(U_{z}(\eta,z\wedge b(\eta))-U_z(\eta,\alpha+)\right)\eta^{-1}\varphi(\eta)d\eta,\notag
\end{align}
where in the last line we also used $b_\eps\to b$ and $g_\eps\to g$. Combining \eqref{eq:c2-2} and \eqref{eq:c2-3}, and using Fubini's theorem once more we find
\begin{align*}
&\hspace{-10pt}\int_0^\infty v_r(\eta,z)\varphi'(\eta)d\eta\\
=&-\int_0^{g(\alpha)}\frac{2}{\gamma^2}\left(\int_\alpha^{b(\eta)\wedge z}\Big[\rho(\eta)U(\eta,y)-\mu U_z(\eta,y)-k(\theta-\eta)U_r(\eta,y)\Big]dy\right)\eta^{-1}\varphi(\eta)d\eta\notag\\
&+\frac{\sigma^2}{\gamma^2}\int_0^{g(\alpha)}\left(U_{z}(\eta,z\wedge b(\eta))-U_z(\eta,\alpha+)\right)\eta^{-1}\varphi(\eta)d\eta=-\int_0^\infty f(\eta,z)\varphi(\eta)d\eta,\notag
\end{align*}
where, noticing that $\{\eta\le g(\alpha)\}=\{b(\eta)\ge \alpha\}$, we have defined
\begin{align*}
f(\eta,z):=&\mathds{1}_{\{b(\eta)\ge \alpha\}}\frac{2}{\gamma^2}\left(\int_\alpha^{b(\eta)\wedge z}\Big[\rho(\eta)U(\eta,y)-\mu U_z(\eta,y)-k(\theta-\eta)U_r(\eta,y)\Big]dy\right)\eta^{-1}\notag\\
&-\mathds{1}_{\{b(\eta)\ge \alpha\}}\frac{\sigma^2}{\gamma^2}\left(U_{z}(\eta,z\wedge b(\eta))-U_z(\eta,\alpha+)\right)\eta^{-1}.
\end{align*}
It follows that $v_{rr}=f$ in the weak sense.
However, it is not hard to verify that $f\in L^\infty_{loc}(\mathcal{O})$ thanks to Proposition \ref{prop:C1} and Proposition \ref{prop:bC}. Hence $v_{rr}\in L^\infty_{loc}(\mathcal{O})$, as claimed.

Finally, notice that since $r \mapsto b(r)$ is nonincreasing and right-continuous, then it has at most countably many jumps for $r \in (0,\infty)$, hence $f(r,z)= \lim_{\varepsilon \downarrow 0}f(r-\varepsilon,z)$ for a.e.\ $r \in (0,\infty)$ (here the null set depends on $z\ge \alpha$). Let also $(r^J_k)_{k\geq 1}$ be the collection of jump points of the free boundary $b$, and set
$$\mathcal{N}:= \bigcup_{k\geq 1}\big([b(r^J_k),\infty)\times\{r^J_k\}\big).$$
Then $f(r,z)= \lim_{\varepsilon \downarrow 0}f(r-\varepsilon,z)$ for $(r,z) \in \mathcal{O} \setminus \mathcal{N}$. Since $\mathcal{N}$ is a subset of $\mathcal{O}$ with null Lebesgue measure, we conclude that \eqref{eq:vrr} holds true.
\end{proof}

In order to use Theorem \ref{thm:verif} we need to show that $v_{rr}$ is continuous as well in the closure $\overline \cC$ of the continuation set $\cC$, and we accomplish that in the next proposition. We remark that global $C^2$ regularity of a solution to \eqref{HJB1} is far from being a trivial result and, in particular, we are not aware of any probabilistic proof of this fact.

\begin{proposition}
\label{prop:C2}
One has that $v_{rr}$ is continuous in $\overline{\mathcal{C}} \cap \mathcal{O}$.
\end{proposition}
\begin{proof}
It suffices to observe that for any $(r,z) \in \overline{\mathcal{C}} \cap \mathcal{O}$ we have $z \leq b_-(r)$. Hence
\begin{align*}
&v_{rr}(r,z)=\mathds{1}_{\{b_-(r)\ge \alpha\}}\frac{2}{\gamma^2}\left(\int_\alpha^{z}\Big[\rho(r)U(r,y)-\mu U_z(r,y)-k(\theta-r)U_r(r,y)\Big]dy\right)r^{-1}\notag\\
&\hspace{1.5cm}-\mathds{1}_{\{b_-(r)\ge \alpha\}}\frac{\sigma^2}{\gamma^2}\left(U_{z}(r,z)-U_z(r,\alpha+)\right)r^{-1},
\end{align*}
and the claimed continuity follows from Proposition \ref{prop:C1}. Notice that $\mathds{1}_{\{b_-(r)\ge \alpha\}}=1$ for all $r<r_\alpha$, where $r_\alpha:=\sup\{r>0\,:\,b_-(r)>\alpha\}$.
\end{proof}

We conclude this section by proving that indeed $V=v$ and by providing an optimal dividend strategy.
\begin{theorem}
\label{thm:verifico}
Recall $b$ from \eqref{eq:b}, $V$ from \eqref{eq:Vbis} and $v$ from \eqref{eq:v}. Then $V(r,z)=v(r,z)$ for all $(r,z)\in\overline{\mathcal{O}}$ and the process
\begin{align}\label{eq:Oc2}
D^*_t:=\sup_{0\le s\le t}\left[Z^0_s-b(R_s)\right]^+\,,\qquad t\ge0
\end{align}
is an optimal dividend strategy; i.e., for all $(r,z)\in \overline{\mathcal{O}}$ we have
\[
v(r,z)=V(r,z)=\E_{r,z}\left[\int_{0-}^{\tau_\alpha^{D^*}}e^{-\int_0^t\rho(R_t)dt}dD^*_t\right].
\]
\end{theorem}
\begin{proof}
It suffices to check that $v$ of \eqref{eq:v} satisfies all the conditions in Theorem \ref{thm:verif}. The function $v$ is continuous everywhere. Moreover, by Proposition \ref{prop:C1}, $v_z$, $v_{zz}$, $v_r$ and $v_{zr}$ are continuous on $\mathcal{O}$, and, by Proposition \ref{prop:C2}, $v_{rr}$ is continuous in $\overline{\mathcal{C}} \cap \mathcal{O}$.

Since $U\geq 1$ we have that $v_z \geq 1$, with equality for $z \geq b(r)$, $r>0$. Moreover, by \eqref{eq:Uz} we see that $v_{zz}=U_z < 0$ for all $(r,z) \in \mathcal{O}$ such that $\alpha \leq z < b(r)$. Hence $v_z > 1$ for such values of $(z,r)$ because $v_z(r,b(r))=U(r,b(r))=1$. Also, $0 \leq v(r,z) \leq h_0 (z-\alpha)$ for any $(r,z) \in \mathcal{O}$ due to \eqref{eq:U-bound}.

For $r \in \R_+$ and $\alpha < z < b(r)$ we have by Corollary \ref{cor:fbp} and the dominated convergence therorem that
\begin{align*}
 0 =& \int_{\alpha}^{z} \big(\cL - \rho(r)\big)U(r,y)dy \\
 =& \frac{1}{2}\sigma^2 v_{zz}(r,z) + \mu v_z(r,z) - \big(\frac{1}{2}\sigma^2 v_{zz}(r,\alpha+) + \mu v_z(r,\alpha+)\big) \\
&  + \frac{1}{2}\gamma^2 r v_{rr}(r,z) + \kappa(\theta - r)v_r(r,z) - \rho(r)v(r,z) = \big(\cL - \rho(r)\big)v(r,z),
\end{align*}
upon observing that $\frac{1}{2}\sigma^2 v_{zz}(r,\alpha+) + \mu v_z(r,\alpha+)=0$ by Corollary \ref{cor:elb}. Repeating the same calculations for $z > b(r)$, $r >0$, we find that $\big(\cL - \rho(r)\big)v(r,z) \leq 0$. Hence, $\big(\cL - \rho(r)\big)v(r,z) \leq 0$ for a.e.\ $(r,z) \in \mathcal{O}$.

Therefore we have verified all the conditions in \eqref{HJB1}, and it thus follows that $v=V$ and $D^*\equiv D^b$ is optimal.
\end{proof}

%%%%%%%%%%%%%%%%%%%%%%%%%%%%%%%%%%%%%%%%%%%%%%%%%%%%%%%%

\section{Concluding remarks}
\label{sec:finalrem}

\subsection{Some Comments on the Optimal Dividend Policy}
\label{rem:secfinal}

The optimal control from \eqref{eq:Oc2} prescribes to pay dividends in such a way to keep the surplus process below the stochastic threshold $t\mapsto b(R_t)$ at all times. In particular, the company distributes the minimum amount of dividends that prevents the current surplus level from exceeding the current optimal ceiling $b(R_t)$. Any excess of the surplus is paid as a lump sum.
Figure \ref{fig:1} below provides an illustration of the curve $r\mapsto b(r)$, of the process $(Z,R)$, and of the optimal dividend payout {(we refer to Section \ref{sec:num} for the numerical evaluation of the free boundary for some specific choices of the discount rate)}. The optimal dividends distribution is therefore of {\em barrier type} but, differently to classical models with constant discount rate and constant optimal barrier (see, e.g., Section 2.5.2 in Chapter 2 of \cite{Schmidli}), here we observe dynamic (stochastic) adjustments of the barrier. This strategy shows how the firm's manager responds to the fluctuations of the spot rate and allows to draw some economic/financial conclusions in a dynamic (random) macro-economic set-up. In particular, since the free boundary $b$ is a decreasing function, we observe that in scenarios where the interest rate tends to increase, the firm manager will pay dividends more frequently because the expected present value of future dividend payments decays. Of course this behaviour also increases the probability of an early insolvency of the firm since in our model the growth rate of the surplus process is constant and independent of the current spot rate on the market. Despite this general trend, we also observe that no matter how large the spot rate, an immediate liquidation of the firm can never be optimal (final claim in Lemma \ref{lem:bnew}). The combined uncertainty on the future moves of the spot rate and the surplus process indeed encourage gradual liquidation in light of a possible reversion of the spot rate towards lower values and/or upwards excursions of the surplus process.

\begin{figure}
  \includegraphics[width=0.9\linewidth]{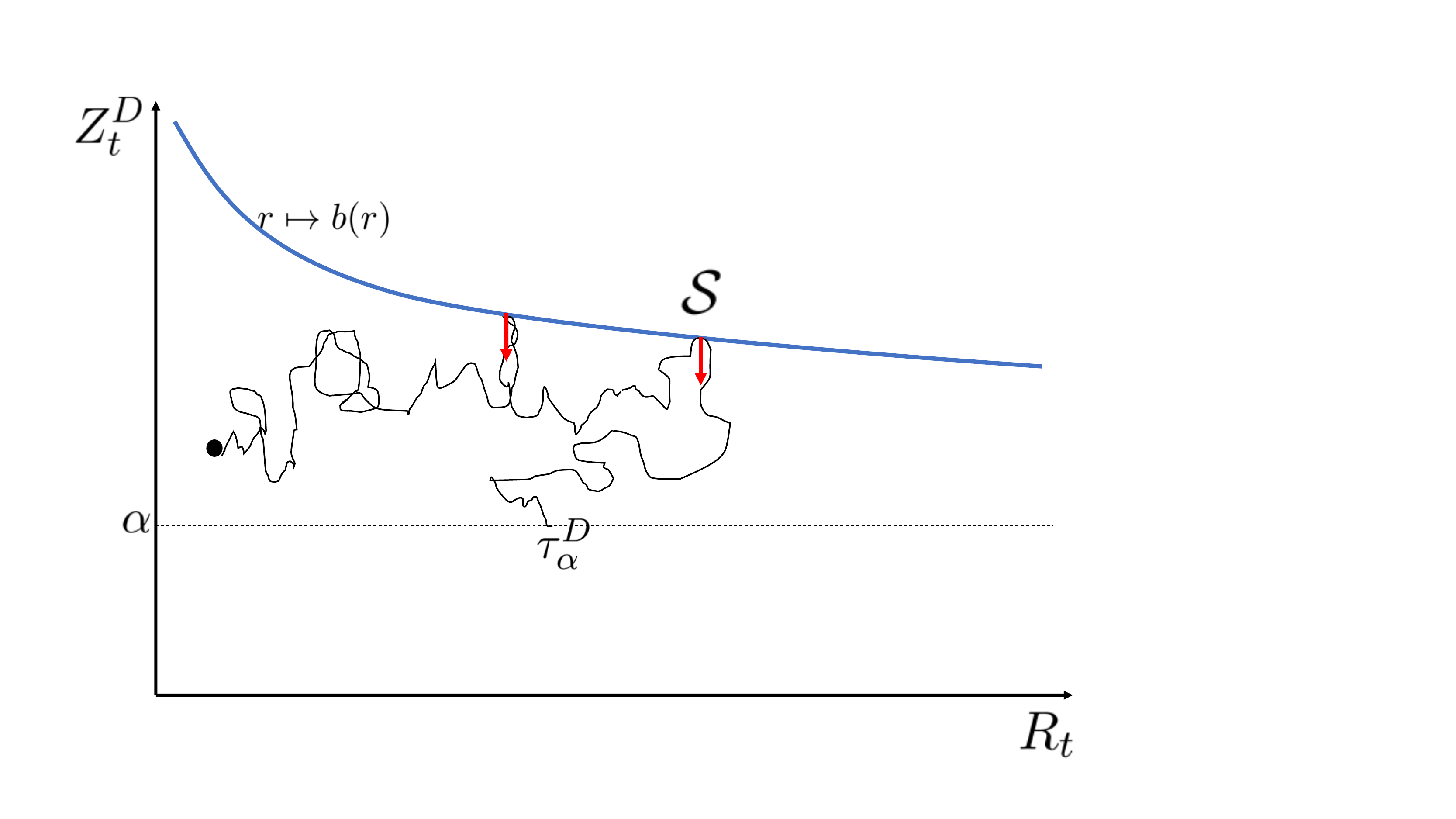}
  \caption{An illustrative drawing of the free boundary $r \mapsto b(r)$ and of the optimal dividend payout. The red arrows illustrate the vertical push that is needed to keep the surplus process below the interest-rate dependent boundary $b$. {\color{black}{In particular, the optimal dividend process defines a continuous measure
$t\mapsto dD^*_t$ on $\R_+$
which is completely singular}}
{with respect to
the Lebesgue measure.}}
  \label{fig:1}
\end{figure}

{\color{black}{
If the discount rate is such that $\rho(r) \geq \rho_0$ for some constant $\rho_0>0$ (e.g., it is of linear form $\rho(r)=\rho_0 + r$), one easily obtains from \eqref{eq:Vbis} that the value function with interest-rate dependent discount force is smaller than the one with $\rho(r) \equiv \rho_0$. However, we also see that if $\rho(r) \geq \rho_0$, then the interest-rate dependent barrier $b$ is uniformly bounded from above by the constant free boundary $z^{\star}_{\rho_0}$ arising in the problem with constant discount rate $\rho_0$ (Proposition \ref{prop:bC}-(ii)). {Continuity of the boundary $r\mapsto b(r)$ implies that} optimal lump sum payments can happen only at the initial time {with $D^*_0=(z-b(r))^+$}. It thus follows that {lump sum payments} are larger than those in the problem with constant discount rate, i.e.\ $(z-z^{\star}_{\rho_0})^+$. Moreover, according to Proposition \ref{prop:bC}-(iii), we see that in the linear case $\rho(r)=\rho_0 + r$ the size of the lump sum payments increases with the value of the interest rate (and indeed it attains its maximum when $r\uparrow \infty$, {with $D^*_0=(z-\alpha)^+$}). {This is in contrast with the case of constant interest rate $\rho(r)=\rho_0$, where $D^*_0=(z-z^\star_{\rho_0})^+$}.}}

%%%%%%%%%%%%

\subsection{{\color{black}{Numerical Illustrations}}}
\label{sec:num}

{\color{black}{{In this section we outline a simple numerical method that allows to compute the free boundary $b$ via the PDE associated to the value function $U$ of the optimal stopping problem \eqref{eq:U2}. A direct study of the PDE for its value function $V$ (cf.\ \eqref{eq:Vbis}) is possible in principle but more involved because the gradient constraint $V_z\ge 1$ is harder to implement than the obstacle constraint $U\ge 1$. While the study of an optimised numerical scheme is outside the scope of our paper, the results in this section show that the connection to optimal stopping also provides useful tools for numerical solution of the original singular control problem. }

{We consider the two cases when} $\rho(r)= r_0 + r$ and $\rho(r) = \sqrt{r_0 + r}$, for $r_0=0.05$ ({notice that Assumption \ref{ass:rho} is satisfied).} The parameters' values are:
\[
\alpha=0,\quad \sigma = 1,\quad \mu = 1,\quad \theta = 0.15,\quad \kappa = 0.5,\quad \gamma=0.3,
\]
{with respect to a time unit of one year (these are for illustrative purpose only and we leave the question of calibration with real market data for future work).}

The free boundary {is} determined as {the boundary of the level} set at $1$ of the {function} $U$. {The function $U$} {is approximated numerically} by {the solution} {of a penalised PDE problem over the truncated domain $\mathcal{O}_{\textrm{Num}}:=(r_{\text{min}},r_{\text{max}}) \times (0,z_{\text{max}})$, where $r_{\text{min}}=0.005$, $r_{\text{max}}=1.1$ and $z_{\text{max}}=2.5$ are chosen arbitrarily. {\color{black}{In our experiments some care is needed for the choice of $r_{\text{min}}$ since $\{0\}$ is non-attainable for the spot rate $R$.}}

Given $\delta=0.01$ we use the software Mathematica's command \texttt{NDsolve} to solve the following penalised problem}:
\begin{align}
&\label{eq:penal}\cL U(r,z)-\rho(r)U(r,z)=\tfrac{1}{\delta}\big(1 - U(r,z)\big)^+,\qquad (r,z)\in \mathcal{O}_{\textrm{Num}}
\end{align}
{with Neumann boundary condition (cf.\ Corollary \ref{cor:elb})
\[
U_z(r,0+)=-\lambda U(r,0+),\qquad r\in(r_{\text{min}}, r_{\text{max}}),
\]
and Dirichlet conditions
\begin{align}\label{Diric}
U(r,z_{\text{max}}) =1 = U(r_{\text{max}},z)\qquad\text{and}\qquad U(0+,z) = \Big(1 - \tfrac{1}{1 + z}\Big) u(z).
\end{align}}
Here, {$u(z)=V^0_z(x)$} is the derivative of the value function {$V^0$} of the optimal dividend problem with constant interest rate $\rho(0)>0$ ({recall that in our case $\rho(0)$ equals either $r_0$ or $\sqrt{r_0}$}), {whose explicit formula can be found in (cf.\ eq.\ (3.3) in \cite{LOKKA})}.

{This system of equations can be justified as follows:
\begin{itemize}
\item[(i)] The penalisation procedure is standard when solving variational inequalities arising in optimal stopping (see, e.g., \cite{BL}). One can show that as $\delta\downarrow 0$ the solution of \eqref{eq:penal} converges to the true value function $U$ uniformly on compacts (provided of course that $U$ is sufficiently regular, as in our case). The advantage of solving \eqref{eq:penal} numerically instead of the free boundary problem in Corollary \ref{cor:fbp} is that the domain in \eqref{eq:penal} does not need to be determined as part of the solution.
%term with coefficient $1/\delta$ appearing on the right-hand side of \eqref{eq:penal} aims at taking care of the constraint $U\geq 1$, while the left-hand term corresponds to the linear PDE solved by $U$ in the continuation region (cf.\ \eqref{eq:fbp1}). Overall, \eqref{eq:penal} can be interpreted as the penalised version of the variational inequality that $U$ is expected to solve.
\item[(ii)] The first condition in \eqref{Diric}, i.e., $U(r_{\text{max}},z)=U(r,z_{\text{max}})=1$, is justified by noticing that the optimal boundary $b(r)$ is bounded (cf.\ Proposition \ref{prop:bC}-(ii)) and converges to $\alpha=0$ as $r \uparrow \infty$ (cf.\ Proposition \ref{prop:bC}-(iii)). So, for large values of $r$ and/or $z$ we expect to be in the stopping region.
\item[(iii)] The second condition in \eqref{Diric} is the most delicate, since $\{0\}$ is not attainable by $R$ and so in theory there is no need for a boundary condition. Numerically, however, such a condition is needed. Here we use that $U(r,z) \leq u(z)$ for any $(r,z)$, and that, theoretically we expect $U(0,z)\approx u(z)$ for large values of $z$.
\end{itemize}
}

Drawings of the optimal stopping boundaries are presented in Figure \ref{fig:2}. The {boundary of the black region is the one obtained for $\rho(r)=r_0 + r$, whereas the boundary of the grey area is the one obtained for $\rho(r)=\sqrt{r_0 + r}$}. {For completeness we also plot the value function $U$ of the optimal stopping problems}.

\begin{figure}
  \includegraphics[width=0.6\linewidth]{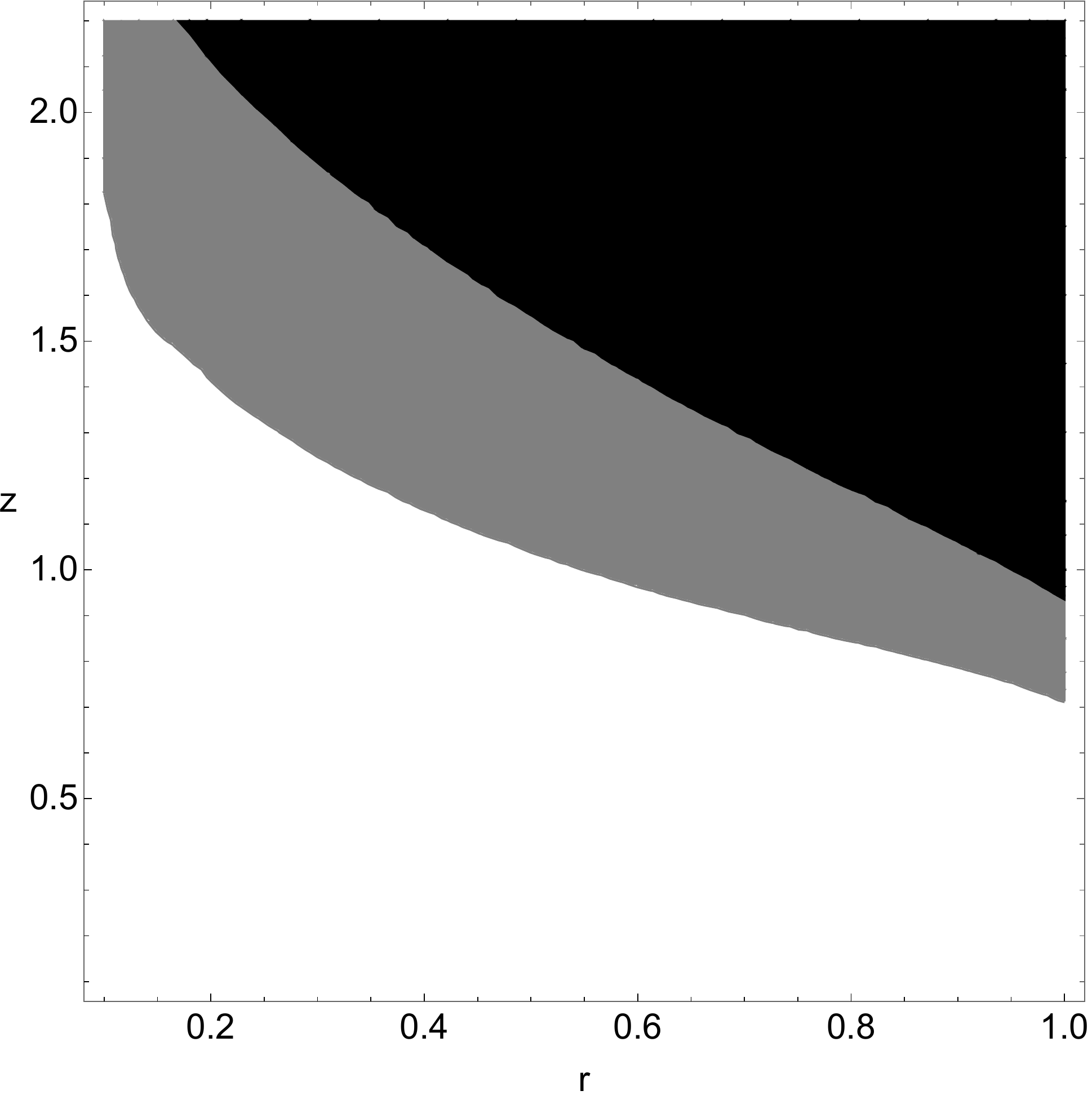}
  \caption{{Plots of the optimal stopping boundary in the case $\rho(r)=r_0 + r$ (boundary of the black area) and $\rho(r)=\sqrt{r_0 + r}$ (boundary of the grey area). Continuation regions lie below the boundaries, while} {stopping regions above the boundaries.}}
  \label{fig:2}
\end{figure}

\begin{figure}
  \includegraphics[width=0.7\linewidth]{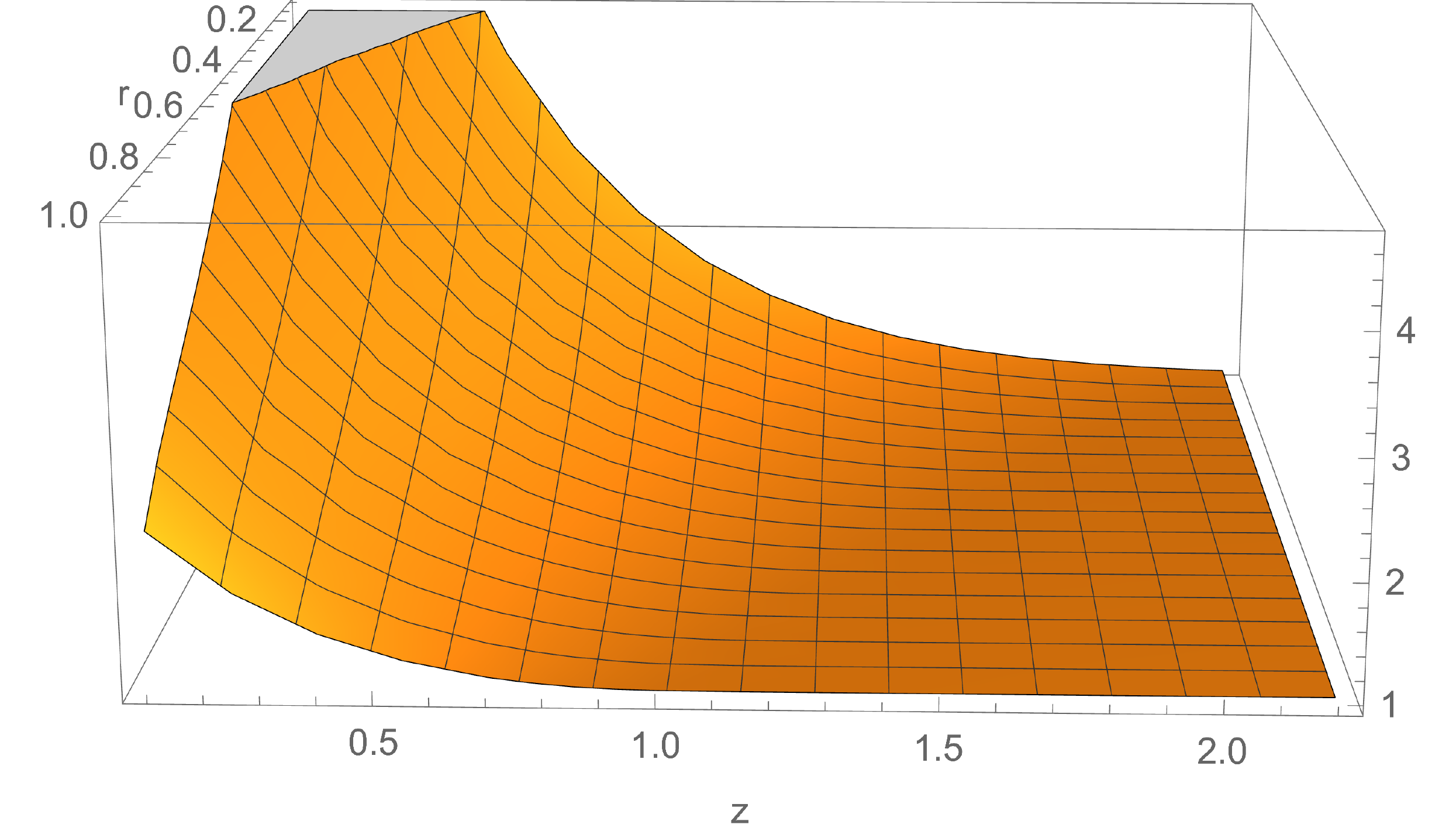}
  \caption{{Plot of $U$ in the case $\rho(r)=r_0 + r$.}}
  \label{fig:3}
\end{figure}

\begin{figure}
  \includegraphics[width=0.7\linewidth]{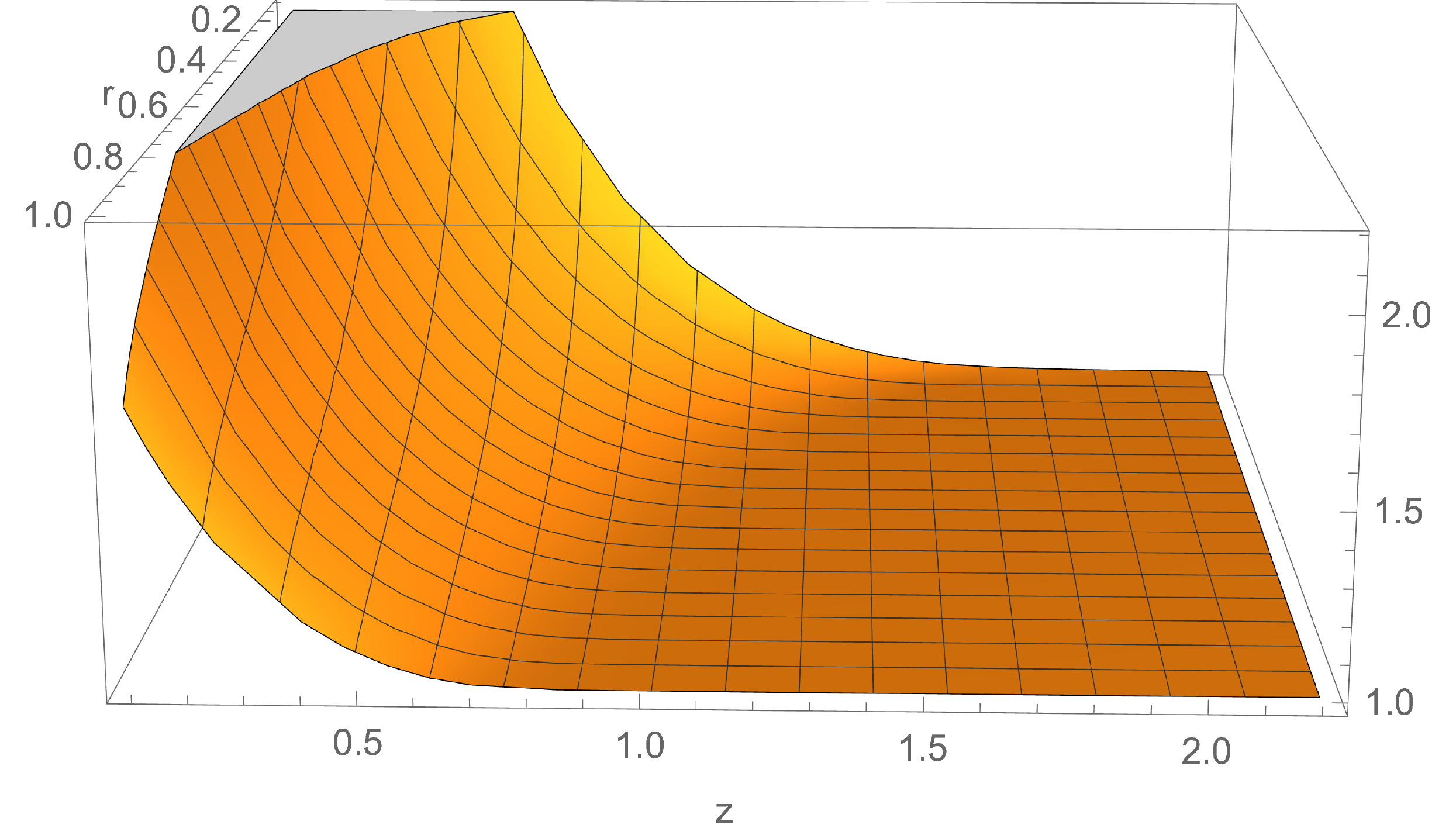}
  \caption{{Plot of $U$ in the case $\rho(r)=\sqrt{r_0 + r}$.}}
  \label{fig:4}
\end{figure}

In Figure \ref{fig:2} we observe that the optimal boundary related to $\rho(r)=\sqrt{r_0 + r}$ is smaller than that related to $\rho(r)=r_0 + r$. {Intuitively,} because of its mean-reverting behaviour, the interest rate process $R$ oscillates around $\theta=0.15$ for all times with large probability. So $\sqrt{r_0 + R_t} \approx 0.2$ (being $r_0=0.05$), with fluctuations of order $\sqrt{R_t}\approx 0.45$. As a consequence, $\sqrt{r_0 + R_t} \gtrsim r_0 + R_t$, which implies that the value function $U$ {with linear discount rate} is larger than {the one discounted with $\rho(r)=\sqrt{r_0+r}$} (see Figures \ref{fig:3} and \ref{fig:4}). This fact in turn yields the ordering between the free boundaries observed in Figure \ref{fig:2}.

From Figure \ref{fig:2} we also notice that the optimal boundary {obtained for} $\rho(r)=\sqrt{r_0 + r}$ {seems more convex than its counterpart in the case of linear discount rate}, in a right neighborhood of $r=0$. {This is due to the fact that when $r=r_{\text{min}}$ the two value functions in Figure \ref{fig:3} and \ref{fig:4} take the same value (cf.\ \eqref{Diric}) but a small increment in $r$ affects the discount rate $\rho(r)=\sqrt{r_0+r}$ more than in the linear case, hence causing a faster drop in the corresponding value function.}}} {\color{black}{We finally notice that employing eq.\ (3.6) in \cite{LOKKA}, among others, the free boundaries associated to the optimal dividend problems with constant discount rate $\rho(r)=r_0$ and $\rho(r)=\sqrt{r_0}$ can be explicitly evaluated. In particular, for our parameter choice they assume values $3.56$ and $1.98$, respectively.}}

%%%%%%%

\subsection{\color{black}On the Case of Correlated Brownian Motions}
\label{sec:corrBM}
{\color{black}{
Throughout this paper we have assumed that $W$ and $B$ are independent. {Here we provide the heuristic connection between the dividend problem and an optimal stopping problem when $W$ and $B$ are correlated (see also \cite{DeA}). The connection used in Section \ref{sec:OS} will then follow as a special case. We do remark however that the stopping problem obtained for correlated Brownian motions is structurally more involved than the one we solved in this paper. A complete study requires different tools and it is left for future work (more details are presented at the end). }

{Recall the dynamics for $(R,Z^D)$ given by \eqref{eq:R} and \eqref{eq:Z} and assume}
$\E[B_t W_t]= \beta t$, for some $\beta \in (-1,1)$. The infinitesimal generator $\cL$ of the pair $(R, Z^0)$ is then defined by its action on twice-continuously differentiable functions $f$ as
\begin{align}\label{eq:Lcorr}
\cL f :=\frac{1}{2} \,\sigma^2\, f_{zz} + \beta \sigma \gamma \sqrt{r} \,f_{rz}+ \frac{1}{2} \,\gamma^2\,r\, f_{rr} + \mu\, f_z + k(\theta- r)\, f_r,
\end{align}
and the HJB equation for the dividend problem reads as in \eqref{HJB1}, but with $\cL$ given now by \eqref{eq:Lcorr}.

Letting $z \downarrow \alpha$ in the second equation of \eqref{HJB1}, {assuming} that $[0,\,\infty) \times \{\alpha\}$ belongs to the inaction set, we get
\begin{align}\label{eq:elas}
\tfrac{1}{2} \,\sigma^2\, v_{zz}(r, \alpha +) + \mu\, v_z(r, \alpha +) + \beta \sigma \gamma \sqrt{r} \,v_{rz}(r, \alpha +)=0,
\end{align}
{using the fact that $v(r,\alpha)=0$ should imply $v_r(r,\alpha)=v_{rr}(r,\alpha)=0$ for sufficiently smooth $v$.}
{Setting $u:= v_z$ and differentiating the second equation in \eqref{HJB1}}, we find
\begin{align} \label{HJBcorr}
	\left\{
	\begin{array}{ll}
	\cL u(r,z) - \rho(r)\, u(r,z)=0,\quad &\textup{on}\,\,\{u >1\}\\[+3pt]
	u(r,z) \geq 1,\quad & {\text{a.e.}\:(r,z) \in \mathcal{O}}
	\\[+3pt]
 u_{z}(r, \alpha +) + \frac{2\beta  \gamma}{\sigma} \sqrt{r} \,u_{r}(r, \alpha +) + \frac{2\mu}{\sigma^2}\, u(r, \alpha +) =0, \quad  &r \geq 0,
	\end{array}
	\right.
	\end{align}
{where the final equation is \eqref{eq:elas}. A further condition of the form
\begin{align}\label{eq:Lineq}
\cL u(r,z) - \rho(r)\, u(r,z)\leq 0,\quad  {\text{a.e.}\:(r,z) \in \mathcal{O}}
\end{align}
should appear in variational problems related to optimal stopping. While this cannot be derived directly from \eqref{HJB1}, we may equally} expect that the variational problem {for $u$ be} related to the optimal stopping problem
\begin{align}
\label{eq:Ucorr}
\widehat{U}(r,z) := \sup_{\tau \geq 0} \E_{r,z}\left[ e^{\frac{2 \mu}{\sigma^2}\,\ell^\alpha_\tau - \int_0^\tau \rho(\widehat{R}_s) \,ds}\right],\quad (r,z)\in\overline{\mathcal{O}},
\end{align}
{where (recall \eqref{K}),
\[
K^z_t=z-Y_t+\ell^\alpha_t,\quad Y_t=-\mu t+\sigma B_t\quad\text{and}\quad \ell^\alpha_t:=(z-\alpha)\vee S_t-(z-\alpha),
\]
so that $(K_t)_{t\ge 0}$ is a Brownian motion with drift $\mu$ and diffusion $\sigma$, starting at $z\geq\alpha$ and reflected at $\alpha$ (see \cite{P06});}
{instead, the dynamics of the process $\widehat{R}$ reads}
\begin{align*}
d \widehat{R}_t = k(\theta - \widehat{R}_t) dt + \gamma \sqrt{\widehat{R}_t}\,dW_t + \frac{2 \beta \gamma}{\sigma}\sqrt{\widehat{R}_t}\,d\ell^\alpha_t, \quad \widehat{R}_{0}=r,
\end{align*}
{where $W$ is also a Brownian motion and $\E[B_t W_t]=\beta t$ as before.}

In order to {clarify why we expect $u=\widehat U$},
%see the aforementioned connection between the PDE \eqref{HJBcorr} and the optimal stopping \eqref{eq:Ucorr}, let
assume $u \in C^2(\bar {\mathcal O})$ be a solution to \eqref{HJBcorr} {with the additional condition \eqref{eq:Lineq}}.
%such that
%$$
%t \mapsto e^{\frac{2 \mu}{\sigma^2}\,\ell^\alpha_t - \int_0^t \rho(\widehat{R}_s) \,ds} u(\widehat{R}_t,  Z_t)
%$$
%is a supermartingale under the measure $\P_{r,z}(\,\cdot\,):=\P_{r,z}(\,\cdot\,|\widehat{R}_0=r, Z_0=z)$.
Applying Dynkin's formula to
\[
e^{\frac{2 \mu}{\sigma^2}\,\ell^\alpha_t - \int_0^t \rho(\widehat{R}_s) \,ds} u(\widehat{R}_t, Z_t)
\]
on the {random} interval $[0,\tau]$,
%by \eqref{HJBcorr} we get under $\P_{r,z}$
{we get
\begin{align*}
	u(r,z) =&\,\E_{r,z}\left[ e^{\frac{2 \mu}{\sigma^2}\,\ell^\alpha_\tau - \int_0^\tau \rho(\widehat{R}_s) \,ds}u(\widehat{R}_\tau, Z_\tau)-\int_0^\tau e^{\frac{2 \mu}{\sigma^2}\,\ell^\alpha_t - \int_0^t \rho(\widehat{R}_s) \,ds}\big(\cL u-\rho u\big)(\widehat R_t,K_t)dt\right]\\
&+\E_{r,z}\left[\int_0^\tau e^{\frac{2 \mu}{\sigma^2}\,\ell^\alpha_t - \int_0^t \rho(\widehat{R}_s) \,ds}\Gamma(\widehat R_t,K_t)d\ell^\alpha_t\right],		
\end{align*}
where $\Gamma(r,z):=u_{z}(r, z) + \frac{2\beta  \gamma}{\sigma} \sqrt{r} \,u_{r}(r, z) + \frac{2\mu}{\sigma^2}\, u(r, z)$. Then, using that $\cL u-\rho u\le 0$ and that $\Gamma(\widehat R_t,K_t)d\ell^\alpha_t=\Gamma(\widehat R_t,\alpha)d\ell^\alpha_t=0$ we obtain}
\begin{align*}
u(r,z)
&\geq \E_{r,z}\left[ e^{\frac{2 \mu}{\sigma^2}\,\ell^\alpha_\tau - \int_0^\tau \rho(\widehat{R}_s) \,ds}u(\widehat{R}_t, Z_t)\right]\\
&\geq \E_{r,z}\left[ e^{\frac{2 \mu}{\sigma^2}\,\ell^\alpha_\tau - \int_0^\tau \rho(\widehat{R}_s) \,ds}\right]
\end{align*}
for any  stopping time $\tau$. Therefore, $u \geq \widehat{U}$. Finally, by the second and fourth formula in \eqref{HJBcorr}, the above inequalities become equalities if we choose
\[
\tau=\inf\{t\ge 0: u(\widehat R_t,K_t)=1\}
\]
{and provided that $\P_{r,z}(\tau<\infty)=1$ and suitable transversality conditions hold. Thus $\widehat U=u$.}

{If $\beta=0$ we fall back into our original setting from Section \ref{sec:OS}, where $\widehat R=R$ and $R$ is independent of $K$ (then also $\widehat U=U$ as in \eqref{eq:U2}). Such independence of the two processes is useful to establish integrability and monotonicity properties of the value function $U$, which instead are no longer guaranteed when $\beta\neq 0$ (the main difficulty is due to $\ell^\alpha$ appearing also in the dynamics of of the discount rate). Therefore, a study of the problem in full generality requires different methods to the one we use in this paper and it is left for future work.}
}}

%%%%%%%%%%

%%%%%%%%%%%%%%%%%%%%%%%%%%%%%%%%%%%%%%%%%%%%%%%%%%%%%%%%%%%%%%%%%%%%%%%%%%%%%%%%%%%%%%%%%%%%%%%%%%%%%%%%%%%%%%%%%%%%%%%%%%%%%%%%%%%%%%%%%%%%%%%

\medskip

\indent \textbf{Acknowledgments.}
 {\color{black}{We wish to thank two anonymous referees for their pertinent and useful comments on a first version of this work.}} Tiziano De Angelis gratefully acknowledges partial supported by EPSRC grant EP/R021201/1. Financial support by the German Research Foundation (DFG) through the Collaborative Research Centre 1283 is gratefully acknowledged by Giorgio Ferrari. Elena Bandini gratefully acknowledges the partial support of INdAM - GNAMPA Project 2018. Elena Bandini and Fausto Gozzi gratefully acknowledge the partial support of PRIN 2015/2016 ``Deterministic and Stochastic Evolution Equations''. This work was initiated during a visit of T.\ De Angelis and G.\ Ferrari at LUISS Guido Carli. Both authors are grateful for the hospitality and the financial support offered by LUISS Guido Carli.

%%%%%%%%%%%%%%%%%%%%%%%%%%%%%%%%%%%%%%%%%%%%%%%%%%%%%%%%%%%%%%%%%%%%%%%%%%%%%%%%%%%%%%%%%%%%%%

\end{document}